\documentclass[12pt]{amsart}
\usepackage{amssymb}
\usepackage{amsfonts}
\usepackage{hyperref}
\usepackage{graphicx}
\usepackage{pgf,tikz}
\usepackage{hyperref}

\usetikzlibrary{shapes.geometric}
\usetikzlibrary{trees}
\hypersetup{
    colorlinks=true,           linkcolor=blue,              citecolor=magenta,            filecolor=magenta,          urlcolor=cyan,               breaklinks=true,  	}  
\allowdisplaybreaks 
\theoremstyle{plain}
\newtheorem{theorem}[equation]{Theorem}
\newtheorem{2WtMax}[equation]{Theorem on Maximal Function  Inequalities}

\newtheorem{2WtSIStrongType}[equation]{Theorem on Maximal Singular Integral Strong-Type Inequalities}
\newtheorem{2WtSIWeakType}[equation]{Theorem on Maximal Singular Integral Weak-Type Inequalities}

\newtheorem{claim}{Claim}

\newtheorem{conjecture}[equation]{Conjecture}

\newtheorem{lemma}[equation]{Lemma}
\newtheorem{notat
ion}{Notation}
\newtheorem{problem}[equation]{Problem}

\numberwithin{equation}{section}
\numberwithin{figure}{section}
\theoremstyle{definition} 
\newtheorem{definition}[equation]{Definition}
\theoremstyle{remark}
\newtheorem{remark}[equation]{Remark}
\newtheorem*{acknowledgement}{Acknowledgment}
\input{tcilatex}

\begin{document}
\title[Two weight norm inequalities for maximal singular integrals]{A
characterization of two weight norm inequalities for maximal singular
integrals with one doubling measure}
\author[M.T. Lacey]{Michael T. Lacey}
\address{School of Mathematics \\
Georgia Institute of Technology \\
Atlanta GA 30332}
\email{lacey@math.gatech.edu}
\thanks{Research supported in part by the NSF, through grant DMS-0456538.}
\author[E.T. Sawyer]{Eric T. Sawyer}
\address{ Department of Mathematics \& Statistics, McMaster University, 1280
Main Street West, Hamilton, Ontario, Canada L8S 4K1 }
\email{sawyer@mcmaster.ca}
\thanks{Research supported in part by NSERC}
\author[I. Uriarte-Tuero]{Ignacio Uriarte-Tuero}
\address{ Department of Mathematics \\
Michigan State University \\
East Lansing MI }
\email{ignacio@math.msu.edu}
\thanks{Research supported in part by the NSF, through grant DMS-0901524.}
\date{}

\begin{abstract}
Let $\sigma $ and $\omega $ be positive Borel measures on $\mathbb{R}$ with $%
\sigma $ doubling. Suppose first that $1<p\leq 2$. We characterize
boundedness of certain maximal truncations of the Hilbert transform $%
T_{\natural }$ from $L^{p}\left( \sigma \right) $ to $L^{p}\left( \omega
\right) $ in terms of the strengthened $A_{p}$ condition%
\begin{equation*}
\left( \int_{\mathbb{R}}s_{Q}\left( x\right) ^{p}d\omega \left( x\right)
\right) ^{\frac{1}{p}}\left( \int_{\mathbb{R}}s_{Q}\left( x\right)
^{p^{\prime }}d\sigma \left( x\right) \right) ^{\frac{1}{p^{\prime }}}\leq
C\left\vert Q\right\vert ,
\end{equation*}%
$s_{Q}\left( x\right) =\frac{\left\vert Q\right\vert }{\left\vert
Q\right\vert +\left\vert x-x_{Q}\right\vert }$, and two testing conditions.
The first applies to a restricted class of functions and is a strong-type
testing condition, 
\begin{equation*}
\int_{Q}T_{\natural }\left( \chi _{E}\sigma \right) (x)^{p}d\omega (x)\leq
C_{1}\int_{Q}d\sigma (x),\qquad \text{for all }E\subset Q\text{,}
\end{equation*}%
and the second is a weak-type or dual interval testing condition, 
\begin{equation*}
\int_{Q}T_{\natural }\left( \chi _{Q}f\sigma \right) (x)d\omega (x)\leq
C_{2}\left( \int_{Q}\left\vert f(x)\right\vert ^{p}d\sigma (x)\right) ^{%
\frac{1}{p}}\left( \int_{Q}d\omega (x)\right) ^{\frac{1}{p^{\prime }}},
\end{equation*}%
for all intervals $Q$ in $\mathbb{R}$ and all functions $f\in L^{p}\left(
\sigma \right) $. In the case $p>2$ the same result holds if we include an
additional necessary condition, the Poisson condition%
\begin{equation*}
\int_{\mathbb{R}}\left( \sum_{r=1}^{\infty }\left\vert I_{r}\right\vert
_{\sigma }\left\vert I_{r}\right\vert ^{p^{\prime }-1}\sum_{\ell =0}^{\infty
}\frac{2^{-\ell }}{\left\vert \left( I_{r}\right) ^{\left( \ell \right)
}\right\vert }\chi _{\left( I_{r}\right) ^{\left( \ell \right) }}\left(
y\right) \right) ^{p}d\omega \left( y\right) \leq C\sum_{r=1}^{\infty
}\left\vert I_{r}\right\vert _{\sigma }\left\vert I_{r}\right\vert
^{p^{\prime }},
\end{equation*}%
for all pairwise disjoint decompositions $Q=\cup _{r=1}^{\infty }I_{r}$ of
the dyadic interval $Q$ into dyadic intervals $I_{r}$. We prove that
analogues of these conditions are sufficient for boundedness of certain
maximal singular integrals in $\mathbb{R}^{n}$ when $\sigma $ is doubling
and $1<p<\infty $. Finally, we characterize the weak-type two weight
inequality for certain maximal singular integrals $T_{\natural }$ in $%
\mathbb{R}^{n}$\ when $1<p<\infty $, without the doubling assumption on $%
\sigma $, in terms of analogues of the second testing condition and the $%
A_{p}$ condition.
\end{abstract}

\maketitle

\section{Introduction}

Two weight inequalities for Maximal Functions and other positive operators
have been characterized in \cite{Saw}, \cite{Saw1}, \cite{Saw2}, with these
characterizations being given in terms of obviously necessary conditions,
that the operators be uniformly bounded on a restricted class of functions,
namely indicators of intervals and cubes. Thus, these characterizations have
a form reminiscent of the $T1$ Theorem of David and Journ\'{e}.

Corresponding results for even the Hilbert transform have only recently been
obtained (\cite{NTV3} and \cite{LaSaUr}) and even then only for $p=2$;
evidently these are much harder to obtain. We comment in more detail on
prior results below, including the innovative work of Nazarov, Treil and
Volberg \cite{NTV1}, \cite{NTV2}, \cite{NTV3}, \cite{NTV4}.

Our focus is on providing characterizations of the boundedness of certain
maximal truncations of a fixed operator of singular integral type. The
singular integrals will be of the usual type, for example the Hilbert
transform or paraproducts. Only size and smoothness conditions on the kernel
are assumed, see \eqref{sizeandsmoothness}. The characterizations are in
terms of certain obviously necessary conditions, in which the class of
functions being tested is simplified. For such examples, we prove
unconditional characterizations of both strong-type and weak-type two weight
inequalities for certain maximal truncations of the Hilbert transform, but
with the additional assumption that $\sigma $ is \emph{doubling} for the
strong type inequality. A major point of our characterizations is that they
hold for \emph{all} $1<p<\infty $. The methods in \cite{LaSaUr}, \cite{NTV1}%
, \cite{NTV2}, \cite{NTV3}, \cite{NTV4} apply only to the case $p=2$, where
the orthogonality of measure-adapted Haar bases prove critical. The doubling
hypothesis on $\sigma $ may not be needed in our theorems, but is required
by the use of Calder\'{o}n-Zygmund decompositions in our method.

As the precise statements of our general results are somewhat complicated,
we illustrate them with an important case here. Let 
\begin{equation*}
Tf\left( x\right) =\lim_{\varepsilon \rightarrow 0}\int_{\mathbb{R}\setminus
\left( -\varepsilon ,\varepsilon \right) }\frac{1}{y}f\left( x-y\right) dy
\end{equation*}%
denote the Hilbert transform, let%
\begin{equation*}
T_{\flat }f\left( x\right) =\sup_{0<\varepsilon <\infty }\left\vert \int_{%
\mathbb{R}\setminus \left( -\varepsilon ,\varepsilon \right) }\frac{1}{y}%
f\left( x-y\right) dy\right\vert
\end{equation*}%
denote the usual maximal singular integral associated with $T$, and finally
let%
\begin{equation*}
T_{\natural }f\left( x\right) =\sup_{0<\varepsilon _{1},\varepsilon
_{2}<\infty :\frac{1}{4}<\frac{\varepsilon _{2}}{\varepsilon _{1}}%
<4}\left\vert \int_{\mathbb{R}\setminus \left( -\varepsilon _{1},\varepsilon
_{2}\right) }\frac{1}{y}f\left( x-y\right) dy\right\vert
\end{equation*}%
denote the new \emph{strongly} (or \emph{noncentered}) maximal singular
integral associated with $T$ that is defined more precisely below. Suppose $%
\sigma $ and $\omega $ are two locally finite positive Borel measures on $%
\mathbb{R}$ that have no point masses in common. Then we have the following
weak and strong type characterizations which we emphasize hold for \emph{all}
$1<p<\infty $.

\begin{itemize}
\item The operator $T_{\flat }$ is \emph{weak} type $\left( p,p\right) $
with respect to $\left( \sigma ,\omega \right) $, i.e.%
\begin{equation}
\left\Vert T_{\flat }\left( f\sigma \right) \right\Vert _{L^{p,\infty
}\left( \omega \right) }\leq C\left\Vert f\right\Vert _{L^{p}\left( \sigma
\right) },  \label{e.imposition}
\end{equation}%
for all $f$ bounded with compact support, \emph{if and only if }the two
weight $A_{p}$ condition 
\begin{equation*}
\frac{1}{\left\vert Q\right\vert }\int_{Q}d\omega \left( \frac{1}{\left\vert
Q\right\vert }\int_{Q}d\sigma \right) ^{p-1}\leq C,
\end{equation*}%
holds for all intervals $Q$; and the dual $T_{\flat }$ interval testing
condition%
\begin{equation}
\int_{Q}T_{\flat }\left( \chi _{Q}f\sigma \right) d\omega \leq C\left(
\int_{Q}\left\vert f\right\vert ^{p}d\sigma \right) ^{\frac{1}{p}}\left(
\int_{Q}d\omega \right) ^{\frac{1}{p^{\prime }}},  \label{dtc}
\end{equation}%
holds for all intervals $Q$ and $f\in L_{Q}^{p}\left( \sigma \right) $ (part 
$4$ of Theorem \ref{weaktwoweightHaar}). The same is true for $T_{\natural }$%
. It is easy to see that \eqref{dtc} is equivalent to the more familiar dual
interval testing condition%
\begin{equation}
\int_{Q}\left\vert L^{\ast }\left( \chi _{Q}\omega \right) \right\vert
^{p^{\prime }}d\sigma \leq C\int_{Q}d\omega ,
\end{equation}%
for all intervals $Q$ and linearizations $L$ of the maximal singular
integral $T_{\flat }$ (see \eqref{unif}).

\item Suppose in addition that $\sigma $ is doubling and $1<p<\infty $. Then
the operator $T_{\natural }$ is strong type $\left( p,p\right) $ with
respect to $\left( \sigma ,\omega \right) $, i.e. 
\begin{equation*}
\left\Vert T_{\natural }\left( f\sigma \right) \right\Vert _{L^{p}\left(
\omega \right) }\leq C\left\Vert f\right\Vert _{L^{p}\left( \sigma \right) },
\end{equation*}%
for all $f$ bounded with compact support, \emph{if and only if} these four
conditions hold. (\textbf{1}) the strengthened $A_{p}$ condition%
\begin{equation*}
\left( \int_{Q}s_{Q}\left( x\right) ^{p}d\omega \left( x\right) \right) ^{%
\frac{1}{p}}\left( \int_{I}s_{Q}\left( x\right) ^{p^{\prime }}d\sigma \left(
x\right) \right) ^{\frac{1}{p^{\prime }}}\leq C\left\vert Q\right\vert ,
\end{equation*}%
$s_{Q}\left( x\right) =\frac{\left\vert Q\right\vert }{\left\vert
Q\right\vert +\left\vert x-x_{Q}\right\vert }$, holds for all intervals $Q$;
(\textbf{2}) the dual $T_{\natural }$ interval testing condition%
\begin{equation*}
\int_{Q}T_{\natural }\left( \chi _{Q}f\sigma \right) d\omega \leq C\left(
\int_{Q}\left\vert f\right\vert ^{p}d\sigma \right) ^{\frac{1}{p}}\left(
\int_{Q}d\omega \right) ^{\frac{1}{p^{\prime }}},
\end{equation*}%
holds for all intervals $Q$ and $f\in L_{Q}^{p}\left( \sigma \right) $; (%
\textbf{3}) the forward $T_{\natural }$ testing condition%
\begin{equation}
\int_{Q}T_{\natural }\left( \chi _{E}\sigma \right) ^{p}d\omega \leq
C\int_{Q}d\sigma ,  \label{ftc}
\end{equation}%
holds for all intervals $Q$ and all compact subsets $E$ of $Q$; and (\textbf{%
4}) the Poisson condition%
\begin{equation*}
\int_{\mathbb{R}}\left( \sum_{r=1}^{\infty }\left\vert I_{r}\right\vert
_{\sigma }\left\vert I_{r}\right\vert ^{p^{\prime }-1}\sum_{\ell =0}^{\infty
}\frac{2^{-\ell }}{\left\vert \left( I_{r}\right) ^{\left( \ell \right)
}\right\vert }\chi _{\left( I_{r}\right) ^{\left( \ell \right) }}\left(
y\right) \right) ^{p}d\omega \left( y\right) \leq C\sum_{r=1}^{\infty
}\left\vert I_{r}\right\vert _{\sigma }\left\vert I_{r}\right\vert
^{p^{\prime }},
\end{equation*}%
for all pairwise disjoint decompositions $Q=\cup _{r=1}^{\infty }I_{r}$ of
the dyadic interval $Q$ into dyadic intervals $I_{r}$, for any fixed dyadic
grid. In the case $1<p\leq 2$, only the first three conditions are needed
(Theorem \ref{improved}). Note that in \eqref{ftc} we are required to test
over all compact subsets $E$ of $Q$ on the left side, but retain the upper
bound over the (larger) cube $Q$ on the right side.
\end{itemize}

As these results indicate, the imposition of the weight $\sigma $ on both
sides of \eqref{e.imposition} is a standard part of weighted theory, in
general necessary for the testing conditions to be sufficient. Compare to
the characterization of the two weight maximal function inequalities in
Theorem~\ref{sawyerthm} below.

\begin{problem}
In \eqref{ftc}, our testing condition is more complicated than one would
like, in that one must test over all compact $E\subset Q$ in \eqref{ftc}.
There is a corresponding feature of \eqref{dtc}, seen after one unwinds the
definition of the linearization $L ^{\ast} $. We do not know if these
testing conditions can be further simplified. The form of these testing
conditions is dictated by our use of what we call the `maximum principle,'
see Lemma~\ref{l.maxprin}.
\end{problem}

We now recall the two weight inequalities for the Maximal Function as they
are central to the new results of this paper. Define the Maximal Function 
\begin{equation*}
\mathcal{M}\nu (x)=\sup_{x\in Q}\frac{1}{\left\vert Q\right\vert }%
\int_{Q}\left\vert \nu \right\vert ,\qquad x\in \mathbb{R},
\end{equation*}%
where the supremum is taken over all cubes $Q$ (by which we mean cubes with
sides parallel to the coordinate axes) containing $x$.

\begin{2WtMax}
\label{sawyerthm}Suppose that $\sigma $ and $\omega $ are positive locally
finite Borel measures on $\mathbb{R}^{n}$, and that $1<p<\infty $. The
maximal operator $\mathcal{M}$ satisfies the two weight norm inequality (%
\cite{Saw1})%
\begin{equation}
\left\Vert \mathcal{M}(f\sigma )\right\Vert _{L^{p}\left( \omega \right)
}\leq C\left\Vert f\right\Vert _{L^{p}\left( \sigma \right) },\qquad f\in
L^{p}\left( \sigma \right) ,  \label{M2weight}
\end{equation}%
\emph{if and only if} for all cubes $Q\subset \mathbb{R}^{n}$,%
\begin{equation}
\int_{Q}\mathcal{M}\left( \chi _{Q}\sigma \right) (x)^{p}d\omega (x)\leq
C_{1}\int_{Q}d\sigma (x).  \label{testmax}
\end{equation}%
The maximal operator $\mathcal{M}$ satisfies the \emph{weak-type} two weight
norm inequality (\cite{Mu})%
\begin{equation}
\left\Vert \mathcal{M}(f\sigma )\right\Vert _{L^{p,\infty }\left( \omega
\right) }\equiv \sup_{\lambda >0}\lambda \left\vert \left\{ \mathcal{M}%
\left( f\sigma \right) >\lambda \right\} \right\vert _{\omega }^{\frac{1}{p}%
}\leq C\left\Vert f\right\Vert _{L^{p}\left( \sigma \right) },\qquad f\in
L^{p}\left( \sigma \right) ,  \label{weakM2weight}
\end{equation}%
\emph{if and only if} the two weight $A_{p}$ condition holds for all cubes $%
Q\subset \mathbb{R}^{n}$:%
\begin{equation}
\left[ \frac{1}{\left\vert Q\right\vert }\int_{Q}d\omega \right] ^{\frac{1}{p%
}}\left[ \frac{1}{\left\vert Q\right\vert }\int_{Q}d\sigma \right] ^{\frac{1%
}{p^{\prime }}}\leq C_{2}.  \label{Ap}
\end{equation}
\end{2WtMax}

The necessary and sufficient condition \eqref{testmax} for the strong type
inequality \eqref{M2weight} states that one need only test the strong type
inequality for functions of the form $\chi _{Q}\sigma $. Not only that, but
the full $L^{p}(\omega )$ norm of $\mathcal{M}(\chi _{Q}\sigma )$ need not
be evaluated. There is a corresponding weak-type interpretation of the $%
A_{p} $ condition \eqref{Ap}. Finally, the proofs given in \cite{Saw1} and 
\cite{Mu} for absolutely continuous weights carry over without difficulty
for\ the locally finite measures considered here.

\subsection{Two Weight Inequalities for Singular Integrals}

Let us set notation for our Theorems. Consider a kernel function $K(x,y)$
defined on $\mathbb{R}^{n}\times \mathbb{R}^{n}$ satisfying the following
size and smoothness conditions,%
\begin{equation}
\begin{split}
\left\vert K(x,y)\right\vert & \leq C\left\vert x-y\right\vert ^{-n}, \\
\left\vert K(x,y)-K\left( x^{\prime },y\right) \right\vert & \leq C\delta
\left( \frac{\left\vert x-x^{\prime }\right\vert }{\left\vert x-y\right\vert 
}\right) \left\vert x-y\right\vert ^{-n},\qquad \frac{\left\vert x-x^{\prime
}\right\vert }{\left\vert x-y\right\vert }\leq \frac{1}{2},
\end{split}
\label{sizeandsmoothness}
\end{equation}%
where $\delta $ is a Dini modulus of continuity, i.e. a nondecreasing
function on $\left[ 0,1\right] $ with $\delta (0)=0$ and $\int_{0}^{1}\delta
(s)\frac{ds}{s}<\infty $.

Next we describe the truncations we consider. Let $\zeta ,\eta $ be fixed
smooth functions on the real line satisfying%
\begin{eqnarray*}
&&\zeta (t)=0\text{ for }t\leq \frac{1}{2}\text{ and }\zeta (t)=1\text{ for }%
t\geq 1, \\
&&\eta (t)=0\text{ for }t\geq 2\text{ and }\eta (t)=1\text{ for }t\leq 1, \\
&&\zeta \text{ is nondecreasing and }\eta \text{ is nonincreasing}.
\end{eqnarray*}%
Given $0<\varepsilon <R<\infty $, set $\zeta _{\varepsilon }(t)=\zeta \left( 
\frac{t}{\varepsilon }\right) $ and $\eta _{R}(t)=\eta \left( \frac{t}{R}%
\right) $ and define the smoothly truncated operator $T_{\varepsilon ,R}$ on 
$L_{\QTR{up}{loc}}^{1}\left( \mathbb{R}^{n}\right) $ by the absolutely
convergent integrals 
\begin{equation*}
T_{\varepsilon ,R}f(x)=\int K(x,y)\zeta _{\varepsilon }(\left\vert
x-y\right\vert )\eta _{R}(\left\vert x-y\right\vert )f(y)dy,\qquad f\in L_{%
\QTR{up}{loc}}^{1}\left( \mathbb{R}^{n}\right) .
\end{equation*}%
Define the \emph{maximal} singular integral operator $T_{\flat }$ on $L_{%
\QTR{up}{loc}}^{1}\left( \mathbb{R}^{n}\right) $ by%
\begin{equation*}
T_{\flat }f(x)=\sup_{0<\varepsilon <R<\infty }\left\vert T_{\varepsilon
,R}f(x)\right\vert ,\qquad x\in \mathbb{R}^{n}.
\end{equation*}

We also define a corresponding \emph{new} notion of \emph{strongly maximal}
singular integral operator $T_{\natural }$ as follows. In dimension $n=1$ we
set%
\begin{equation*}
T_{\natural }f(x)=\sup_{0<\varepsilon _{i}<R<\infty ,\frac{1}{4}\leq \frac{%
\varepsilon _{1}}{\varepsilon _{2}}\leq 4}\left\vert T_{\mathbf{\varepsilon }%
,R}f(x)\right\vert ,\qquad x\in \mathbb{R},
\end{equation*}%
where $\mathbf{\varepsilon }=\left( \varepsilon _{1},\varepsilon _{2}\right) 
$ and%
\begin{equation*}
T_{\mathbf{\varepsilon },R}f(x)=\int K(x,y)\left\{ \zeta _{\varepsilon
_{1}}(x-y)+\zeta _{\varepsilon _{2}}(y-x)\right\} \eta _{R}\left( \left\vert
x-y\right\vert \right) f(y)dy.
\end{equation*}%
Thus the local singularity has been removed by a \emph{noncentered} smooth
cutoff - $\varepsilon _{1}$ to the left of $x$ and $\varepsilon _{2}$ to the
right of $x$, but with controlled eccentricity $\frac{\varepsilon _{1}}{%
\varepsilon _{2}}$. There is a similar definition of $T_{\natural }f$ in
higher dimensions involving in place of $\zeta _{\varepsilon }\left(
\left\vert x-y\right\vert \right) $, a product of smooth cutoffs,%
\begin{equation*}
\zeta _{\mathbf{\varepsilon }}(x-y)\equiv 1-\prod_{k=1}^{n}\left[ 1-\left\{
\zeta _{\varepsilon _{2k-1}}(x_{k}-y_{k})+\zeta _{\varepsilon
_{2k}}(y_{k}-x_{k})\right\} \right] ,
\end{equation*}%
satisfying $\frac{1}{4}\leq \frac{\varepsilon _{2k-1}}{\varepsilon _{2k}}%
\leq 4$ for $1\leq k\leq n$. The advantage of this larger operator $%
T_{\natural }$ is that in many cases boundedness of $T_{\natural }$ (or
collections thereof) implies boundedness of the maximal operator $\mathcal{M}
$. Our method of proving boundedness of $T_{\flat }$ and $T_{\natural }$
requires boundedness of the maximal operator $\mathcal{M}$ anyway, and as a
result we can in some cases give necessary and sufficient conditions for
strong boundedness of $T_{\natural }$. As for weak-type boundedness, we can
in many more cases give necessary and sufficient conditions for weak
boundedness of the usual truncations $T_{\flat }$.

\begin{definition}
\label{d.standard} We say that $T$ is a \emph{standard singular integral
operator with kernel $K$} if $T$ is a bounded linear operator on $%
L^{q}\left( \mathbb{R}^{n}\right) $ for some fixed $1<q<\infty $, that is 
\begin{equation}
\left\Vert Tf\right\Vert _{L^{q}\left( \mathbb{R}^{n}\right) }\leq
C\left\Vert f\right\Vert _{L^{q}\left( \mathbb{R}^{n}\right) },\qquad f\in
L^{q}\left( \mathbb{R}^{n}\right) ,  \label{Tbound}
\end{equation}%
if $K(x,y)$ is defined on $\mathbb{R}^{n}\times \mathbb{R}^{n}$ and
satisfies both \eqref{sizeandsmoothness} and the H\"{o}rmander condition, 
\begin{equation}
\int_{B\left( y,2\varepsilon \right) ^{c}}\left\vert K(x,y)-K\left(
x,y^{\prime }\right) \right\vert dx\leq C,\qquad y^{\prime }\in B\left(
y,\varepsilon \right) ,\varepsilon >0,  \label{Hormander}
\end{equation}%
and finally if $T$ and $K$ are related by%
\begin{equation}
Tf(x)=\int K(x,y)f(y)dy,\qquad \text{a.e.-}x\notin supp\ f,  \label{rel}
\end{equation}%
whenever $f\in L^{q}\left( \mathbb{R}^{n}\right) $ has compact support in $%
\mathbb{R}^{n}$. We call a kernel $K(x,y)$ \emph{standard} if it satisfies (%
\ref{sizeandsmoothness}) and \eqref{Hormander}.
\end{definition}

For standard singular integral operators, we have this classical result.
(See the appendix on truncation of singular integrals on page 30 of \cite{St}
for the case $R=\infty $; the case $R<\infty $ is similar.)

\begin{theorem}
\label{classmax}Suppose that $T$ is a standard singular integral operator.
Then the map $f\rightarrow T_{\flat }f$ is of weak-type $( 1,1) $, and
bounded on $L^{p}\left( \mathbb{R}\right) $ for $1<p<\infty $. There exist
sequences $\varepsilon _{j}\rightarrow 0$ and $R_{j}\rightarrow \infty $
such that for $f\in L^{p}\left( \mathbb{R}\right) $ with $1\leq p<\infty $,%
\begin{equation*}
\lim_{j\rightarrow \infty }T_{\varepsilon _{j},R_{j}}f(x) \equiv T_{0,\infty
}f(x)
\end{equation*}%
exists for $a.e.\ x\in \mathbb{R}$. Moreover, there is a bounded measurable
function $a(x) $ (depending on the sequences) satisfying%
\begin{equation*}
Tf(x) =T_{0,\infty }f(x) +a(x) f( x) ,\qquad x\in \mathbb{R}^{n}.
\end{equation*}
\end{theorem}

We state a conjecture, so that the overarching goals of this subject are
clear.

\begin{conjecture}
\label{j.tooStrong} Suppose that $\sigma $ and $\omega $ are positive Borel
measures on $\mathbb{R}^{n}$, $1<p<\infty $, and $T$ is a standard singular
integral operator on $\mathbb{R}^{n}$. Then the following two statements are
equivalent: 
\begin{gather*}
\int \lvert T(f\sigma )\rvert ^{p}\omega \leq C\int \lvert f\rvert
^{p}\sigma \,,\qquad f\in C_{0}^{\infty }, \\
\left\{ 
\begin{array}{c}
\left[ \frac{1}{\left\vert Q\right\vert }\int_{Q}d\omega \right] ^{\frac{1}{p%
}}\left[ \frac{1}{\left\vert Q\right\vert }\int_{Q}d\sigma \right] ^{\frac{1%
}{p^{\prime }}}\leq C\,, \\ 
\int_{Q}\lvert T\chi _{Q}\sigma \rvert ^{p}\leq C^{\prime }\int_{Q}\sigma \,,
\\ 
\int_{Q}\lvert T^{\ast }\chi _{Q}\omega \rvert ^{p^{\prime }}\sigma \leq
C^{\prime \prime }\int_{Q}\omega \,,%
\end{array}%
\right. \ \ \ \ \ for\ all\ cubes\ Q.\qquad
\end{gather*}
\end{conjecture}

%%%%%%%%%%%%%%%%%%%%%%%%%%%%%% REMARK REMARK REMARK

\begin{remark}
\label{r.poisson} The first of the three testing conditions above is the
two-weight $A_p$ condition. We would expect that this condition can be
strengthened to a `Poisson two-weight $A_p$ condition.' See \cite{NTV3,Vo}.
\end{remark}

%%%%%%%%%%%%%%%%%%%%%%%%%%%%%% REMARK REMARK REMARK

The most important instances of this Conjecture occur when $T$ is one of a
few canonical singular integral operators, such as the Hilbert transform,
the Beurling Transform, or the Riesz Transforms. This question occurs in
different instances, such as the Sarason Conjecture concerning the
composition of Hankel operators, or the semi-commutator of Toeplitz
operators (see \cite{CrMaPe}, \cite{Zh}), Mathematical Physics \cite{PeVoYu}%
, as well as perturbation theory of some self-adjoint operators. See
references in \cite{Vo}.

To date, this has only been verified for positive operators, such as Poisson
integrals, and fractional integral operators \cite{Saw}, \cite{Saw1} and 
\cite{Saw2}. Recently the authors have used the methods of Nazarov, Treil
and Volberg to prove a special case of the conjecture for the Hilbert
transform when $p=2$ and an energy hypothesis is assumed (\cite{LaSaUr}).
Earlier in \cite{NTV3} NTV used a stronger pivotal condition in place of the
energy hypothesis, but neither of these conditions are necessary (\cite%
{LaSaUr}). The two weight Helson-Szego Theorem was proved many years earlier
by Cotlar and Sadosky \cite{CoSa} and \cite{CoSa2}, thus the $L^{2}$ case
for the Hilbert transform is completely settled.

Nazarov, Treil and Volberg \cite{NTV1}, \cite{NTV3} have characterized those
weights for which the class of Haar multipliers is bounded when $p=2$. They
also have a result for an important special class of singular integral
operators, the `well-localized' operators of \cite{NTV2}. Citing the
specific result here would carry us too far afield, but this class includes
the important Haar shift examples, such as the one found by S. Petermichl 
\cite{Pet}, and generalized in \cite{PeVo}. Consequently, characterizations
are given in \cite{Vo} and \cite{NTV3} for the Hilbert transform and Riesz
transforms in weighted $L^{2}$ spaces under various additional hypotheses.
In particular they obtain an analogue of the case $p=2$ of the strong type
theorem below. Our results can be reformulated in the context there, which
theme we do not pursue further here.

We now characterize the weak-type two weight norm inequality for both
maximal singular integrals and strongly maximal singular integrals.

\begin{2WtSIWeakType}
\label{weaktwoweightHaar}Suppose that $\sigma $ and $\omega $ are positive
locally finite Borel measures on $\mathbb{R}^{n}$, $1<p<\infty $, and let $%
T_{\flat }$ and $T_{\natural }$ be the maximal singular integral operators
as above with kernel $K(x,y)$ satisfying \eqref{sizeandsmoothness}.

\begin{enumerate}
\item Suppose that the maximal operator $\mathcal{M}$ satisfies (\ref%
{weakM2weight}). Then $T_{\natural }$ satisfies the weak-type two weight
norm inequality%
\begin{equation}
\left\Vert T_{\natural }(f\sigma )\right\Vert _{L^{p,\infty }\left( \omega
\right) }\leq C\left\Vert f\right\Vert _{L^{p}\left( \sigma \right) },\qquad
f\in L^{p}\left( \sigma \right) ,  \label{weak2weight}
\end{equation}%
\emph{if and only if} 
\begin{equation}
\int_{Q}T_{\natural }\left( \chi _{Q}f\sigma \right) (x)d\omega (x)\leq
C_{2}\left( \int_{Q}\left\vert f(x)\right\vert ^{p}d\sigma (x)\right) ^{%
\frac{1}{p}}\left( \int_{Q}d\omega (x)\right) ^{\frac{1}{p^{\prime }}},
\label{Tsharpomega}
\end{equation}%
for all cubes $Q\subset \mathbb{R}^{n}$ and all functions $f\in L^{p}\left(
\sigma \right) $.

\item The same characterization as above holds for $T_{\flat }$ in place of $%
T_{\natural }$ everywhere.

\item Suppose that $\sigma $ and $\omega $ are absolutely continuous with
respect to Lebesgue measure, that the maximal operator $\mathcal{M}$
satisfies \eqref{weakM2weight}, and that $T$ is a {standard} singular
integral operator with kernel $K$ as above. If \eqref{weak2weight} holds for 
$T_{\natural }$ or $T_{\flat }$, then it also holds for $T$:%
\begin{equation}
\left\Vert T(f\sigma )\right\Vert _{L^{p,\infty }\left( \omega \right) }\leq
C\left\Vert f\right\Vert _{L^{p}\left( \sigma \right) },\ \ \ \ \ f\in
L^{p}\left( \sigma \right) ,f\sigma \in L^{\infty },supp\ f\sigma \ \text{%
compact}.  \label{Tabs}
\end{equation}

\item Suppose $c>0$ and that $\left\{ K_{j}\right\} _{j=1}^{J}$ is a
collection of \emph{standard} kernels such that for \emph{each} unit vector $%
\mathbf{u}$ there is $j$ satisfying 
\begin{equation}
\left\vert K_{j}\left( x,x+t\mathbf{u}\right) \right\vert \geq ct^{-n},\ \ \
\ \ t\in \mathbb{R}.  \label{Kt}
\end{equation}%
Suppose also that $\sigma $ and $\omega $ have no common point masses, i.e. $%
\sigma \left( \left\{ x\right\} \right) \cdot \omega \left( \left\{
x\right\} \right) =0$ for all $x\in \mathbb{R}^{n}$. Then%
\begin{equation*}
\left\Vert \left( T_{j}\right) _{\flat }(f\sigma )\right\Vert _{L^{p,\infty
}\left( \omega \right) }\leq C\left\Vert f\right\Vert _{L^{p}\left( \sigma
\right) },\qquad f\in L^{p}\left( \sigma \right) ,\ \ \ \ \ 1\leq j\leq J,
\end{equation*}%
if and only if the two weight $A_{p}$ condition \eqref{Ap} holds and%
\begin{eqnarray*}
\int_{Q}\left( T_{j}\right) _{\flat }\left( \chi _{Q}f\sigma \right)
(x)d\omega (x) &\leq &C_{2}\left( \int_{Q}\left\vert f(x)\right\vert
^{p}d\sigma (x)\right) ^{\frac{1}{p}}\left( \int_{Q}d\omega (x)\right) ^{%
\frac{1}{p^{\prime }}}, \\
f &\in &L^{p}\left( \sigma \right) ,\ cubes\ Q\subset \mathbb{R}^{n},\ 1\leq
j\leq J.
\end{eqnarray*}
\end{enumerate}
\end{2WtSIWeakType}

While in (1)---(3), we assume that the Maximal Function inequality holds, in
point (4), we obtain an \emph{unconditional} characterization of the
weak-type inequality for\ a large class of families of (centered) maximal
singular integral operators $T_{\flat }$. This class includes the individual
maximal Hilbert transform in one dimension, the individual maximal Beurling
transform in two dimensions, and the families of maximal Riesz transforms in
higher dimensions, see Lemma~\ref{weakdom}.

Note that in (1) above, there is only size and smoothness assumptions placed
on the kernel, so that it could for instance be a degenerate fractional
integral operator, and therefore unbounded on $L^{2}(dx)$. But, the
characterization still has content in this case, if $\omega $ and $\sigma $
are not of full dimension.

In (3), we deduce a two weight inequality for standard singular integrals $T$
without truncations when the measures are absolutely continuous. The proof
of this is easy. From \eqref{weak2weight} and the pointwise inequality $%
T_{0,\infty }f\sigma (x)\leq T_{\flat }f\sigma (x)\leq T_{\natural }f\sigma
(x)$, we obtain that for any limiting operator $T_{0,\infty }$ the map $%
f\rightarrow $ $T_{0,\infty }f\sigma $ is bounded from $L^{p}\left( \sigma
\right) $ to $L^{p,\infty }\left( \omega \right) $. By \eqref{weakM2weight} $%
f\rightarrow \mathcal{M}f\sigma $ is bounded, hence $f\rightarrow f\sigma $
is bounded, and so Theorem \ref{classmax} shows that $f\rightarrow Tf\sigma
=T_{0,\infty }f\sigma +af\sigma $ is also bounded, provided we initially
restrict attention to functions $f$ for which $f\sigma $ is bounded with
compact support.

The characterizing condition \eqref{Tsharpomega} is a weak-type condition,
with the restriction that one only needs to test the weak-type condition for
functions supported on a given cube, and test the weak-type norm over that
given cube. It also has an interpretation as a dual inequality $%
\int_{Q}\left\vert L^{\ast }\left( \chi _{Q}\omega \right) \right\vert
^{p^{\prime }}d\sigma \leq C_{2}\int_{Q}d\omega $, which we return to below,
see \eqref{unif} and \eqref{unif'}.

\bigskip

We now consider the two weight norm inequality for a strongly maximal
singular integral $T_{\natural }$, but assuming that the measure $\sigma $
is doubling.

\begin{2WtSIStrongType}
\label{twoweightHaar}Suppose that $\sigma $ and $\omega $ are positive
locally finite Borel measures on $\mathbb{R}^{n}$ with $\sigma $ doubling, $%
1<p<\infty $, and let $T_{\flat }$ and $T_{\natural }$ be the maximal
singular integral operators as above with kernel $K(x,y)$ satisfying %
\eqref{sizeandsmoothness}.\medskip

\begin{enumerate}
\item Suppose that the maximal operator $\mathcal{M}$ satisfies (\ref%
{M2weight}) and also the `dual' inequality%
\begin{equation}
\left\Vert \mathcal{M}(g\omega )\right\Vert _{L^{p^{\prime }}\left( \sigma
\right) }\leq C\left\Vert g\right\Vert _{L^{p^{\prime }}\left( \omega
\right) },\qquad g\in L^{p^{\prime }}\left( \omega \right) .
\label{M2weightdual}
\end{equation}%
Then $T_{\natural }$ satisfies the two weight norm inequality 
\begin{equation}
\int_{\mathbb{R}^{n}}T_{\natural }(f\sigma )(x)^{p}d\omega (x)\leq C\int_{%
\mathbb{R}^{n}}\left\vert f(x)\right\vert ^{p}d\sigma (x),  \label{2weight}
\end{equation}%
for all $f\in L^{p}\left( \sigma \right) $ that are bounded with compact
support in $\mathbb{R}^{n}$, \emph{if and only if } both the dual cube
testing condition \eqref{Tsharpomega} and the condition below hold: 
\begin{equation}
\int_{Q}T_{\natural }\left( \chi _{Q}g\sigma \right) (x)^{p}d\omega (x)\leq
C_{1}\int_{Q}d\sigma (x),  \label{Tsharpsigma}
\end{equation}%
for all cubes $Q\subset \mathbb{R}^{n}$ and all functions $\left\vert
g\right\vert \leq 1$.

\item The same characterization as above holds for $T_{\flat }$ in place of $%
T_{\natural }$ everywhere. In fact%
\begin{equation*}
\left\vert T_{\natural }f\sigma (x)-T_{\flat }f\sigma (x)\right\vert \leq C%
\mathcal{M}\left( f\sigma \right) (x).
\end{equation*}

\item Suppose that $\sigma $ and $\omega $ are absolutely continuous with
respect to Lebesgue measure, that the maximal operator $\mathcal{M}$
satisfies \eqref{M2weight}, and that $T$ is a {standard} singular integral
operator. If \eqref{2weight} holds for $T_{\natural }$ or $T_{\flat }$, then
it also holds for $T$:%
\begin{equation*}
\int_{\mathbb{R}^{n}}\left\vert T(f\sigma )(x)\right\vert ^{p}d\omega
(x)\leq C\int_{\mathbb{R}^{n}}\left\vert f(x)\right\vert ^{p}d\sigma (x),\ \
\ \ \ f\in L^{p}\left( \sigma \right) ,f\sigma \in L^{\infty }, \QTR{up}{supp%
}(f\sigma) \ \text{compact}.
\end{equation*}

\item Suppose that $\left\{ K_{j}\right\} _{j=1}^{n}$ is a collection of 
\emph{standard} kernels satisfying for some $c>0$, 
\begin{equation}
\pm \func{Re}K_{j}(x,y)\geq \frac{c}{\left\vert x-y\right\vert ^{n}},\qquad 
\text{for }\pm (y_{j}-x_{j})\geq \frac{1}{4}\left\vert x-y\right\vert ,
\label{j}
\end{equation}%
where $x=\left( x_{j}\right) _{1\leq j\leq n}$. If both $\omega $ and $%
\sigma $ are doubling, then \eqref{2weight} holds for $\left( T_{j}\right)
_{\natural }$ and $\left( T_{j}^{\ast }\right) _{\natural }$ for all $1\leq
j\leq n$, \emph{if and only if} both \eqref{Tsharpsigma} and %
\eqref{Tsharpomega} hold for $\left( T_{j}\right) _{\natural }$ and $\left(
T_{j}^{\ast }\right) _{\natural }$ for all $1\leq j\leq n$.
\end{enumerate}
\end{2WtSIStrongType}

Note that the second condition \eqref{Tsharpsigma} is a stronger condition
than we would like: it is the $L^{p}$ inequality, applied to functions \emph{%
bounded by }$1$ and supported on a cube $Q$, but with the $L^{p}(\sigma )$
norm of $\mathbf{1}_{Q}$ on the right side. It is easy to see that the
bounded function $g$ in \eqref{Tsharpsigma} can be replaced by $\chi _{E}$
for every compact subset $E$ of $Q$. Indeed if $L$ ranges over all
linearizations of $T_{\natural }$, then with $g_{h,Q,L}=\frac{L^{\ast
}\left( \chi _{Q}h\omega \right) }{\left\vert L^{\ast }\left( \chi
_{Q}h\omega \right) \right\vert }$ we have 
\begin{eqnarray*}
\sup_{\left\vert g\right\vert \leq 1}\int_{Q}T_{\natural }\left( \chi
_{Q}g\sigma \right) ^{p}\omega &=&\sup_{\left\vert g\right\vert \leq
1}\sup_{L}\sup_{\left\Vert h\right\Vert _{L^{p^{\prime }}\left( \omega
\right) }\leq 1}\left\vert \int_{Q}L\left( \chi _{Q}g\sigma \right) h\omega
\right\vert \\
&=&\sup_{L}\sup_{\left\Vert h\right\Vert _{L^{p^{\prime }}\left( \omega
\right) }\leq 1}\sup_{\left\vert g\right\vert \leq 1}\left\vert
\int_{Q}L^{\ast }\left( \chi _{Q}h\omega \right) g\sigma \right\vert \\
&=&\sup_{L}\sup_{\left\Vert h\right\Vert _{L^{p^{\prime }}\left( \omega
\right) }\leq 1}\int_{Q}L^{\ast }\left( \chi _{Q}h\omega \right)
g_{h,Q,L}\sigma \\
&=&\sup_{\left\Vert h\right\Vert _{L^{p^{\prime }}\left( \omega \right)
}\leq 1}\sup_{L}\int_{Q}L\left( \chi _{Q}g_{h,Q,L}\right) h\omega \sigma \\
&\leq &\sup_{\left\Vert h\right\Vert _{L^{p^{\prime }}\left( \omega \right)
}\leq 1}\sup_{L}\int_{Q}T_{\natural }\left( \chi _{Q}g_{h,Q,L}\sigma \right)
^{p}\omega .
\end{eqnarray*}%
Since $g_{h,Q,L}$ takes on only the values $\pm 1$, it is easy to see that
we can take $g=\chi _{E}$. Point (3) is again easy, just as in the previous
weak-type theorem.

And in (4), we note that the truncations in the way that we formulate them,
dominate the Maximal Function, so that our assumption on $\mathcal{M}$ in
(1)---(3) is not unreasonable. The main result of \cite{NTV3} assumes $p=2$
and that $T$ is the Hilbert transform, and makes similar kinds of
assumptions. In fact it is essentially the same as our result in the case $%
p=2$, but without doubling on $\sigma $ and only for $T$ and not $T_{\flat }$
or $T_{\natural }$. Finally, we observe that by our definition of the
truncation $T_{\natural }$, we obtain in point (4), a characterization for
doubling measures of the strong-type inequality for appropriate families of
standard singular integrals and their adjoints, including the Hilbert and
Riesz transforms, see Lemma \ref{dom}.

We do not know if the bounded function $g$ in condition \eqref{Tsharpsigma}
can be replaced by the constant function $1$.

\bigskip

We now give a characterization of the strong type weighted norm inequality
for the \emph{individual} strongly maximal Hilbert transform $T_{\natural }$
when $1<p<\infty $ and the measure $\sigma $ is \emph{doubling}. When $p>2$
we use an extra necessary condition, see \eqref{Poisson condition} below,
that involves a `dyadic' Poisson function $\sum_{\ell =0}^{\infty }\frac{%
2^{-\ell }}{\left\vert I^{\left( \ell \right) }\right\vert }\chi _{I^{\left(
\ell \right) }}\left( y\right) $ where $I$ is a dyadic interval and $%
I^{\left( \ell \right) }$ denotes its $\ell ^{th}$ ancestor in the dyadic
grid, i.e. the unique dyadic interval containing $I$ with $\left\vert
I^{\left( \ell \right) }\right\vert =2^{\ell }\left\vert I\right\vert $.
This condition is a variant of the pivotal condition of Nazarov, Treil and
Volberg in \cite{NTV3}, and in the case $1<p\leq 2$ it is a consequence of
the $A_{p}$ condition \eqref{Ap}.

\begin{theorem}
\label{improved}Suppose that $\sigma $ and $\omega $ are positive locally
finite Borel measures on $\mathbb{R}$ with $\sigma $ \emph{doubling}, $%
1<p<\infty $, and let $T_{\natural }$ be the strongly maximal Hilbert
transform. Then $T_{\natural }$ is \emph{strong} type $\left( p,p\right) $
with respect to $\left( \sigma ,\omega \right) $, i.e. 
\begin{equation*}
\left\Vert T_{\natural }\left( f\sigma \right) \right\Vert _{L^{p}\left(
\omega \right) }\leq C\left\Vert f\right\Vert _{L^{p}\left( \sigma \right) },
\end{equation*}%
for all $f$ bounded with compact support, \emph{if and only if} the
following four conditions hold. In the case $1<p\leq 2$, the fourth
condition \eqref{Poisson condition} is implied by the $A_{p}$ condition (\ref%
{Ap}), and so in this case we only need the first \emph{three} conditions
below:

\begin{enumerate}
\item the dual $T_{\natural }$ interval testing condition%
\begin{equation}
\int_{Q}T_{\natural }\left( \chi _{Q}f\sigma \right) d\omega \leq C\left(
\int_{Q}\left\vert f\right\vert ^{p}d\sigma \right) ^{\frac{1}{p}}\left(
\int_{Q}d\omega \right) ^{\frac{1}{p^{\prime }}},
\label{dual testing condition Q}
\end{equation}%
holds for all intervals $Q$ and $f\in L_{Q}^{p}\left( \sigma \right) $;

\item the forward $T_{\natural }$ testing condition%
\begin{equation}
\int_{Q}T_{\natural }\left( \chi _{E}\sigma \right) ^{p}d\omega \leq
C\int_{Q}d\sigma ,  \label{testing condition E}
\end{equation}%
holds for all intervals $Q$ and all compact subsets $E$ of $Q$;

\item the strengthened $A_{p}$ condition%
\begin{equation}
\left( \int_{\mathbb{R}}\left( \frac{\left\vert Q\right\vert }{\left\vert
Q\right\vert +\left\vert x-x_{Q}\right\vert }\right) ^{p}d\omega \left(
x\right) \right) ^{\frac{1}{p}}\left( \int_{\mathbb{R}}\left( \frac{%
\left\vert Q\right\vert }{\left\vert Q\right\vert +\left\vert
x-x_{Q}\right\vert }\right) ^{p^{\prime }}d\sigma \left( x\right) \right) ^{%
\frac{1}{p^{\prime }}}\leq C\left\vert Q\right\vert ,  \label{skirt Ap'}
\end{equation}%
holds for all intervals $Q$.

\item the Poisson condition%
\begin{equation}
\int_{\mathbb{R}}\left( \sum_{r=1}^{\infty }\left\vert I_{r}\right\vert
_{\sigma }\left\vert I_{r}\right\vert ^{p^{\prime }-1}\sum_{\ell =0}^{\infty
}\frac{2^{-\ell }}{\left\vert \left( I_{r}\right) ^{\left( \ell \right)
}\right\vert }\chi _{\left( I_{r}\right) ^{\left( \ell \right) }}\left(
y\right) \right) ^{p}d\omega \left( y\right) \leq C\sum_{r=1}^{\infty
}\left\vert I_{r}\right\vert _{\sigma }\left\vert I_{r}\right\vert
^{p^{\prime }},  \label{Poisson condition}
\end{equation}%
for all pairwise disjoint decompositions $Q=\cup _{r=1}^{\infty }I_{r}$ of
the dyadic interval $Q$ into dyadic intervals $I_{r}$, for any fixed dyadic
grid.
\end{enumerate}
\end{theorem}

\begin{remark}
The strengthened $A_{p}$ condition \eqref{skirt Ap'} can be replaced with
the weaker `half' condition where the first factor on the left is replaced
by $\left( \int_{Q}d\omega \right) ^{\frac{1}{p}}$. We do not know if the
first three conditions suffice when $p>2$.
\end{remark}

\begin{acknowledgement}
The authors began this work during research stays at the Fields Institute,
Toronto Canada, and continued at the Centre de Recerca Matem\`{a}tica,
Barcelona Spain. They thank these institutions for their generous
hospitality. In addition, this paper has been substantially improved by the
careful attention of the referee, for which we are particularly grateful.
\end{acknowledgement}

\section{Overview of the Proofs, General Principles}

If $Q$ is a cube then $\ell (Q)$ is its side length, $\left\vert
Q\right\vert $ is its Lebesgue measure and for a positive Borel measure $\nu 
$, $\left\vert Q\right\vert _{\nu }=\int_{Q}d\nu $ is its $\nu $-measure.

\subsection{Calder\'on-Zygmund Decompositions}

Our starting place is the argument in \cite{Saw2} used to prove a two weight
norm inequality for {fractional} integral operators on Euclidean space. Of
course the fractional integral is a positive operator, with a monotone
kernel, which properties we do not have in the current setting.

A central tool arises from the observation that for any positive Borel
measure $\mu $, one has the boundedness of a maximal function associated
with $\mu $. Define the dyadic $\mu $-maximal operator $\mathcal{M}_{\mu
}^{dy}$ by 
\begin{equation}
\mathcal{M}_{\mu }^{dy}f(x)=\sup_{\substack{ Q\in \mathcal{D}  \\ x\in Q}}%
\frac{1}{\left\vert Q\right\vert _{\mu }}\int_{Q}\left\vert f\right\vert \mu
,  \label{e.MmuDef}
\end{equation}%
with the supremum taken over all dyadic cubes $Q\in \mathcal{D}$ containing $%
x$. It is immediate to check that $\mathcal{M}_{\mu }^{dy}$ satisfies the
weak-type $(1,1)$ inequality, and the $L^{\infty }(\mu )$ bound is obvious.
Hence we have 
\begin{equation}
\int \left( \mathcal{M}_{\mu }^{dy}f\right) ^{p}\mu \leq C\int f^{p}\mu
,\;\;\;\;\;f\geq 0\text{ on }\mathbb{R}^{n}.  \label{maxthm}
\end{equation}%
This observation places certain Calder\'{o}n-Zygmund decompositions at our
disposal. Exploitation of this brings in the testing condition (\ref%
{Tsharpsigma}) involving the bounded function $g$ on a cube $Q$, and indeed, 
$g$ turns out to be the \textquotedblleft good\textquotedblright\ function
in a Calder\'{o}n-Zygmund decomposition of $f$ on $Q$. The associated `bad'
function requires the dual testing condition \eqref{Tsharpomega} as well.

\subsection{Edge effects of dyadic grids}

\label{s.shift}

Our operators are not dyadic operators, nor---in contrast to the fractional
integral operators---can they be easily obtained from dyadic operators. This
leads to the necessity of considering for instance triples of dyadic cubes,
which are not dyadic.

Also, dyadic grids distinguish points by for instance making some points on
the boundary of many cubes. As our measures are arbitrary, they could
conspire to assign extra mass to some of these points. To address this
point, Nazarov-Treil-Volberg \cite{NTV3,NTV4,NTV5} use a random shift of the
grid.

A random approach would likely work for us as well, though the argument
would be different from those in the cited papers above. Instead, we will
use a non-random technique of {shifted dyadic grid} from \cite{MR1887641},
which goes back to P. Jones and J. Garnett. Define a \emph{shifted dyadic
grid} to be the collection of cubes 
\begin{equation}
\mathcal{D}^{\alpha }=\bigl\{2^{j}(k+[0,1)^{n}+(-1)^{j}\alpha )\;:\;j\in 
\mathbb{Z},k\in \mathbb{Z}^{n}\bigr\}\,,\qquad \alpha \in \{0,\tfrac{1}{3},%
\tfrac{2}{3}\}^{n}\,.  \label{e.shifted}
\end{equation}%
The basic properties of these collections are these: In the first place,
each $\mathcal{D}^{\alpha }$ is a grid, namely for $Q,Q^{\prime }\in 
\mathcal{D}^{\alpha }$ we have $Q\cap Q^{\prime }\in \{\emptyset
\,,\,Q\,,\,Q^{\prime }\}$ and $Q$ is a union of $2^{n}$ elements of $%
\mathcal{D}^{\alpha }$ of equal volume. In the second place (and this is the
novel property for us), for any cube $Q\subset \mathbb{R}^{n}$, there is a
choice of some $\alpha $ and some $Q^{\prime }\in \mathcal{D}_{\alpha }$ so
that $Q\subset \frac{9}{10}Q^{\prime }$ and $\lvert Q^{\prime }\rvert \leq
C\lvert Q\rvert $.

We define the analogs of the dyadic maximal operator in \eqref{e.MmuDef},
namely 
\begin{equation}
\mathcal{M}_{\mu }^{\alpha }f(x)=\sup_{\substack{ Q\in \mathcal{D}^{\alpha } 
\\ x\in Q}}\frac{1}{\left\vert Q\right\vert _{\mu }}\int_{Q}\left\vert
f\right\vert \mu \,.  \label{e.MmuAlphaDef}
\end{equation}%
These operators clearly satisfy \eqref{maxthm}. Shifted dyadic grids will
return in \S ~\ref{s.czd}.

\subsection{A Maximum Principle}

A second central tool is a `maximum principle' (or good $\lambda $
inequality) which will permit one to localize large values of a singular
integral, provided the Maximal Function is bounded. It is convenient for us
to describe this in conjunction with another fundamental tool of this paper,
a family of Whitney decompositions.

We begin with the Whitney decompositions. Fix a finite measure $\nu $ with
compact support on $\mathbb{R}^{n}$ and for $k\in \mathbb{Z}$ let%
\begin{equation}
\Omega _{k}=\left\{ x\in \mathbb{R}^{n}:T_{\natural }\nu (x)>2^{k}\right\} .
\label{e.OmegaK}
\end{equation}%
Note that $\Omega _{k}\neq \mathbb{R}^{n}$ has compact closure for such $\nu 
$. Fix an integer $N\geq 3$. We can choose $R_{W}\geq 3$ sufficiently large,
depending only on the dimension and $N$, such that there is a collection of
cubes $\left\{ Q_{j}^{k}\right\} _{j}$ which satisfy the following
properties: 
\begin{equation}
\left\{ 
\begin{array}{ll}
\text{(disjoint cover)} & \Omega _{k}=\bigcup_{j}Q_{j}^{k}\text{ and }%
Q_{j}^{k}\cap Q_{i}^{k}=\emptyset \text{ if }i\neq j \\ 
\text{(Whitney condition)} & R_{W}Q_{j}^{k}\subset \Omega _{k}\text{ and }%
3R_{W}Q_{j}^{k}\cap \Omega _{k}^{c}\neq \emptyset \text{ for all }k,j \\ 
\text{(bounded overlap)} & \sum_{j}\chi _{NQ_{j}^{k}}\leq C\chi _{\Omega
_{k}}\text{ for all }k \\ 
\text{(crowd control)} & \#\left\{ Q_{s}^{k}:Q_{s}^{k}\cap NQ_{j}^{k}\neq
\emptyset \right\} \leq C\text{ for all }k,j \\ 
\text{(nested property)} & Q_{j}^{k}\varsubsetneqq Q_{i}^{\ell }\text{
implies }k>\ell%
\end{array}%
\right. .  \label{Whitney}
\end{equation}

Indeed, one should choose the the $\left\{ Q_{j}^{k}\right\} _{j}$
satisfying the Whitney condition, and then show that the other properties
hold. The different combinatorial properties above are fundamental to the
proof. And alternate Whitney decompositions are constructed in \S ~\ref%
{s.whitneyShift} below.

\begin{remark}
\label{3N} Our use of the Whitney decomposition and the maximum principle
are derived from the two weight fractional integral argument of Sawyer, see {%
Sec 2 of} \cite{Saw2}. In particular, the properties above are as in \cite%
{Saw2}, aside from the the crowd control property above, which is $N=3$ in 
\emph{op. cit.}
\end{remark}

%%%%%%%%%%%%%%%%%%%%%%%%%%%%%% REMARK REMARK REMARK

\begin{remark}
\label{r.super/sub} In our notation for the Whitney cubes, the superscript
indicates a `height' and the subscript an arbitrary enumeration of the
cubes. We will use super- and sub-scripts below in this manner consistently
throughout the paper. It is important to note that a fixed cube $Q$ can
arise in \emph{many} Whitney decompositions: There are integers $K_- (Q)\le
K_+ (Q)$ with $Q= Q ^{k} _{j(k) }$ for some choice of $j (k)$ for all $K_-
(Q)\le k \le K_+ (Q)$. (The last point follows from the nested property.)
There is no \emph{a priori} upper bound on $K_+ (Q) - K_- (Q)$.
\end{remark}

%%%%%%%%%%%%%%%%%%%%%%%%%%%%%% REMARK REMARK REMARK

\begin{lemma}
\label{l.maxprin}[Maximum Principle] Let $\nu \ $be a finite (signed)
measure with compact support. For any cube $Q_{j}^{k}$ as above we have the
pointwise inequality 
\begin{equation}
\sup_{x\in Q_{j}^{k}}T_{\natural }\left( \chi _{(3 Q_{j}^{k})^{c}}\nu
\right) (x)\leq 2^{k}+C\mathbf{P}\left( Q_{j}^{k},\nu \right) \leq
2^{k}+CM\left( Q_{j}^{k},\nu \right) ,  \label{maxprinc}
\end{equation}%
where $\mathbf{P}\left( Q,\nu \right) $ and $M\left( Q,\nu \right) $ are
defined by 
\begin{eqnarray}
\mathbf{P}\left( Q,\nu \right) &\equiv &\frac{1}{\left\vert Q\right\vert }%
\int_{Q}d\left\vert \nu \right\vert +\sum_{\ell =0}^{\infty }\frac{\delta
\left( 2^{-\ell }\right) }{\left\vert 2^{\ell +1}Q\right\vert }\int_{2^{\ell
+1}Q\setminus 2^{\ell }Q}d\left\vert \nu \right\vert ,  \label{MQmu} \\
M\left( Q,\nu \right) &\equiv &\sup_{Q^{\prime }\supset Q}\frac{1}{%
\left\vert Q^{^{\prime }}\right\vert }\int_{Q^{\prime }}d\left\vert \nu
\right\vert \,.  \notag
\end{eqnarray}
\end{lemma}

The bound in terms of $\mathbf{P}\left( Q,\nu \right)$ should be regarded as
one in terms of a modified Poisson integral. It is both slightly sharper
than that of $M\left( Q,\nu \right)$, and a linear expression in $\lvert
\nu\rvert $, which fact will be used in the proof of the strong type
estimates.

\begin{proof}
To see this, take $x\in Q_{j}^{k}$ and note that for each $\eta >0$ there is 
$\mathbf{\varepsilon }$ with $\ell (Q_{j}^{k})<\max_{1\leq j\leq
n}\varepsilon _{j}<R<\infty $ and $\theta \in \left[ 0,2\pi \right) $ such
that%
\begin{eqnarray*}
T_{\natural }\left( \chi _{(3Q_{j}^{k})^{c}}\nu \right) (x) &\leq &\left(
1+\eta \right) \left\vert \int_{(3Q_{j}^{k})^{c}}K(x,y)\zeta _{\mathbf{%
\varepsilon }}(x-y)\eta _{R}(x-y)d\nu (y)\right\vert \\
&=&\left( 1+\eta \right) e^{i\theta }T_{\mathbf{\varepsilon },R}\left( \chi
_{(3Q_{j}^{k})^{c}}\nu \right) (x).
\end{eqnarray*}%
For convenience we take $\eta =0$ in the sequel. By the Whitney condition in %
\eqref{Whitney}, there is a point $z\in 3R_{W}Q_{j}^{k}\cap \Omega _{k}^{c}$
and it now follows that (remember that $\ell (Q_{j}^{k})<\max_{1\leq j\leq
n}\varepsilon _{j}$), 
\begin{eqnarray*}
&&\left\vert T_{\mathbf{\varepsilon },R}\left( \chi _{(3Q_{j}^{k})^{c}}\nu
\right) (x)-T_{\mathbf{\varepsilon },R}\nu (z)\right\vert \\
&\leq &C\frac{1}{\left\vert 6R_{W}Q_{j}^{k}\right\vert }%
\int_{6R_{W}Q_{j}^{k}}d\left\vert \nu \right\vert +\left\vert T_{\mathbf{%
\varepsilon },R}\left( \chi _{(6R_{W}Q_{j}^{k})^{c}}\nu \right) (x)-T_{%
\mathbf{\varepsilon },R}\left( \chi _{(6R_{W}Q_{j}^{k})^{c}}\nu \right)
(z)\right\vert \\
&=&C\frac{1}{\left\vert 6R_{W}Q_{j}^{k}\right\vert }\int_{6R_{W}Q_{j}^{k}}d%
\left\vert \nu \right\vert \\
&&{}+\int_{(6R_{W}Q_{j}^{k})^{c}}\left\vert K(x,y)\zeta _{\mathbf{%
\varepsilon }}(x-y)\eta _{R}(x-y)-K(z,y)\zeta _{\mathbf{\varepsilon }%
}(z-y)\eta _{R}(z-y)\right\vert d\left\vert \nu \right\vert (y) \\
&\leq &C\frac{1}{\left\vert 6R_{W}Q_{j}^{k}\right\vert }%
\int_{6R_{W}Q_{j}^{k}}d\left\vert \nu \right\vert
+C\int_{(6R_{W}Q_{j}^{k})^{c}}\delta \left( \frac{\left\vert x-z\right\vert 
}{\left\vert x-y\right\vert }\right) \frac{1}{\left\vert x-y\right\vert ^{n}}%
d\left\vert \nu \right\vert (y) \\
&\leq &C\mathbf{P}\left( Q_{j}^{k},\nu \right) .
\end{eqnarray*}%
Thus%
\begin{equation*}
T_{\natural }\left( \chi _{(3Q_{j}^{k})^{c}}\nu \right) (x)\leq \left\vert
T_{\natural }\nu (z)\right\vert +C\mathbf{P}\left( Q_{j}^{k},\nu \right)
\leq 2^{k}+C\mathbf{P}\left( Q_{j}^{k},\nu \right) ,
\end{equation*}%
which yields \eqref{maxprinc} since $\mathbf{P}\left( Q,\nu \right) \leq
CM\left( Q,\nu \right) $.
\end{proof}

\subsection{Linearizations}

We now make comments on the linearizations of our maximal singular integral
operators. We would like, at different points, to treat $T_{\natural }$ as a
linear operator, which of course it is not. Nevertheless $T_{\natural }$ is
a pointwise supremum of the linear truncation operators $T_{\mathbf{%
\varepsilon },R}$, and as such, the supremum can be linearized with
measurable selection of the parameters $\mathbf{\varepsilon }$ and $R$, as
was just done in the previous proof. We make this a definition.

\begin{definition}
\label{d.linearization} We say that $L$ is a linearization of $T_{\natural }$
if there are measurable functions $\mathbf{\varepsilon }(x)\in \left(
0,\infty \right) ^{n}$ and $R(x)\in \left( 0,\infty \right) $ with $\frac{1}{%
4}\leq \frac{\varepsilon _{i}}{\varepsilon _{j}}\leq 4$, $\max_{1\leq i\leq
n}\varepsilon _{i}<R(x)<\infty $ and $\theta (x)\in \left[ 0,2\pi \right) $
such that 
\begin{equation}
Lf(x)=e^{i\theta (x)}T_{\mathbf{\varepsilon }(x),R(x)}f(x),\qquad x\in 
\mathbb{R}^{n}.  \label{defLa}
\end{equation}
\end{definition}

For fixed $f$ and $\delta >0$, we can always choose a linearization $L$ so
that $T_{\natural }f(x)\leq \left( 1+\delta \right) Lf(x)$ for all $x$. In a
typical application of this Lemma, one takes $\delta $ to be one.

Note that condition \eqref{Tsharpsigma} is obtained from inequality (\ref%
{2weight}) by testing over $f$ of the form $f=\chi _{Q}g$ with $\left\vert
g\right\vert \leq 1$, and then restricting integration on the left to $Q$.
By passing to linearizations $L$, we can `dualize' \eqref{Tsharpomega} to
the testing conditions%
\begin{equation}
\int_{Q}\left\vert L^{\ast }\left( \chi _{Q}\omega \right) (x)\right\vert
^{p^{\prime }}d\sigma (x)\leq C_{2}\int_{Q}d\omega (x),  \label{unif}
\end{equation}%
or equivalently (note that in \eqref{Tsharpsigma} the presence of $g$ makes
a difference, but not here),%
\begin{equation}
\int_{Q}\left\vert L^{\ast }\left( \chi _{Q}g\omega \right) (x)\right\vert
^{p^{\prime }}d\sigma (x)\leq C_{2}\int_{Q}d\omega (x),\qquad \left\vert
g\right\vert \leq 1,  \label{unif'}
\end{equation}%
with the requirement that these inequalities hold \emph{uniformly} in all
linearizations $L$ of $T_{\natural }$.

While the smooth truncation operators $T_{\mathbf{\varepsilon },R}$ are
essentially self-adjoint, the dual of a linearization $L$ is generally
complicated. Nevertheless, the dual $L^{\ast }$ does satisfy one important
property which plays a crucial role in the proof of Theorem \ref%
{twoweightHaar}, the $L^{p}$-norm inequalities.

\begin{lemma}
\label{constant}$L^{\ast }\mu $ is $\delta $-H\"{o}lder continuous (where $%
\delta $ is the Dini modulus of continuity of the kernel $K$) with constant $%
C\mathbf{P}\left( Q,\mu \right) $ on any cube $Q$ satisfying $%
\int_{3Q}d\left\vert \mu \right\vert =0$, i.e.%
\begin{equation}
\left\vert L^{\ast }\mu (y)-L^{\ast }\mu \left( y^{\prime }\right)
\right\vert \leq C\mathbf{P}\left( Q,\mu \right) \delta \left( \frac{%
\left\vert y-y^{\prime }\right\vert }{\ell (Q)}\right) ,\ \ \ \ \
y,y^{\prime }\in Q.  \label{Lip}
\end{equation}%
Here, recall the definition \eqref{MQmu} and that $\mathbf{P}\left( Q,\mu
\right) \leq CM\left( Q,\mu \right) $.
\end{lemma}

\begin{proof}
Suppose $L$ is as in \eqref{defLa}. Then for any finite measure $\nu $,%
\begin{equation*}
L\nu (x)=e^{i\theta (x)}\int \zeta _{\mathbf{\varepsilon }(x)}(x-y)\eta
_{R(x)}(x-y)K(x,y)d\nu (y).
\end{equation*}%
Fubini's theorem shows that the dual operator $L^{\ast }$ is given on a
finite measure $\mu $ by%
\begin{equation}
L^{\ast }\mu (y)=\int \zeta _{\mathbf{\varepsilon }(x)}(x-y)\eta
_{R(x)}(x-y)K(x,y)e^{i\theta (x)}d\mu (x).  \label{dualL}
\end{equation}%
For $y,y^{\prime }\in Q$ and $\left\vert \mu \right\vert (3Q)=0$, we thus
have%
\begin{eqnarray*}
&&L^{\ast }\mu (y)-L^{\ast }\mu \left( y^{\prime }\right) \\
&=&\int \left\{ \left( \zeta _{\mathbf{\varepsilon }(x)}\eta _{R(x)}\right)
(x-y)-\left( \zeta _{\mathbf{\varepsilon }(x)}\eta _{R(x)}\right) \left(
x-y^{\prime }\right) \right\} K(x,y)e^{i\theta (x)}d\mu (x) \\
&&{}+\int \left( \zeta _{\mathbf{\varepsilon }(x)}\eta _{R(x)}\right) \left(
x-y^{\prime }\right) \left\{ K(x,y)-K\left( x,y^{\prime }\right) \right\}
e^{i\theta (x)}d\mu (x),
\end{eqnarray*}%
from which \eqref{Lip} follows easily if we split the two integrals in $x$
over dyadic annuli centered at the center of $Q$.
\end{proof}

\subsection{Control of Maximal Functions}

Next we record the facts that $T$ and $T_{\natural }$ control $\mathcal{M}$
for many (collections of) standard singular integrals $T$, including the
Hilbert transform, the Beurling transform and the collection of Riesz
transforms in higher dimensions.

\begin{lemma}
\label{weakdom}Suppose that $\sigma $ and $\omega $ have no point masses in
common, and that $\left\{ K_{j}\right\} _{j=1}^{J}$ is a collection of \emph{%
standard} kernels satisfying \eqref{sizeandsmoothness} and \eqref{Kt}. If
the corresponding operators $T_{j}$ given by \eqref{rel} satisfy%
\begin{equation*}
\left\Vert \chi _{E}T_{j}(f\sigma )\right\Vert _{L^{p,\infty }\left( \omega
\right) }\leq C\left\Vert f\right\Vert _{L^{p}\left( \sigma \right) },\ \ \
\ \ E=\mathbb{R}^{n}\setminus supp\ f,
\end{equation*}%
for $1\leq j\leq J$, then the two weight $A_{p}$ condition \eqref{Ap} holds,
and hence also the weak-type two weight inequality \eqref{weakM2weight}.
\end{lemma}

\begin{proof}
Part of the `one weight' argument on page 211 of Stein \cite{St2} yields the 
\emph{asymmetric} two weight $A_{p}$ condition%
\begin{equation}
\left\vert Q\right\vert _{\omega }\left\vert Q^{\prime }\right\vert _{\sigma
}^{p-1}\leq C\left\vert Q\right\vert ^{p},  \label{asym}
\end{equation}%
where $Q$ and $Q^{\prime }$ are cubes of equal side length $r$ and distance
approximately $C_{0}r$ apart for some fixed large positive constant $C_{0}$
(for this argument we choose the unit vector $\mathbf{u}$ in \eqref{Kt} to
point in the direction from the center of $Q$ to the center of $Q^{\prime }$%
, and then with $j$ as in \eqref{Kt}, $C_{0}$ is chosen large enough by (\ref%
{sizeandsmoothness}) that \eqref{Kt} holds for all unit vectors $\mathbf{u}$
pointing from a point in $Q$ to a point in $Q^{\prime }$). In the one weight
case treated in \cite{St2} it is easy to obtain from this (even for a \emph{%
single} direction $\mathbf{u}$) the usual (symmetric) $A_{p}$ condition (\ref%
{Ap}). Here we will instead use our assumption that $\sigma $ and $\omega $
have no point masses in common for this purpose.

So fix an open dyadic cube $Q_{0}$ in $\mathbb{R}^{n}$, say with side length 
$1$, let $\mathsf{Q}_{0}=Q_{0}\times Q_{0}$\ and set 
\begin{eqnarray*}
\Omega  &=&\left\{ \mathsf{Q}=Q\times Q^{\prime }\text{ dyadic}:\mathsf{Q}%
\subset \mathsf{Q}_{0}\text{, }\ell \left( Q\right) =\ell \left( Q^{\prime
}\right) \approx C_{0}^{-1}dist\left( Q,Q^{\prime }\right) \right.  \\
&&\ \ \ \ \ \ \ \ \ \ \ \ \ \ \ \ \ \ \ \ \ \ \ \ \ \ \ \ \ \ \ \ \ \ \
\left. \text{ and \eqref{asym} holds for }Q\text{ and }Q^{\prime }\right\} .
\end{eqnarray*}

Note that with $\mathsf{Q}=Q\times Q^{\prime }$, then \eqref{asym} can be
written%
\begin{equation}
\mathcal{A}_{p}\left( \omega ,\sigma ;\mathsf{Q}\right) \leq C\left\vert 
\mathsf{Q}\right\vert ^{\frac{p}{2}},  \label{asym'}
\end{equation}%
where%
\begin{equation*}
\mathcal{A}_{p}\left( \omega ,\sigma ;\mathsf{Q}\right) =\left\vert
Q\right\vert _{\omega }\left\vert Q^{\prime }\right\vert _{\sigma }^{p-1}.
\end{equation*}%
Here $\mathcal{A}_{2}\left( \omega ,\sigma ;\mathsf{Q}\right) =\left\vert 
\mathsf{Q}\right\vert _{\omega \times \sigma }$ where $\omega \times \sigma $
denotes product measure on $\mathbb{R}^{n}\times \mathbb{R}^{n}$. 

Suppose first that $1<p\leq 2$. Divide $\mathsf{Q}_{0}$ into $2^{n}\times
2^{n}=4^{n}$ congruent subcubes $\mathsf{Q}_{0}^{1},...\mathsf{Q}_{0}^{4^{n}}
$ of side length $\frac{1}{2}$, and set aside those $\mathsf{Q}_{0}^{j}\in
\Omega $ (those for which \eqref{asym} holds) into a collection of \emph{%
stopping cubes} $\Gamma $. Continue to divide the remaining $\mathsf{Q}%
_{0}^{j}$ into $4^{n}$ congruent subcubes $\mathsf{Q}_{0}^{j,1},...\mathsf{Q}%
_{0}^{j,4^{n}}$ of side length $\frac{1}{4}$, and again, set aside those $%
\mathsf{Q}_{0}^{j,i}\in \Omega $ into $\Gamma $, and continue subdividing
those that remain. We continue with such subdivisions for $N$ generations so
that all the cubes \emph{not }set aside into $\Gamma $ have side length $%
2^{-N}$. The important property these cubes have is that they all lie within
distance $r2^{-N}$ of the diagonal $\mathcal{D}=\left\{ \left( x,x\right)
:\left( x,x\right) \in \mathsf{Q}_{0}\right\} $ in $\mathsf{Q}%
_{0}=Q_{0}\times Q_{0}$ since \eqref{asym} holds for all pairs of cubes $Q$
and $Q^{\prime }$ of equal side length $r$ having distance approximately $%
C_{0}r$ apart. Enumerate the cubes in $\Gamma $ as $\left\{ \mathsf{Q}%
_{\alpha }\right\} _{\alpha }$ and those remaining that are not in $\Gamma $
as $\left\{ \mathsf{P}_{\beta }\right\} _{\beta }$. Thus we have the
pairwise disjoint decomposition%
\begin{equation*}
\mathsf{Q}_{0}=\left( \bigcup_{\alpha }\mathsf{Q}_{\alpha }\right) \bigcup
\left( \bigcup_{\beta }\mathsf{Q}_{\beta }\right) .
\end{equation*}%
In the case $p=2$, the countable additivity of the product measure $\omega
\times \sigma $ shows that%
\begin{equation*}
\mathcal{A}_{2}\left( \omega ,\sigma ;\mathsf{Q}_{0}\right) =\sum_{\alpha }%
\mathcal{A}_{2}\left( \omega ,\sigma ;\mathsf{Q}_{\alpha }\right)
+\sum_{\beta }\mathcal{A}_{2}\left( \omega ,\sigma ;\mathsf{P}_{\beta
}\right) .
\end{equation*}%
For the more general case $1<p\leq 2$, note that at each division described
above we have using $0<p-1\leq 1$,%
\begin{eqnarray*}
\mathcal{A}_{p}\left( \omega ,\sigma ;\mathsf{Q}_{0}\right)  &=&\left(
\sum_{i=1}^{2^{n}}\left\vert \mathsf{Q}_{0}^{j}\right\vert _{\omega }\right)
\left( \sum_{i=1}^{2^{n}}\left\vert \mathsf{Q}_{0}^{j}\right\vert _{\sigma
}\right) ^{p-1} \\
&&\ \ \ \ \ \leq \left( \sum_{i=1}^{2^{n}}\left\vert \mathsf{Q}%
_{0}^{j}\right\vert _{\omega }\right) \left( \sum_{i=1}^{2^{n}}\left\vert 
\mathsf{Q}_{0}^{j}\right\vert _{\sigma }^{p-1}\right)  \\
&&\ \ \ \ \ =\sum_{j=1}^{4^{n}}\mathcal{A}_{p}\left( \omega ,\sigma ;\mathsf{%
Q}_{0}^{j}\right) , \\
\mathcal{A}_{p}\left( \omega ,\sigma ;\mathsf{Q}_{0}^{j}\right)  &\leq
&\sum_{i=1}^{4^{n}}\mathcal{A}_{p}\left( \omega ,\sigma ;\mathsf{Q}%
_{0}^{j,i}\right) ,\ \ \ \ \ \mathsf{Q}_{0}^{j}\notin \Gamma , \\
&&etc.
\end{eqnarray*}%
It follows that%
\begin{eqnarray*}
\mathcal{A}_{p}\left( \omega ,\sigma ;\mathsf{Q}_{0}\right)  &\leq
&\sum_{\alpha }\mathcal{A}_{p}\left( \omega ,\sigma ;\mathsf{Q}_{\alpha
}\right) +\sum_{\beta }\mathcal{A}_{p}\left( \omega ,\sigma ;\mathsf{P}%
_{\beta }\right)  \\
&\leq &C\sum_{\alpha }\left\vert \mathsf{Q}_{\alpha }\right\vert ^{\frac{p}{2%
}}+\sum_{\beta }\mathcal{A}_{p}\left( \omega ,\sigma ;\mathsf{P}_{\beta
}\right)  \\
&\leq &C\left\vert Q_{0}\right\vert ^{p}+\sum_{\beta }\mathcal{A}_{p}\left(
\omega ,\sigma ;\mathsf{P}_{\beta }\right) ,
\end{eqnarray*}%
where in the last line we have used the following consequence of the
Whitney-like construction of the $\mathsf{Q}_{\alpha }$: 
\begin{eqnarray*}
\sum_{\alpha }\left\vert \mathsf{Q}_{\alpha }\right\vert ^{\frac{p}{2}}
&=&\sum_{k\in \mathbb{Z}:\ 2^{k}\leq \ell \left( Q_{0}\right) }\sum_{\alpha
:\ \ell \left( Q_{\alpha }\right) =2^{k}}\left( 2^{2nk}\right) ^{\frac{p}{2}}
\\
&\approx &\sum_{k\in \mathbb{Z}:\ 2^{k}\leq \ell \left( Q_{0}\right)
}\sum_{\alpha :\ \ell \left( Q_{\alpha }\right) =2^{k}}\left( \frac{2^{k}}{%
\ell \left( Q_{0}\right) }\right) ^{-n}\left( 2^{2nk}\right) ^{\frac{p}{2}}\
\ \ \text{(Whitney)} \\
&=&\ell \left( Q_{0}\right) ^{n}\sum_{k\in \mathbb{Z}:\ 2^{k}\leq \ell
\left( Q_{0}\right) }2^{nk\left( p-1\right) } \\
&\leq &C_{p}\ell \left( Q_{0}\right) ^{n}\ell \left( Q_{0}\right) ^{n\left(
p-1\right) }=C_{p}\left\vert Q_{0}\times Q_{0}\right\vert ^{\frac{p}{2}%
}=C_{p}\left\vert \mathsf{Q}_{0}\right\vert ^{p}.
\end{eqnarray*}

Since $\omega $ and $\sigma $ have no point masses in common, it is not hard
to show, using that the side length of $\mathsf{P}_{\beta }=P_{\beta }\times
P_{\beta }^{\prime }$ is $2^{-N}$ and $dist\left( \mathsf{P}_{\beta },%
\mathcal{D}\right) \leq C2^{-N}$, that we have the following limit:%
\begin{equation*}
\sum_{\beta }\mathcal{A}_{p}\left( \omega ,\sigma ;\mathsf{P}_{\beta
}\right) \rightarrow 0\text{ as }N\rightarrow \infty .
\end{equation*}%
Indeed, if $\sigma $ has no point masses at all, then%
\begin{eqnarray*}
\sum_{\beta }\mathcal{A}_{p}\left( \omega ,\sigma ;\mathsf{P}_{\beta
}\right) &=&\sum_{\beta }\left\vert P_{\beta }\right\vert _{\omega
}\left\vert P_{\beta }^{\prime }\right\vert _{\sigma }^{p-1} \\
&\leq &\left( \sum_{\beta }\left\vert P_{\beta }\right\vert _{\omega
}\right) \sup_{\beta }\left\vert P_{\beta }^{\prime }\right\vert _{\sigma
}^{p-1} \\
&\leq &C\left\vert Q_{0}\right\vert _{\omega }\sup_{\beta }\left\vert
P_{\beta }^{\prime }\right\vert _{\sigma }^{p-1}\rightarrow 0\text{ as }%
N\rightarrow \infty .
\end{eqnarray*}%
If $\sigma $ contains a point mass $c\delta _{x}$, then%
\begin{eqnarray*}
\sum_{\beta :x\in P_{\beta }^{\prime }}\mathcal{A}_{p}\left( \omega ,\sigma ;%
\mathsf{P}_{\beta }\right) &\leq &\left( \sum_{\beta :x\in P_{\beta
}^{\prime }}\left\vert P_{\beta }\right\vert _{\omega }\right) \sup_{\beta
:x\in P_{\beta }^{\prime }}\left\vert P_{\beta }^{\prime }\right\vert
_{\sigma }^{p-1} \\
&\leq &C\left( \sum_{\beta :x\in P_{\beta }^{\prime }}\left\vert P_{\beta
}\right\vert _{\omega }\right) \rightarrow 0\text{ as }N\rightarrow \infty
\end{eqnarray*}%
since $\omega $ has no point mass at $x$. The argument in the general case
is technical, but involves no new ideas, and we leave it to the reader. We
thus conclude that 
\begin{equation*}
\mathcal{A}_{p}\left( \omega ,\sigma ;\mathsf{Q}_{0}\right) \leq C\left\vert
Q_{0}\right\vert ^{p},
\end{equation*}%
which is \eqref{Ap}. The case $2\leq p<\infty $ is proved in the same way
using that \eqref{asym} can be written%
\begin{equation*}
\mathcal{A}_{p^{\prime }}\left( \sigma ,\omega ;\mathsf{Q}_{\alpha }\right)
\leq C^{\prime }\left\vert \mathsf{Q}_{\alpha }\right\vert ^{\frac{p^{\prime
}}{2}}.
\end{equation*}%
{}
\end{proof}

\begin{lemma}
\label{dom}If $\left\{ T_{j}\right\} _{j=1}^{n}$ satisfies \eqref{j}, then%
\begin{equation*}
\mathcal{M}\nu (x)\leq C\sum_{j=1}^{n}(T_{j})_{\natural }\nu (x),\qquad x\in 
\mathbb{R}^{n}\text{, }\nu \geq 0\text{ a finite measure with compact support%
}.
\end{equation*}
\end{lemma}

\begin{proof}
We prove the case $n=1$, the general case being similar. Then with $T=T_{1}$
and $r>0$ we have%
\begin{eqnarray*}
&&\func{Re}\left( T_{r,\frac{r}{4},100r}\nu (x)-T_{r,4r,100r}\nu (x)\right)
\\
&=&\int \left\{ \zeta _{\frac{r}{4}}(y-x)-\zeta _{4r}(y-x)\right\} \func{Re}%
K(x,y)d\nu (y)\geq \frac{c}{r}\int_{\left[ x+\frac{r}{2},x+2r\right] }d\nu
(y).
\end{eqnarray*}%
Thus%
\begin{equation*}
T_{\natural }\nu (x)\geq \max \left\{ \left\vert T_{r,\frac{r}{4},100r}\nu
(x)\right\vert ,\left\vert T_{r,4r,100r}\nu (x)\right\vert \right\} \geq 
\frac{c}{r}\int_{\left[ x+\frac{r}{2},x+2r\right] }d\nu (y),
\end{equation*}%
and similarly%
\begin{equation*}
T_{\natural }\nu (x)\geq \frac{c}{r}\int_{\left[ x-2r,x-\frac{r}{2}\right]
}d\nu (y).
\end{equation*}%
It follows that%
\begin{eqnarray*}
\mathcal{M}\nu (x) &\leq &\sup_{r>0}\frac{1}{4r}\int_{\left[ x-2r,x+2r\right]
}d\nu (y) \\
&=&\sup_{r>0}\sum_{k=0}^{\infty }2^{-k}\frac{1}{2^{2-k}r}\int_{\left[
x-2^{1-k}r,x-2^{-1-k}r\right] \cup \left[ x+2^{-1-k}r,x+2^{1-k}r\right]
}d\nu (y) \\
&\leq &CT_{\natural }\nu (x).
\end{eqnarray*}
\end{proof}

Finally, we will use the following covering lemma of Besicovitch type for
multiples of dyadic cubes (the case of triples of dyadic cubes arises in (%
\ref{finover}) below).

\begin{lemma}
\label{Besicovitch}Let $M$ be an odd positive integer, and suppose that $%
\Phi $ is a collection of cubes $P$ with bounded diameters and having the
form $P=MQ$ where $Q$ is dyadic (a product of clopen dyadic intervals). If $%
\Phi ^{\ast }$ is the collection of \emph{maximal} cubes in $\Phi $, i.e. $%
P^{\ast }\in \Phi ^{\ast }$ provided there is no strictly larger $P$ in $%
\Phi $ that contains $P^{\ast }$, then the cubes in $\Phi ^{\ast }$ have
finite overlap at most $M^{n}$.
\end{lemma}

\begin{proof}
Let $Q_{0}=\left[ 0,1\right) ^{n}$ and assign labels $1,2,3 ,\dotsc, M^{n}$
to the dyadic subcubes of side length one of $MQ_{0}$. We say that the
subcube labeled $k$ is of type $k$, and we extend this definition by
translation and dilation to the subcubes of $MQ$ having side length that of $%
Q$. Now we simply observe that if $\left\{ P_{i}^{\ast }\right\} _{i}$ is a
set of cubes in $\Phi ^{\ast }$ containing the point $x$, then for a given $%
k $, there is at most one $P_{i}^{\ast }$ that contains $x$ in its subcube
of type $k$. The reason is that if $P_{j}^{\ast }$ is another such cube and $%
\ell \left( P_{j}^{\ast }\right) \leq \ell \left( P_{i}^{\ast }\right) $, we
must have $P_{j}^{\ast }\subset P_{i}^{\ast }$ (draw a picture in the plane
for example).
\end{proof}

\subsection{Preliminary Precaution}

Given a positive locally finite Borel measure $\mu $ on $\mathbb{R}^{n}$,
there exists a rotation such that all boundaries of rotated dyadic cubes
have $\mu $-measure zero (see \cite{MaMaNiOr} where they actually prove a
stronger assertion when $\mu $ has no point masses, but our conclusion is
obvious for a sum of point mass measures). We will assume that such a
rotation has been made so that all boundaries of rotated dyadic cubes have $%
\left( \omega +\sigma \right) $-measure zero, where $\omega $ and $\sigma $
are the positive Borel measures appearing in the theorems above (of course $%
\sigma $ doubling implies that $\sigma $ cannot contain any point masses,
but this argument works as well for general $\sigma $ as in the weak type
theorem). While this assumption is not essential for the proof, it relieves
the reader of having to consider the possibility that boundaries of dyadic
cubes have positive measure at each step of the argument below.

Recall also (see e.g. Theorem 2.18 in \cite{Ru}) that any positive locally
finite Borel measure on $\mathbb{R}^{n}$ is both inner and outer regular.

\section{The proof of Theorem \protect\ref{weaktwoweightHaar}: Weak-type
Inequalities}

We begin with the necessity of condition \eqref{Tsharpomega}:%
\begin{eqnarray*}
\int_{Q}T_{\natural }\left( \chi _{Q}f\sigma \right) \omega
&=&\int_{0}^{\infty }\min \left\{ \left\vert Q\right\vert _{\omega
},\left\vert \left\{ T_{\natural }\left( \chi _{Q}f\sigma \right) >\lambda
\right\} \right\vert _{\omega }\right\} d\lambda \\
&\leq &\left\{ \int_{0}^{A}+\int_{A}^{\infty }\right\} \min \left\{
\left\vert Q\right\vert _{\omega },C\lambda ^{-p}\int \left\vert
f\right\vert ^{p}d\sigma \right\} d\lambda \\
&\leq &A\left\vert Q\right\vert _{\omega }+CA^{1-p}\int \left\vert
f\right\vert ^{p}d\sigma \\
&=&(C+1)\left\vert Q\right\vert _{\omega }^{\frac{1}{p^{\prime }}}\left(
\int \left\vert f\right\vert ^{p}d\sigma \right) ^{\frac{1}{p}},
\end{eqnarray*}%
if we choose $A=\left( \frac{\int \left\vert f\right\vert ^{p}d\sigma }{%
\left\vert Q\right\vert _{\omega }}\right) ^{\frac{1}{p}}$.

\bigskip

Now we turn to proving \eqref{weak2weight}, assuming both (\ref{Tsharpomega}%
) and \eqref{weakM2weight}, and moreover that $f$ is bounded with compact
support. We will prove the quantitative estimate%
\begin{gather}
\left\Vert T_{\natural }f\sigma \right\Vert _{L^{p,\infty }(\omega )}\leq C%
\bigl\{\mathfrak{A}+\mathfrak{T}_{\ast }\bigr\}\left\Vert f\right\Vert
_{L^{p}(\sigma )}\,,  \label{quanweak} \\
\mathfrak{A}=\sup_{Q}\sup_{\lVert f\rVert _{L^{p}(\sigma )}=1}\sup_{\lambda
>0}\lambda \left\vert \left\{ \mathcal{M}\left( f\sigma \right) >\lambda
\right\} \right\vert _{\omega }^{\frac{1}{p}}\,,  \label{e.Afrak} \\
\mathfrak{T}_{\ast }=\sup_{\lVert f\rVert _{L^{p}(\sigma
)}=1}\sup_{Q}\,\lvert Q\rvert _{\omega }^{-1/p^{\prime }}\int_{Q}T_{\natural
}\left( \chi _{Q}f\sigma \right) (x)d\omega (x)\,.  \label{e.Tfrak'}
\end{gather}%
We should emphasize that the term \eqref{e.Afrak} is comparable to the two
weight $A_{p}$ condition \eqref{Ap}.

Standard considerations (\cite{Saw}, Section 2) show that it suffices to
prove the following good-$\lambda $ inequality: There is a positive constant 
$C$ so that for $\beta >0$ sufficiently small, and provided 
\begin{equation}
\sup_{0<\lambda <\Lambda }\lambda ^{p}\left\vert \left\{ x\in \mathbb{R}%
^{n}:T_{\natural }f\sigma \left( x\right) >\lambda \right\} \right\vert
_{\omega }<\infty \,,\qquad \Lambda <\infty \,,  \label{e.presume}
\end{equation}%
we have this inequality: 
\begin{eqnarray}
&&\left\vert \left\{ x\in \mathbb{R}^{n}:T_{\natural }f\sigma \left(
x\right) >2\lambda \text{ and }\mathcal{M}f\sigma \left( x\right) \leq \beta
\lambda \right\} \right\vert _{\omega }  \label{goodlam} \\
&&\ \ \ \ \ \leq C\beta \mathfrak{T}_{\ast }^{p}\left\vert \left\{ x\in 
\mathbb{R}^{n}:T_{\natural }f\sigma \left( x\right) >\lambda \right\}
\right\vert _{\omega }+C\beta ^{-p}\lambda ^{-p}\int \left\vert f\right\vert
^{p}d\sigma \,.  \notag
\end{eqnarray}%
Our presumption \eqref{e.presume} holds due to the $A_{p}$ condition %
\eqref{Ap} and the fact that%
\begin{equation*}
\left\{ x\in \mathbb{R}^{n}:T_{\natural }f\sigma \left( x\right) >\lambda
\right\} \subset B\left( 0,c\lambda ^{-\frac{1}{n}}\right) ,\qquad \lambda >0%
\text{ small},
\end{equation*}%
Hence it is enough to prove \eqref{goodlam}.

To prove \eqref{goodlam} we choose $\lambda =2^{k}$, and apply the
decomposition in \eqref{Whitney}. In this argument, we can take $k$ to be
fixed, so that we suppress its appearance as a superscript in this section.
(When we come to $L^{p}$ estimates, we will not have this luxury.)

Define 
\begin{equation*}
E_{j}=\left\{ x\in Q_{j}:T_{\natural }f\sigma \left( x\right) >2\lambda 
\text{ and }\mathcal{M}f\sigma \left( x\right) \leq \beta \lambda \right\} .
\end{equation*}%
Then for $x\in E_{j}$, we can apply Lemma~\ref{l.maxprin} to deduce 
\begin{equation}
T_{\natural }\left( \chi _{\left( 3Q_{j}\right) ^{c}}f\sigma \right) \left(
x\right) \leq \left( 1+C\beta \right) \lambda .  \label{e.maxprinc}
\end{equation}

If we take $\beta >0$ so small that $1+C\beta \leq \frac{3}{2}$, then (\ref%
{e.maxprinc}) implies that for $x\in E_{j}$ 
\begin{eqnarray*}
2\lambda &<&T_{\natural }f\sigma \left( x\right) \leq T_{\natural }\chi
_{3Q_{j}}f\sigma \left( x\right) +T_{\natural }\chi _{\left( 3Q_{j}\right)
^{c}}f\sigma \left( x\right) \\
&\leq &T_{\natural }\chi _{3Q_{j}^{k}}f\sigma \left( x\right) +\tfrac{3}{2}%
\lambda .
\end{eqnarray*}%
Integrating this inequality with respect to $\omega $ over $E_{j}$ we obtain 
\begin{equation}
\lambda \left\vert E_{j}\right\vert _{\omega }\leq 2\int_{E_{j}}\left(
T_{\natural }\chi _{3Q_{j}}f\sigma \right) \omega .  \label{ETsharp'}
\end{equation}

The disjoint cover condition in \eqref{Whitney} shows that the sets $E_{j}$
are disjoint, and this suggests we should sum their $\omega $-measures. We
split this sum into two parts, according to the size of $\lvert E_{j}\rvert
_{\omega }/\lvert 3Q_{j}\rvert _{\omega }$. The left-hand side of (\ref%
{goodlam}) satisfies%
\begin{eqnarray*}
\sum_{j}\left\vert E_{j}\right\vert _{\omega } &\leq &\beta
\sum_{j:\left\vert E_{j}\right\vert _{\omega }\leq \beta \left\vert
3Q_{j}\right\vert _{\omega }}\left\vert 3Q_{j}\right\vert _{\omega } \\
&&+\beta ^{-p}\sum_{j:\left\vert E_{j}\right\vert _{\omega }>\beta
\left\vert 3Q_{j}\right\vert _{\omega }}\left\vert E_{j}\right\vert _{\omega
}\left( \frac{2}{\lambda }\frac{1}{\left\vert 3Q_{j}\right\vert _{\omega }}%
\int_{E_{j}}\left( T_{\natural }\chi _{3Q_{j}}f\sigma \right) \omega \right)
^{p} \\
&=&I+I\!I.
\end{eqnarray*}

Now%
\begin{equation*}
I\leq \beta \sum_{j}\left\vert 3Q_{j}^{k}\right\vert _{\omega }\leq C\beta
\left\vert \Omega \right\vert _{\omega },
\end{equation*}%
by the finite overlap condition in \eqref{Whitney}. From \eqref{Tsharpomega}
with $Q=3Q_{j}$ we have 
\begin{eqnarray*}
I\!I &\leq &\left( \frac{2}{\beta \lambda }\right) ^{p}\sum_{j}\left\vert
E_{j}\right\vert _{\omega }\left( \frac{1}{\left\vert 3Q_{j}\right\vert
_{\omega }}\int_{E_{j}^{k}}\left( T_{\natural }\chi _{3Q_{j}}f\sigma \right)
\omega \right) ^{p} \\
&\leq &C\left( \frac{2}{\beta \lambda }\right) ^{p}\mathfrak{T}_{\ast
}^{p}\sum_{j}\left\vert E_{j}\right\vert _{\omega }\frac{1}{\left\vert
3Q_{j}\right\vert _{\omega }^{p}}\left\vert 3Q_{j}\right\vert _{\omega
}^{p-1}\int_{3Q_{j}}\left\vert f\right\vert ^{p}d\sigma \\
&\leq &C\left( \frac{2}{\beta \lambda }\right) ^{p}\mathfrak{T}_{\ast
}^{p}\int \left( \sum_{j}\chi _{3Q_{j}^{k}}\right) \left\vert f\right\vert
^{p}d\sigma \\
&\leq &C\left( \frac{2}{\beta \lambda }\right) ^{p}\mathfrak{T}_{\ast
}^{p}\int \left\vert f\right\vert ^{p}d\sigma ,
\end{eqnarray*}%
by the finite overlap condition in \eqref{Whitney} again. This completes the
proof of the good-$\lambda $ inequality \eqref{goodlam}.

The proof of assertion 2 regarding $T_{\flat }$ is similar. Assertion 3 was
discussed earlier and assertion 4 follows readily from assertion 2 and Lemma %
\ref{weakdom}.

\section{The proof of Theorem~\protect\ref{twoweightHaar}: Strong-type
Inequalities}

Since conditions \eqref{Tsharpsigma} and \eqref{Tsharpomega} are obviously
necessary for \eqref{2weight}, we turn to proving the weighted inequality (%
\ref{2weight}) for the strongly maximal singular integral $T_{\natural }$.

\subsection{The Quantitative Estimate}

In particular, we will prove 
\begin{align}
\left\Vert T_{\natural }f\sigma \right\Vert _{L^{p}(\omega )}& \leq C\bigl\{%
\mathfrak{M}+\gamma ^{2}\mathfrak{M}_{\ast }+\gamma ^{2}\mathfrak{T}+%
\mathfrak{T}_{\ast }\bigr\}\left\Vert f\right\Vert _{L^{p}(\sigma )}\,,
\label{quanstrong} \\
\mathfrak{M}& =\sup_{\lVert f\rVert _{L^{p}(\sigma )}=1}\lVert \mathcal{M}%
\left( f\sigma \right) \Vert _{L^{p}(\omega )}\,,  \label{e.Mfrak} \\
\mathfrak{M}_{\ast }& =\sup_{\lVert g\rVert _{L^{p^{\prime }}(\omega
)}=1}\lVert \mathcal{M}\left( g\omega \right) \Vert _{L^{p^{\prime }}(\sigma
)}\,,  \label{e.Mfrak*} \\
\mathfrak{T}& =\sup_{Q}\sup_{\lVert f\rVert _{L^{\infty }}\leq 1}\lvert
Q\rvert _{\sigma }^{-1/p}\lVert \chi _{Q}T_{\natural }\left( \chi
_{Q}f\sigma \right) \rVert _{L^{p}(\omega )}\,,  \label{e.TfrakStar} \\
\mathfrak{T}_{\ast }& =\sup_{\lVert f\rVert _{L^{p}(\sigma
)}=1}\sup_{Q}\,\lvert Q\rvert _{\omega }^{-1/p^{\prime }}\int_{Q}T_{\natural
}\left( \chi _{Q}f\sigma \right) (x)d\omega (x)\,,  \label{e.Tfrak}
\end{align}%
where $\gamma \geq 2$ is a doubling constant for the measure $\sigma $, see (%
\ref{doubling property}) below. Note that $\gamma $ appears only in
conjunction with $\mathfrak{T}$ and $\mathfrak{M}_{\ast }$. The norm
estimates on the maximal function \eqref{e.Mfrak} and \eqref{e.Mfrak*} are
equivalent to the testing conditions in \eqref{testmax} and its dual
formulation. The term $\mathfrak{T}_{\ast }$ also appeared in %
\eqref{e.Tfrak'}.

\subsection{The Initial Construction\label{initial construction}}

We suppose that both \eqref{Tsharpsigma} and \eqref{Tsharpomega} hold, i.e. (%
\ref{e.TfrakStar}) and \eqref{e.Tfrak} are finite, and that $f$ is bounded
with compact support on $\mathbb{R}^{n}$. Moreover, in the case \eqref{j}
holds, we see that \eqref{Tsharpsigma} (the finiteness of (\ref{e.TfrakStar}%
)) implies \eqref{testmax} by Lemma \ref{dom}, and so by Theorem \ref%
{sawyerthm} we may also assume that the maximal operator $\mathcal{M}$
satisfies the two weight norm inequality \eqref{M2weight}. It now follows
that $\int \left( T_{\natural }f\sigma \right) ^{p}\omega <\infty $ for $f$
bounded with compact support. Indeed, $T_{\natural }f\sigma \leq C\mathcal{M}%
f\sigma $ far away from the support of $f$, while $T_{\natural }f\sigma $ is
controlled by the finiteness of the testing condition \eqref{e.TfrakStar}
near the support of $f$.

Let $\{Q_{j}^{k}\}$ be the cubes as in \eqref{e.OmegaK} and \eqref{Whitney},
with the measure $\nu $ that appears in there being $\nu =f\sigma $. We will
use Lemma~\ref{l.maxprin} with this choice of $\nu $ as well. Now define an
`exceptional set' associated to $Q_{j}^{k}$ to be 
\begin{equation*}
E_{j}^{k}=Q_{j}^{k}\cap \left( \Omega _{k+1}\setminus \Omega _{k+2}\right) .
\end{equation*}%
See Figure~\ref{f.1}. One might anticipate the definition of the exceptional
set to be more simply $Q_{j}^{k}\cap \Omega _{k+1}$. We are guided to this
choice by the work on fractional integrals \cite{Saw2}. And indeed, the
choice of exceptional set above enters in a decisive way in the analysis of
the bad function at the end of the proof.

%%%%%%%%%%%%%%% Figure
\begin{figure}[ptb]
\begin{center}
\includegraphics{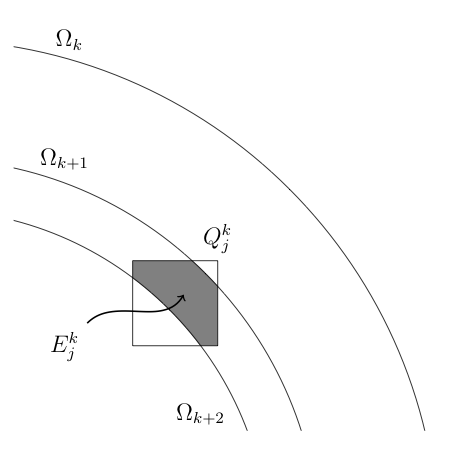}
\end{center}
\caption{The set $E_{j}^{k}\left( Q\right) $.}
\label{f.1}
\end{figure}
%
%%%%%%%%%%%%%%% Figure

We estimate the left side of \eqref{2weight} in terms of this family of
dyadic cubes $\left\{ Q_{j}^{k}\right\} _{k,j}$ by 
\begin{align}
\int \left( T_{\natural }f\sigma \right) ^{p}\omega (dx)& \leq \sum_{k\in 
\mathbb{Z}}(2^{k+2})^{p}\left\vert \Omega _{k+1}\setminus \Omega
_{k+2}\right\vert _{\omega }  \label{leftside} \\
& \leq \sum_{k,j}(2^{k+2})^{p}\left\vert E_{j}^{k}\right\vert _{\omega }. 
\notag
\end{align}

Choose a linearization $L$ of $T_{\natural }$ as in \eqref{defLa} so that
(recall $R(x)$ is the upper limit of truncation)%
\begin{eqnarray}
R(x) &\leq &\frac{1}{2}\ell (Q_{j}^{k}),\ \ \ \ \ x\in E_{j}^{k},
\label{loc} \\
\text{\textup{and}}\quad T_{\natural }\left( \chi _{3Q_{j}^{k}}f\sigma
\right) (x) &\leq &2L\left( \chi _{3Q_{j}^{k}}f\sigma \right) (x)+C\frac{1}{%
\left\vert 3Q_{j}^{k}\right\vert }\int_{3Q_{j}^{k}}\left\vert f\right\vert
\sigma ,\ \ \ \ \ x\in E_{j}^{k}.  \notag
\end{eqnarray}%
For $x\in E_{j}^{k}$, the maximum principle \eqref{maxprinc} yields 
\begin{eqnarray*}
T_{\natural }\chi _{3Q_{j}^{k}}f\sigma (x) &\geq &T_{\natural }f\sigma
(x)-T_{\natural }\chi _{(3Q_{j}^{k})^{c}}f\sigma (x) \\
&>&2^{k+1}-2^{k}-C\mathbf{P}\left( Q_{j}^{k},f\sigma \right) \\
&=&2^{k}-C\mathbf{P}\left( Q_{j}^{k},f\sigma \right) .
\end{eqnarray*}%
From \eqref{loc} we conclude that%
\begin{equation*}
L\chi _{3Q_{j}^{k}}f\sigma (x)\geq 2^{k-1}-C\mathbf{P}\left(
Q_{j}^{k},f\sigma \right) .
\end{equation*}%
Thus either $2^{k}\leq 4\inf_{E_{j}^{k}}L\chi _{3Q_{j}^{k}}f\sigma $ or $%
2^{k}\leq 4C\mathbf{P}\left( Q_{j}^{k},f\sigma \right) \leq 4CM\left(
Q_{j}^{k},f\sigma \right) $. So we obtain either%
\begin{equation}
\left\vert E_{j}^{k}\right\vert _{\omega }\leq C2^{-k}\int_{E_{j}^{k}}(L\chi
_{3Q_{j}^{k}}f\sigma )\omega (dx),  \label{ETsharp}
\end{equation}%
or%
\begin{equation}
\left\vert E_{j}^{k}\right\vert _{\omega }\leq C2^{-pk}\left\vert
E_{j}^{k}\right\vert _{\omega }M\left( Q_{j}^{k},f\sigma \right) ^{p}\leq
C2^{-pk}\int_{E_{j}^{k}}\left( \mathcal{M}f\sigma \right) ^{p}\omega (dx).
\label{EMsigma}
\end{equation}

Now consider the following decomposition of the set of indices $(k,j)$: 
\begin{eqnarray}
\mathbb{E} &=&\left\{ (k,j):\left\vert E_{j}^{k}\right\vert _{\omega }\leq
\beta \left\vert NQ_{j}^{k}\right\vert _{\omega }\right\} ,  \notag
\label{e.Gdef} \\
\mathbb{F} &=&\left\{ (k,j):\text{\eqref{EMsigma} holds}\right\} ,  \notag \\
\mathbb{G} &=&\left\{ (k,j):\left\vert E_{j}^{k}\right\vert _{\omega }>\beta
\left\vert NQ_{j}^{k}\right\vert _{\omega }\text{ and \eqref{ETsharp} holds}%
\right\} ,
\end{eqnarray}%
where $0<\beta <1$ will be chosen sufficiently small at the end of the
argument. (It will be of the order of $c^{p}$ for a small constant $c$.) By
the `bounded overlap' condition of \eqref{Whitney}, we have 
\begin{equation}
\sum_{j}\chi _{NQ_{j}^{k}}\leq C,\qquad k\in \mathbb{Z}.  \label{Nover}
\end{equation}%
We then have the corresponding decomposition:%
\begin{eqnarray}
\int \left( T_{\natural }f\sigma \right) ^{p}\omega &\leq &\left\{
\sum_{(k,j)\in \mathbb{E}}+\sum_{(k,j)\in \mathbb{F}}+\sum_{(k,j)\in \mathbb{%
G}}\right\} (2^{k+2})^{p}\left\vert E_{j}^{k}\right\vert _{\omega }
\label{e.T33} \\
&\leq &\beta \sum_{(k,j)\in \mathbb{E}}(2^{k+2})^{p}\left\vert
NQ_{j}^{k}\right\vert _{\omega }+C\sum_{(k,j)\in \mathbb{F}%
}\int_{E_{j}^{k}}\left( \mathcal{M}f\sigma \right) ^{p}\omega  \notag \\
&&{}+C\sum_{(k,j)\in \mathbb{G}}\left\vert E_{j}^{k}\right\vert _{\omega
}\left( \frac{1}{\beta \left\vert NQ_{j}^{k}\right\vert _{\omega }}%
\int_{E_{j}^{k}}\left( L\chi _{3Q_{j}^{k}}f\sigma \right) \omega \right) ^{p}
\notag \\
&=&J(1) +J(2) +J(3)  \notag \\
&\leq &C_{0}\left\{ \beta \int \left( T_{\natural }f\sigma \right)
^{p}\omega +\beta ^{-p}\int \lvert f\rvert ^{p}\sigma \right\} ,
\end{eqnarray}%
where $C_{0}\leq C\bigl\{\mathfrak{M}+\gamma ^{2}\mathfrak{M}_{\ast }+\gamma
^{2}\mathfrak{T}+\mathfrak{T}_{\ast }\bigr\}^{p}$. The last line is the
claim that we take up in the remainder of the proof. Once it is proved, note
that if we take $0<C_{0}\beta <\tfrac{1}{2}$ and use the fact that $\int
\left( T_{\natural }f\sigma \right) ^{p}\omega <\infty $ for $f$ bounded
with compact support, we have proved assertion (1) of Theorem \ref%
{twoweightHaar}, and in particular \eqref{quanstrong}.

\smallskip

The proof of the strong type inequality requires a complicated series of
decompositions of the dominating sums, which are illustrated for the
reader's convenience as a schematic tree in Figure~\ref{f.2}.

%%%%%%%%%%%%%%% Figure
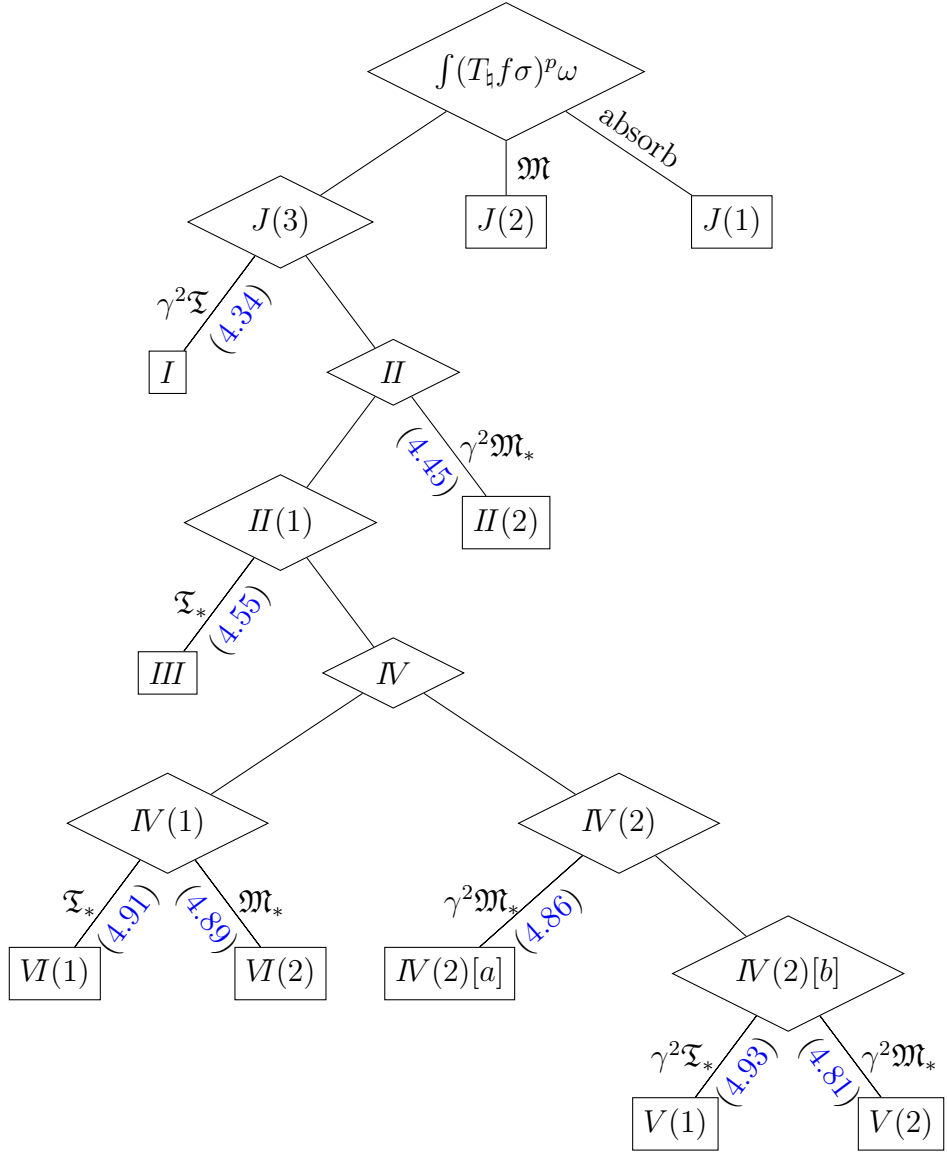
\begin{figure}[ptb]
\begin{center}
\begin{tikzpicture}
[level distance=20mm]
\node[shape aspect=2,diamond,draw]  { $ \int ( T_{\natural }f\sigma ) ^{p}\omega $ }
  child  {node[shape aspect=2,diamond,draw] {$ J( 3)$}
  	child {node[rectangle,draw] {$  I$} 
		  	   edge from parent node[left] {$ \gamma ^2  \mathfrak T$}
		  	   edge from parent node[below,sloped] {\eqref{e.I<}}
	       } 
	  child[missing] {node {}}
	child {node[shape aspect=2,diamond,draw] {$  I\!I$}
		child {node[shape aspect=2,diamond,draw] {$  I\!I(1)$}
			child  {node[rectangle,draw] {$  I\!I\!I$}
			  	   edge from parent node[left] {$ \mathfrak T _{ \ast  }$}
			  	    edge from parent node[below,sloped] {\eqref{e.III<}} 
				}
			  child[missing] {node {}}
			child  {node[shape aspect=2,diamond,draw] (iv) {$  I\!V$}
				child {node[shape aspect=2,diamond,draw] {$  I\!V(1)$}
				child {node[rectangle,draw] {$  V\!\!I(1)$}
					  	   edge from parent node[left] {$ \mathfrak T _{\ast }$}
					  	    edge from parent node[below,sloped] {\eqref{e.VI1<}}
					  	   }
				child[missing] {node {}}
				child {node[rectangle,draw] {$  V\!\!I(2)$}
					  	   edge from parent node[right] {$ \mathfrak M _{\ast }$}
					  	    edge from parent node[below,sloped] {\eqref{e.VI2<}}
					 }
					}
				child[missing] {node {}}   % to get some space in the diagram
				child[missing] {node {}}   % to get some space in the diagram
				child[missing] {node {}}
				child {node[shape aspect=2,diamond,draw] {$I\!V(2)$}
					child {node[rectangle,draw] {$I\!V(2)[a]$}
					  	   edge from parent node[left] {$\gamma ^2 \mathfrak M _{\ast} $}
					  	    edge from parent node[below,sloped] {\eqref{e.IVts2a<}}
						}
					child[missing] {node {}}
					child[missing] {node {}}
					child {node[shape aspect=2,diamond,draw] {$I\!V(2)[b]$}
					  	 child {node[rectangle,draw] {$V(1)$}  
					  	  edge from parent node[left] {$ \gamma ^2 \mathfrak T _{\ast} $}
					  	    edge from parent node[below,sloped] {\eqref{e.V1<}}}
					  	  child[missing] {node {}}
					  	  child {node[rectangle,draw] {$V(2)$}  
					  	  edge from parent node[right] {$\gamma ^2  \mathfrak M _{\ast } $}
					  	    edge from parent node[below,sloped] {\eqref{e.V2<}}
						} } 
					}
				}
			}
		child[missing] {node {}}
		child {node[rectangle,draw] {$  I\!I(2)$}
			edge from parent node[right] {$  \gamma ^2 \mathfrak M _{\ast }$} 
						 node[sloped,below] {\eqref{e.II2<}}
			} 
			}
			}
        child[missing] {node {}}
        child  {node[rectangle,draw] {$ J( 2)$}
  	   edge from parent node[right] {$ \mathfrak M$}
  	  }
	child[missing] {node {}}
	child {node[rectangle,draw] {$ J( 1)$}
  	 edge from parent node[above,sloped] {absorb }
	 	};    
\end{tikzpicture}
\end{center}
\caption{This is a schematic tree of how the integral $\protect\int \left(
T_{\natural }f\protect\sigma \right) ^{p}\protect\omega $ has been, and will
continue to be, decomposed. We have suppressed superscripts, subscripts and
sums in the tree. Terms in diamonds are further decomposed, while terms in
rectangles are final estimates. The edges leading into rectangles are
labelled by $\mathfrak{M}$, $\mathfrak{M}_{\ast }$, $\mathfrak{I}$ or $%
\mathfrak{I}_{\ast }$ whose finiteness is used to control that term. Those
terms controlled by the doubling constant $\protect\gamma $ are also
indicated. Equation references are to where the final estimates on the term
is obtained. The word `absorb' leading into $J(1) $ indicates that this term
is a small multiple of $\protect\int \left( T_{\natural }f\protect\sigma %
\right) ^{p}\protect\omega $ and can be absorbed into the left-hand side of
the inequality. As most of the terms involve the maximal theorem (\protect
\ref{maxthm}), we do not indicate its use in the schematic tree.}
\label{f.2}
\end{figure}
%
%%%%%%%%%%%%%%% Figure

\subsection{Two Easy Estimates}

Note that the first term $J(1) $ in \eqref{e.T33} satisfies 
\begin{equation*}
J(1) =\beta \sum_{(k,j)\in \mathbb{E}}(2^{k+2})^{p}\left\vert
NQ_{j}^{k}\right\vert _{\omega }\leq C\beta \int \left( T_{\natural }f\sigma
\right) ^{p}\omega ,
\end{equation*}%
by the finite overlap condition \eqref{Nover}. The second term $J\left(
2\right) $ is dominated by%
\begin{equation*}
C\sum_{(k,j)\in \mathbb{F}}\int_{E_{j}^{k}}\left( \mathcal{M}f\sigma \right)
^{p}\omega \leq C\mathfrak{M}^{p}\left\Vert f\right\Vert _{L^{p}(\sigma
)}^{p}\,,
\end{equation*}%
by our assumption \eqref{M2weight}. It is useful to note that this is the 
\emph{only} time in the proof that we use the maximal function inequality (%
\ref{M2weight}) - from now on we use the \emph{dual} maximal function
inequality \eqref{M2weightdual}.

\begin{remark}
In the arguments below we can use Theorem 2 of \cite{Saw2}\ to replace the
dual maximal function assumption $\mathfrak{M}_{\ast }<\infty $ with two
assumptions, namely a `Poisson two weight $A_{p}$ condition' and the
analogue of the dual pivotal condition of Nazarov, Treil and Volberg \cite%
{NTV3}. The Poisson two weight $A_{p}$ condition is in fact necessary for
the two weight inequality, but the pivotal conditions are \emph{not}
necessary for the Hilbert transform two weight inequality (\cite{LaSaUr}).
On the other hand, the assumption $\mathfrak{M}<\infty $ cannot be weakened
here, reflecting that our method requires the maximum principle in Lemma \ref%
{l.maxprin}.
\end{remark}

It is the third term $J(3) $ that is the most involved, see Figure~\ref{f.2}%
. The remainder of the proof is taken up with the proof of 
\begin{equation}
\sum_{(k,j)\in \mathbb{G}}R_{j}^{k}\left\vert \int_{E_{j}^{k}}\left( L\chi
_{3Q_{j}^{k}}f\sigma \right) \omega \right\vert ^{p}\leq C\{\gamma ^{2p}%
\mathfrak{M}_{\ast }^{p}+\gamma ^{2p}\mathfrak{T}^{p}+\mathfrak{T}_{\ast
}^{p}\}\left\Vert f\right\Vert _{L^{p}\sigma (dx)}^{p}\,,
\label{e.remainder}
\end{equation}%
where%
\begin{equation}
R_{j}^{k}=\frac{\left\vert E_{j}^{k}\right\vert _{\omega }}{\left\vert
NQ_{j}^{k}\right\vert _{\omega }^{p}}\,.  \label{Rkj'}
\end{equation}%
Once this is done, the proof of \eqref{e.T33} is complete, and the proof of
assertion (1) is finished.

\subsection{The Calder\'on-Zygmund Decompositions}

\label{s.czd}

To carry out this proof, we implement Calder\'{o}n-Zygmund Decompositions
relative to the measure $\sigma $. These Decompositions will be done at 
\emph{all heights simultaneously}. We will use the shifted dyadic grids, see %
\eqref{e.shifted}. Suppose that $\gamma \geq 2$ is a doubling constant for
the measure $\sigma $:%
\begin{equation}
\left\vert 3Q\right\vert _{\sigma }\leq \gamma \left\vert Q\right\vert
_{\sigma },\ \ \ \ \ \text{all cubes }Q.  \label{doubling condition}
\end{equation}%
For $\alpha \in \{0,\tfrac{1}{3},\tfrac{2}{3}\}^{n}$, let%
\begin{equation*}
\mathcal{M}_{\sigma }^{\alpha }f\left( x\right) =\sup_{x\in Q\in \mathcal{D}%
^{\alpha }}\frac{1}{\left\vert Q\right\vert _{\sigma }}\int_{Q}\left\vert
f\right\vert d\sigma ,
\end{equation*}%
\begin{equation}
\Gamma _{t}^{\alpha }=\left\{ x\in \mathbb{R}:\mathcal{M}_{\sigma }^{\alpha
}f(x)>\gamma ^{t}\right\} =\bigcup_{s}G_{s}^{\alpha ,t},  \label{e.GammaTDef}
\end{equation}%
where $\left\{ G_{s}^{\alpha ,t}\right\} _{\left( t,s\right) \in \mathbb{L}%
^{\alpha }}$ are the maximal $\mathcal{D}^{\alpha }$ cubes in $\Gamma
_{t}^{\alpha }$, and $\mathbb{L}^{\alpha }$ is the set of pairs we use to
label the cubes. This implies that we have the nested property: If $%
G_{s}^{\alpha ,t}\subsetneqq G_{s^{\prime }}^{\alpha ,t^{\prime }}$ then $%
t>t^{\prime }$. Moreover, if $t>t^{\prime }$ there is some $s^{\prime }$
with $G_{s}^{\alpha ,t}\subset G_{s^{\prime }}^{\alpha ,t^{\prime }}$. These
are the cubes used to make a Calder\'{o}n-Zygmund Decomposition at height $%
\gamma ^{t}$ for the grid $\mathcal{D}^{\alpha }$ with respect to the
measure $\sigma $. We will refer to the cubes $\left\{ G_{s}^{\alpha
,t}\right\} _{\left( t,s\right) \in \mathbb{L}^{\alpha }}$ as \textbf{%
principal cubes}.

Of course we have from the maximal inequality in \eqref{maxthm} 
\begin{equation}
\sum_{\left( t,s\right) \in \mathbb{L}^{\alpha }}\gamma ^{pt}\lvert
G_{s}^{\alpha ,t}\rvert _{\sigma }\leq C\left\Vert f\right\Vert
_{L^{p}(\sigma )}^{p}\,.  \label{e.MGst}
\end{equation}%
The point of these next several definitions is to associate to each dyadic
cube $Q$, a good shifted dyadic grid, and an appropriate height, at which we
will build our Calder\'{o}n-Zygmund Decomposition.

It is now that we will use the following consequence of the doubling
condition \eqref{doubling condition} for the measure $\sigma $:%
\begin{equation}
\left\vert P\left( G\right) \right\vert _{\sigma }\leq \gamma \left\vert
G\right\vert _{\sigma },\ \ \ \ \ G\in \mathcal{D}^{\alpha }.
\label{doubling property}
\end{equation}%
The average $\frac{1}{\left\vert G_{s}^{\alpha ,t}\right\vert _{\sigma }}%
\int_{G_{s}^{\alpha ,t}}\left\vert f\right\vert d\sigma $ is thus at most $%
\gamma ^{t+1}$ by \eqref{doubling property} and the maximality of the cubes
in \eqref{e.GammaTDef}:%
\begin{equation}
\gamma ^{t}<\frac{1}{\left\vert G_{s}^{\alpha ,t}\right\vert _{\sigma }}%
\int_{G_{s}^{\alpha ,t1}}\left\vert f\right\vert d\sigma \leq \frac{%
\left\vert P\left( G_{s}^{\alpha ,t}\right) \right\vert _{\sigma }}{%
\left\vert G_{s}^{\alpha ,t}\right\vert _{\sigma }}\frac{1}{\left\vert
P\left( G_{s}^{\alpha ,t}\right) \right\vert _{\sigma }}\int_{P\left(
G_{s}^{\alpha ,t}\right) }\left\vert f\right\vert d\sigma \leq \gamma \gamma
^{t}=\gamma ^{t+1}.  \label{average bound}
\end{equation}

\begin{description}
\item[Select a shifted grid] Let $\vec{\alpha}\;:\;\mathcal{D}%
\longrightarrow \{0,\tfrac{1}{3},\tfrac{2}{3}\}^{n}$ be a map so that for $%
Q\in \mathcal{D}$, there is a $\widehat{Q}\in \mathcal{D}^{\vec{\alpha}(Q)}$
so that $3Q\subset \widehat{Q}$ and $\lvert \widehat{Q}\rvert \leq C\lvert
Q\rvert $. Here, $C$ is an appropriate constant depending only on dimension.
Thus, $\vec{\alpha}(Q)$ picks a `good' shifted dyadic grid for $Q$. Moreover
we will assume that $\widehat{Q}$ is the smallest such cube. Note that we
are discarding the extra requirement that $3Q\subset \frac{9}{10}\widehat{Q}$
since this property will not be used. Also we have 
\begin{equation}
\widehat{Q}\subset MQ,  \label{QhatContained}
\end{equation}%
for some positive dimensional constant $M$. The cubes $\widehat{Q_{j}^{k}}$
will play a critical role below. See Figure~\ref{f.4}

\item[Select a principal cube] Define $\mathcal{A}(Q)$ to be the smallest
cube from the collection $\{G_{s}^{\vec{\alpha}(Q),t}\mid \left( t,s\right)
\in \mathbb{L}^{\alpha }\}$ that contains $3Q$; $\mathcal{A}(Q)$ is uniquely
determined by $Q$ and the choice of function $\vec{\alpha}$. 
% Fix $(t,s)\in \mathbb{L}^{\alpha }$ and let $\mathcal{A}(Q_{j}^{k})=G_{s}^{\alpha ,t}$.
Define 
\begin{equation}
\mathbb{H}_{s}^{\alpha ,t}=\left\{ (k,j):\mathcal{A}(Q_{j}^{k})=G_{s}^{%
\alpha ,t}\right\} \,,\qquad (s,t)\in \mathbb{L}^{\alpha }\,.  \label{e.HST}
\end{equation}%
This is an important definition for us. The combinatorial structure this
places on the corresponding cubes is essential for this proof to work. Note
that $3Q_{j}^{k}\subset \widehat{Q_{j}^{k}}\subset \mathcal{A}\left(
Q_{j}^{k}\right) $.

\item[Parents] For any of the shifted dyadic grids $\mathcal{D}^{\alpha }$,
a $Q\in \mathcal{D}^{\alpha }$ has a unique parent denoted as $P(Q)$, the
smallest member of $\mathcal{D}^{\alpha }$ that strictly contains $Q$. We
suppress the dependence upon $\alpha $ here.

\item[Indices] Let 
\begin{equation}
\mathcal{K}_{s}^{\alpha ,t}=\left\{ r\mid G_{r}^{\alpha ,t+1}\subset
G_{s}^{\alpha ,t}\right\} \,.  \label{e.Kst}
\end{equation}%
We use a calligraphic font $\mathcal{K}$ for sets of indices related to the
grid $\left\{ G_{s}^{\alpha ,t}\right\} $, and a blackboard font $\mathbb{H}$
for sets of indices related to the grid $\left\{ Q_{j}^{k}\right\} $.

\item[The good and bad functions] Let $A_{G_{r}^{\alpha ,t+1}}=\frac{1}{%
\left\vert G_{r}^{\alpha ,t+1}\right\vert _{\sigma }}\int_{G_{r}^{\alpha
,t+1}}f\sigma $ be the $\sigma $-average of $f$ on $G_{r}^{\alpha ,t+1}$.
Define functions $g_{s}^{\alpha ,t}$ and $h_{s}^{\alpha ,t}$ satisfying $%
f=g_{s}^{\alpha ,t}+h_{s}^{\alpha ,t}$ on $G_{s}^{\alpha ,t}$ by 
\begin{align}
g_{s}^{\alpha ,t}(x)& =%
\begin{cases}
A_{G_{r}^{\alpha ,t+1}} & x\in G_{r}^{\alpha ,t+1}\text{ with }r\in \mathcal{%
K}_{s}^{\alpha ,t} \\ 
f(x) & x\in G_{s}^{\alpha ,t}\setminus \bigcup \left\{ G_{r}^{\alpha
,t+1}:r\in \mathcal{K}_{s}^{\alpha ,t}\right\}%
\end{cases}%
,  \label{e.hst} \\
h_{s}^{\alpha ,t}(x)& =%
\begin{cases}
f(x)-A_{G_{r}^{\alpha ,t+1}} & x\in G_{r}^{\alpha ,t+1}\text{ with }r\in 
\mathcal{K}_{s}^{\alpha ,t} \\ 
0 & x\in G_{s}^{\alpha ,t}\setminus \bigcup \left\{ G_{r}^{\alpha ,t+1}:r\in 
\mathcal{K}_{s}^{\alpha ,t}\right\}%
\end{cases}%
.
\end{align}%
We extend both $g_{s}^{\alpha ,t}$ and $h_{s}^{\alpha ,t}$ to all of $%
\mathbb{R}^{n}$ by defining them to vanish outside $G_{s}^{\alpha ,t}$.
\end{description}

Now $\left\vert A_{G_{r}^{\alpha ,t+1}}\right\vert \leq \gamma ^{t+2}$ by (%
\ref{average bound}). Thus Lebesgue's differentiation theorem shows that
(any of the standard proofs can be adapted to the dyadic setting for
positive locally finite Borel measures on $\mathbb{R}^{n}$) 
\begin{equation}
\left\vert g_{s}^{\alpha ,t}(x)\right\vert \leq \gamma ^{t+2}<\frac{\gamma
^{2}}{\left\vert G_{s}^{\alpha ,t}\right\vert _{\sigma }}\int_{G_{s}^{\alpha
,t}}\left\vert f\right\vert \sigma ,\qquad \sigma \text{-a.e. }x\in
G_{s}^{\alpha ,t},\ \left( t,s\right) \in \mathbb{L}^{\alpha }.  \label{Leb}
\end{equation}%
That is, $g_{s}^{\alpha ,t}$ is the `good' function and $h_{s}^{\alpha ,t}$
is the `bad' function. %See Figure~\ref{f.3}.

%%%%%%%%%%%%%%%% Figure
%\begin{figure}[ptb]\begin{center}
%\includegraphics{goodBadCropped}\end{center}
%\caption{The relative positions for the cubes $Q_{j}^{k}$, $\mathcal{A}%
%(Q_{j}^{k})=G_{s}^{\protect\alpha ,t}$, and cubes $P(G_{r}^{\protect\alpha %
%,t+1})$ for $r\in \mathbb{K}_{s}^{\protect\alpha ,t}$. Note that $g_{s}^{%
%\protect\alpha ,t}$ is supported on $G_{s}^{\protect\alpha ,t}$, and has $%
%L^{\infty }$ norm at most $\protect\gamma ^{t+2}$, and that the function $%
%h_{s}^{\protect\alpha ,t}$ is supported on the cubes $G_{r}^{\protect\alpha %
%,t+1}$, and has integral zero with respect to $\protect\sigma $-measure on
%each such cube.}\label{f.3}%
%\end{figure}%
%%%%%%%%%%%%%%%% Figure
%%%%%%%%%%%%%%%% Figure

We can now refine the final sum on the left side of \eqref{e.remainder}
according to the decomposition of $\mathcal{M}_{\sigma }^{\alpha }f$. We
carry this out in three steps. In the first step, we fix an $\alpha \in \{0,%
\tfrac{1}{3},\tfrac{2}{3}\}^{n}$, and for the remainder of the proof, we
only consider $Q_{j}^{k}$ for which $\vec{\alpha}(Q_{j}^{k})=\alpha $.
Namely, we will modify the important definition of $\mathbb{G}$ in %
\eqref{e.Gdef} to 
\begin{equation}
\mathbb{G}^{\alpha }=\left\{ (k,j):\vec{\alpha}(Q_{j}^{k})=\alpha \,,\
\left\vert E_{j}^{k}\right\vert _{\omega }>\beta \left\vert
NQ_{j}^{k}\right\vert _{\omega }\text{ and \eqref{ETsharp} holds}\right\} ,
\label{e.GdefAlpha}
\end{equation}%
In the second step, we partition the indices $(k,j)$ into the sets $\mathbb{H%
}_{s}^{\alpha ,t}$ in \eqref{e.HST} for $\left( t,s\right) \in \mathbb{L}%
^{\alpha }$. In the third step, for $(k,j)\in \mathbb{H}_{s}^{\alpha ,t}$,
we split $f$ into the corresponding good and bad parts. This yields the
decomposition 
\begin{gather}  \label{e.yyI}
\sum_{(k,j)\in \mathbb{G}^{\alpha }}R_{j}^{k}\left\vert
\int_{E_{j}^{k}}\left( L\chi _{3Q_{j}^{k}}f\sigma \right) \omega \right\vert
^{p}\leq C\left( I+I\!I\right) , \\
I=\sum_{\left( t,s\right) \in \mathbb{L}^{\alpha }}I_{s}^{t}\,,\qquad
I\!I=\sum_{\left( t,s\right) \in \mathbb{L}^{\alpha }}I\!I_{s}^{t}
\label{e.Idef} \\
I_{s}^{t}=\sum_{(k,j)\in \mathbb{I}_{s}^{\alpha ,t}}R_{j}^{k}\left\vert
\int_{E_{j}^{k}}\left( L\chi _{3Q_{j}^{k}}g_{s}^{\alpha ,t}\sigma \right)
\omega \right\vert ^{p} \\
I\!I_{s}^{t}=\sum_{(k,j)\in \mathbb{I}_{s}^{\alpha ,t}}R_{j}^{k}\left\vert
\int_{E_{j}^{k}}\left( L\chi _{3Q_{j}^{k}}h_{s}^{\alpha ,t}\sigma \right)
\omega \right\vert ^{p}  \label{e.IIst} \\
\mathbb{I}_{s}^{\alpha ,t}=\mathbb{G}\cap \mathbb{H}_{s}^{\alpha ,t}
\end{gather}%
Recall the definition of $R_{j}^{k}$ in \eqref{Rkj'}. In the definitions of $%
I$, $I_{s}^{t}$ and $I\!I$, $I\!I_{s}^{t}$ we will suppress the dependence
on $\alpha \in \{0,\tfrac{1}{3},\tfrac{2}{3}\}^{n}$. The same will be done
for the subsequent decompositions of the (difficult) term $I\!I$, although
we usually retain the superscript $\alpha $ in the quantities arising in the
estimates. In particular, we emphasize that the combinatorial properties of
the cubes associated with $\mathbb{I}_{s}^{\alpha ,t}$ are essential to
completing this proof.

Term $I$ requires only the forward testing condition \eqref{Tsharpsigma} and
the maximal theorem \eqref{maxthm}, while term $I\!I$ requires only the dual
testing condition \eqref{Tsharpomega}, along with the dual maximal function
inequality \eqref{M2weightdual} and the maximal theorem \eqref{maxthm}. The
reader is again directed to Figure \ref{f.2} for a map of the various
decompositions of the terms and the conditions used to control them.

\subsection{The Analysis of the good function}

We claim that 
\begin{equation}
I\leq C\gamma ^{2p}\mathfrak{T}^{p}\lVert f\rVert _{L^{p}(\sigma )}^{p}\,.
\label{e.I<}
\end{equation}

\begin{proof}
We use boundedness of the `good' function $g_{s}^{\alpha ,t}$, as defined in %
\eqref{e.hst}, the testing condition \eqref{Tsharpsigma} for $T_{\natural }$%
, see also \eqref{e.TfrakStar}, and finally the universal maximal function
bound \eqref{maxthm} with $\mu =\omega $. Here are the details. For $x\in
E_{j}^{k}$, \eqref{loc} implies that $L\chi _{3Q_{j}^{k}}g_{s}^{\alpha
,t}\sigma \left( x\right) =Lg_{s}^{\alpha ,t}\sigma \left( x\right) $ and so 
\begin{eqnarray*}
I &=&\sum_{\left( t,s\right) \in \mathbb{L}^{\alpha
}}I_{s}^{t}=C\sum_{\left( t,s\right) \in \mathbb{L}^{\alpha }}\sum_{(k,j)\in 
\mathbb{G}^{\alpha }\cap \mathbb{H}_{s}^{\alpha ,t}}R_{j}^{k}\left\vert
\int_{E_{j}^{k}}\left( Lg_{s}^{\alpha ,t}\sigma \right) \omega \right\vert
^{p} \\
&\leq &C\sum_{\left( t,s\right) \in \mathbb{L}^{\alpha }}\int \left\vert 
\mathcal{M}_{\omega }^{dy}\left( \chi _{G_{s}^{\alpha ,t}}Lg_{s}^{\alpha
,t}\sigma \right) \right\vert ^{p}\omega \\
&\leq &C\sum_{\left( t,s\right) \in \mathbb{L}^{\alpha }}\int_{G_{s}^{\alpha
,t}}\left\vert Lg_{s}^{\alpha ,t}\sigma \right\vert ^{p}\omega \\
&\leq &C\gamma ^{2p}\sum_{\left( t,s\right) \in \mathbb{L}^{\alpha }}\gamma
^{pt}\int_{G_{s}^{\alpha ,t}}\left( T_{\natural }\frac{g_{s}^{\alpha ,t}}{%
\gamma ^{t+2}}\sigma \right) ^{p}\omega \\
&\leq &C\gamma ^{2p}\mathfrak{T}^{p}\sum_{\left( t,s\right) \in \mathbb{L}%
^{\alpha }}\gamma ^{pt}\left\vert G_{s}^{\alpha ,t}\right\vert _{\sigma },
\end{eqnarray*}%
where we have used \eqref{Leb} and \eqref{Tsharpsigma} with $g=\frac{%
g_{s}^{\alpha ,t}}{\gamma ^{t+2}}$ in the final inequality. This last sum is
controlled by \eqref{e.MGst}, and completes the proof of \eqref{e.I<}.
\end{proof}

\subsection{The Analysis of the Bad Function: Part 1}

It remains to estimate term $I\!I$, as in \eqref{e.IIst}, but this is in
fact the harder term. Recall the definition of $\mathcal{K}_{s}^{\alpha ,t}$
in \eqref{e.Kst}. We now write%
\begin{equation}
h_{s}^{\alpha ,t}=\sum_{r\in \mathcal{K}_{s}^{\alpha ,t}}\left[
f-A_{G_{r}^{\alpha ,t+1}}\right] \chi _{G_{r}^{\alpha ,t+1}}\equiv
\sum_{r\in \mathcal{K}_{s}^{\alpha ,t}}b_{r},  \label{e.brDef}
\end{equation}%
where the `bad' functions $b_{r}$ are supported in the cube $G_{r}^{\alpha
,t+1}$ and have $\sigma $-mean zero, $\int_{G_{r}^{\alpha ,t+1}}b_{r}\sigma
=0$. To take advantage of this, we will pass to the dual $L^{\ast }$ below.

But first we must address the fact that the triples of the $\mathcal{D}%
^{\alpha }$ cubes $G_{r}^{\alpha ,t+1}$ do not form a grid. Fix $\left(
t,s\right) \in \mathbb{L}^{\alpha }$ and let 
\begin{equation}
\mathfrak{C}_{s}^{\alpha ,t}=\left\{ 3G_{r}^{\alpha ,t+1}:r\in \mathcal{K}%
_{s}^{\alpha ,t}\right\}  \label{e.K...}
\end{equation}%
be the collection of triples of the $\mathcal{D}^{\alpha }$ cubes $%
G_{r}^{\alpha ,t+1}$ with $r\in \mathcal{K}_{s}^{\alpha ,t}$. We select the 
\emph{maximal} triples 
\begin{equation}
\left\{ 3G_{r_{\ell }}^{t+1}\right\} _{\ell \in \mathcal{L}_{s}^{\alpha
,t}}\equiv \left\{ T_{\ell }\right\} _{\ell \in \mathcal{L}_{s}^{\alpha ,t}}
\label{maximal triples}
\end{equation}%
from the collection $\mathfrak{C}_{s}^{\alpha ,t}$, and assign to each $r\in 
\mathcal{K}_{s}^{\alpha ,t}$ the maximal triple $T_{\ell }=T_{\ell \left(
r\right) }$ containing $3G_{r}^{\alpha ,t+1}$ with least $\ell $. Note that $%
T_{\ell \left( r\right) }$ extends outside $G_{s}^{\alpha ,t}$ if $%
G_{r}^{\alpha ,t+1}$ and $G_{s}^{\alpha ,t}$ share a face. By Lemma \ref%
{Besicovitch} applied to $\mathcal{D}^{\alpha }$ the maximal triples $%
\left\{ T_{\ell }\right\} _{\ell \in \mathcal{L}_{s}^{\alpha ,t}}$ have
finite overlap $3^{n}$, and this will prove crucial in \eqref{e.III<}, %
\eqref{e.V1<} and \eqref{finover} below.

We will pass to the dual of the linearization.%
\begin{eqnarray}
\int_{E_{j}^{k}}\left( Lh_{s}^{\alpha ,t}\sigma \right) \omega &=&\sum_{r\in 
\mathcal{K}_{s}^{\alpha ,t}}\int_{E_{j}^{k}}\left( Lb_{r}\sigma \right)
\omega  \label{e.ML1} \\
&=&\sum_{r\in \mathcal{K}_{s}^{\alpha ,t}}\int_{G_{r}^{\alpha ,t+1}\cap
3Q_{j}^{k}}\left( L^{\ast }\chi _{E_{j}^{k}}\omega \right) b_{r}\sigma 
\notag
\end{eqnarray}%
Note that \eqref{loc} implies $L^{\ast }\nu $ is supported in $3Q_{j}^{k}$
if $\nu $ is supported in $E_{j}^{k}$, explaining the range of integration
above. Continuing, we have for fixed $\left( k,j\right) \in \mathbb{I}%
_{s}^{\alpha ,t}$,%
\begin{eqnarray}
\left\vert \eqref{e.ML1}\right\vert &\leq &\left\vert \sum_{r\in \mathcal{K}%
_{s}^{\alpha ,t}}\int_{G_{r}^{\alpha ,t+1}\cap 3Q_{j}^{k}}\left( L^{\ast
}\chi _{E_{j}^{k}\cap T_{\ell \left( r\right) }}\omega \right) b_{r}\sigma
\right\vert  \label{e.3rd} \\
&&{}+C\sum_{r\in \mathcal{K}_{s}^{\alpha ,t}}\mathbf{P}\left( G_{r}^{\alpha
,t+1},\chi _{E_{j}^{k}\setminus 3G_{r}^{\alpha ,t+1}}\omega \right)
\int_{G_{r}^{\alpha ,t+1}}\left\vert f\right\vert \sigma .  \notag
\end{eqnarray}%
To see the above inequality, note that for $r\in \mathcal{K}_{s}^{\alpha ,t}$
we are splitting the set $E_{j}^{k}$ into $E_{j}^{k}\cap T_{\ell \left(
r\right) }$ and $E_{j}^{k}\setminus T_{\ell \left( r\right) }$. On the
latter set, the hypotheses of Lemma~\ref{constant} are in force, namely the
set $E_{j}^{k}\setminus T_{\ell \left( r\right) }$ does not intersect $%
3G_{r}^{\alpha ,t+1}$, whence we have an estimate on the $\delta $-H\"{o}%
lder modulus of continuity of $L^{\ast }\chi _{E_{j}^{k}\setminus T_{\ell
\left( r\right) }}\omega $. Combine this with the fact that $b_{r}$ has $%
\sigma $-mean zero on $G_{r}^{\alpha ,t+1}$ to derive the estimate below, in
which $y_{r}^{t+1}$ is the center of the cube $G_{r}^{\alpha ,t+1}$. 
\begin{eqnarray}
&&\left\vert \int_{G_{r}^{\alpha ,t+1}}\left( L^{\ast }\chi
_{E_{j}^{k}\setminus T_{\ell \left( r\right) }}\omega \right) b_{r}\sigma
\right\vert  \notag  \label{e.1st<} \\
&=&\left\vert \int_{G_{r}^{\alpha ,t+1}}\left( L^{\ast }\chi
_{E_{j}^{k}\setminus T_{\ell \left( r\right) }}\omega (y)-L^{\ast }\chi
_{E_{j}^{k}\setminus T_{\ell \left( r\right) }}\omega (y_{r}^{t+1})\right)
\left( b_{r}\sigma \right) \right\vert  \notag \\
&\leq &\int_{G_{r}^{\alpha ,t+1}\cap 3Q_{j}^{k}}C\mathbf{P}\left(
G_{r}^{\alpha ,t+1},\chi _{E_{j}^{k}\setminus T_{\ell \left( r\right)
}}\omega \right) \delta \left( \frac{\left\vert y-y_{r}^{t+1}\right\vert }{%
\ell \left( G_{r}^{\alpha ,t+1}\right) }\right) \left\vert b_{r}\left(
y\right) \right\vert d\sigma \left( y\right) \\
&\leq &C\mathbf{P}\left( G_{r}^{\alpha ,t+1},\chi _{E_{j}^{k}\setminus
3G_{r}^{\alpha ,t+1}}\omega \right) \int_{G_{r}^{\alpha ,t+1}}\left\vert
f\right\vert d\sigma .  \notag
\end{eqnarray}

We have after application of \eqref{e.3rd}, 
\begin{align}
I\!I_{s}^{t}& =\sum_{(k,j)\in \mathbb{I}_{s}^{\alpha ,t}}R_{j}^{k}\left[
\int_{E_{j}^{k}}\left( Lh_{s}^{\alpha ,t}\sigma \right) \omega \right] ^{p}
\label{e.IIstDEF2} \\
& \leq I\!I_{s}^{t}(1)+I\!I_{s}^{t}(2)\,,  \label{IItilda} \\
I\!I_{s}^{t}(1)& =\sum_{(k,j)\in \mathbb{I}_{s}^{\alpha
,t}}R_{j}^{k}\left\vert \sum_{r\in \mathcal{K}_{s}^{\alpha
,t}}\int_{G_{r}^{\alpha ,t+1}}\left( L^{\ast }\chi _{E_{j}^{k}\cap T_{\ell
\left( r\right) }}\omega \right) b_{r}\sigma \right\vert ^{p},  \label{e.II1}
\\
I\!I_{s}^{t}(2)& =\sum_{(k,j)\in \mathbb{I}_{s}^{\alpha ,t}}R_{j}^{k}\left[
\sum_{r\in \mathcal{K}_{s}^{\alpha ,t}}\mathbf{P}\left( G_{r}^{\alpha
,t+1},\chi _{E_{j}^{k}}\omega \right) \int_{G_{r}^{\alpha ,t+1}}\left\vert
f\right\vert \sigma \right] ^{p}.  \label{e.II2}
\end{align}%
Note that we may further restrict the integration in \eqref{e.II1} to $%
G_{r}^{\alpha ,t+1}\cap 3Q_{j}^{k}$ since $L^{\ast }\chi _{E_{j}^{k}\cap
T_{\ell \left( r\right) }}\omega $ is supported in $3Q_{j}^{k}$.

%%%%%%%%%%%%%%%%%%%%%%%%%%%%%% SUBSUBSECTION SUBSUBSECTION SUBSUBSECTION 

\subsubsection{Analysis of $I\!I(2)$.}

We claim that%
\begin{equation}
\sum_{\left( t,s\right) \in \mathbb{L}^{\alpha }}I\!I_{s}^{t}(2)\leq C\gamma
^{2p}\mathfrak{M}_{\ast }^{p}\int \left\vert f\right\vert ^{p}\sigma \,.
\label{e.II2<}
\end{equation}%
Recall the definition of $\mathfrak{M}_{\ast }$ in \eqref{e.Mfrak*}.

\begin{proof}
We begin by defining a linear operator by 
\begin{equation}
\mathbf{P}_{j}^{k}\left( \mu \right) \equiv \sum_{r\in \mathcal{K}%
_{s}^{\alpha ,t}}\mathbf{P}\left( G_{r}^{\alpha ,t+1},\chi _{E_{j}^{k}}\mu
\right) \chi _{G_{r}^{\alpha ,t+1}}.  \label{e.Pjk}
\end{equation}%
In this notation, we have for $\left( k,j\right) \in \mathbb{I}_{s}^{\alpha
,t}$ (See \eqref{e.HST} and \eqref{e.IIst}), 
\begin{align*}
\sum_{r\in \mathcal{K}_{s}^{\alpha ,t}}\mathbf{P}\Bigl(& G_{r}^{\alpha
,t+1},\chi _{E_{j}^{k}}\omega (dx)\Bigr)\int_{G_{r}^{\alpha ,t+1}}\left\vert
f\right\vert \sigma \\
& =\sum_{r\in \mathcal{K}_{s}^{\alpha ,t}}\mathbf{P}\Bigl(G_{r}^{\alpha
,t+1},\chi _{E_{j}^{k}}\omega \Bigr)\int_{G_{r}^{\alpha ,t+1}}\sigma \\
& \times \left\{ \frac{1}{\left\vert G_{r}^{\alpha ,t+1}\right\vert _{\sigma
}}\int_{G_{r}^{\alpha ,t+1}}\left\vert f\right\vert \sigma \right\} \\
& \leq \gamma ^{t+2}\int_{G_{s}^{\alpha ,t}}\mathbf{P}_{j}^{k}(\omega
)\sigma =\gamma ^{t+2}\int_{E_{j}^{k}}(\mathbf{P}_{j}^{k})^{\ast }(\chi
_{G_{s}^{\alpha ,t}}\sigma )\omega .
\end{align*}

By assumption, the maximal function $\mathcal{M}(\omega \cdot )$ maps $%
L^{p^{\prime }}(\omega )$ to $L^{p^{\prime }}(\sigma )$, and we now note a
particular consequence of this. In the definition \eqref{e.Pjk} we were
careful to insert $\chi _{E_{j}^{k}}$ on the right hand side. These sets are
pairwise disjoint, whence we have the inequality below for measures $\mu $. 
\begin{eqnarray}
&&\sum_{\left( k,j\right) \in \mathbb{I}_{s}^{\alpha ,t}}\mathbf{P}%
_{j}^{k}\left( \mu \right) \left( x\right)  \label{sumkj} \\
&&\ \ \ \ \ \leq \sum_{\left( k,j\right) \in \mathbb{I}_{s}^{\alpha
,t}}\sum_{r\in \mathcal{K}_{s}^{\alpha ,t}}\sum_{\ell =0}^{\infty }\frac{%
\delta \left( 2^{-\ell }\right) }{\left\vert 2^{\ell }G_{r}^{\alpha
,t+1}\right\vert }\left( \int_{2^{\ell }G_{r}^{\alpha ,t+1}}\chi
_{E_{j}^{k}}\mu \right) \chi _{G_{r}^{\alpha ,t+1}}\left( x\right)  \notag \\
&&\ \ \ \ \ \leq \sum_{\ell =0}^{\infty }\sum_{r\in \mathcal{K}_{s}^{\alpha
,t}}\frac{\delta \left( 2^{-\ell }\right) }{\left\vert 2^{\ell
}G_{r}^{\alpha ,t+1}\right\vert }\left( \int_{2^{\ell }G_{r}^{\alpha
,t+1}\cap G_{s}^{\alpha ,t}}\mu \right) \chi _{G_{r}^{\alpha ,t+1}}\left(
x\right)  \notag \\
&&\ \ \ \ \ \leq C\chi _{G_{s}^{\alpha ,t}}\mathcal{M}\left( \chi
_{G_{s}^{\alpha ,t}}\mu \right) \left( x\right) .  \notag
\end{eqnarray}%
Thus the inequality%
\begin{equation}
\lVert \chi _{G_{s}^{\alpha ,t}}\sum_{\left( k,j\right) \in \mathbb{I}%
_{s}^{\alpha ,t}}{\mathbf{P}}_{j}^{k}(\lvert g\rvert \omega )\rVert
_{L^{p^{\prime }}(\sigma )}\leq C\mathfrak{M}_{\ast }\lVert \chi
_{G_{s}^{\alpha ,t}}g\rVert _{L^{p^{\prime }}(\omega )}
\label{dual pivotal substitute}
\end{equation}%
follows immediately. By duality we then have 
\begin{equation}
\bigl\lVert\chi _{G_{s}^{\alpha ,t}}\sum_{\left( k,j\right) \in \mathbb{I}%
_{s}^{\alpha ,t}}({\mathbf{P}}_{j}^{k})^{\ast }(\lvert h\rvert \sigma )%
\bigr\rVert_{L^{p}(w)}\leq C\mathfrak{M}_{\ast }\lVert \chi _{G_{s}^{\alpha
,t}}h\rVert _{L^{p}(\sigma )}\ .  \label{e.p'app}
\end{equation}%
Note that it was the linearity that we wanted in \eqref{e.Pjk}, so that we
could appeal to the dual maximal function assumption.

We thus obtain 
\begin{equation*}
I\!I_{s}^{t}(2)\leq \gamma ^{p\left( t+2\right) }\sum_{(k,j)\in \mathbb{I}%
_{s}^{\alpha ,t}}R_{j}^{k}\left[ \int_{Q_{j}^{k}}(\mathbf{P}_{j}^{k})^{\ast
}\left( \chi _{G_{s}^{\alpha ,t}}\sigma \right) d\omega \right] ^{p}.
\end{equation*}%
Summing in $(t,s)$ and using $({\mathbf{P}}_{j}^{k})^{\ast }\leq \sum_{(\ell
,i)\in \mathbb{I}_{s}^{\alpha ,t}}({\mathbf{P}}_{i}^{\ell })^{\ast }$ for $%
(k,j)\in \mathbb{I}_{s}^{\alpha ,t}$ we obtain%
\begin{eqnarray}
\sum_{\left( t,s\right) \in \mathbb{L}^{\alpha }}I\!I_{s}^{t}(2) &\leq
&C\gamma ^{2p}\sum_{\left( t,s\right) \in \mathbb{L}^{\alpha }}\gamma
^{pt}\sum_{(k,j)\in \mathbb{I}_{s}^{\alpha ,t}}R_{j}^{k}\left[
\int_{Q_{j}^{k}}(\mathbf{P}_{j}^{k})^{\ast }\left( \chi _{G_{s}^{\alpha
,t}}\sigma \right) d\omega \right] ^{p}  \label{sum in ts} \\
&=&C\gamma ^{2p}\sum_{\left( t,s\right) \in \mathbb{L}^{\alpha }}\gamma
^{pt}\sum_{(k,j)\in \mathbb{I}_{s}^{\alpha ,t}}\left\vert
E_{j}^{k}\right\vert _{\omega }\left[ \frac{1}{\left\vert
NQ_{j}^{k}\right\vert _{w}}\int_{Q_{j}^{k}}({\mathbf{P}}_{j}^{k})^{\ast
}\left( \chi _{G_{s}^{\alpha ,t}}\sigma \right) \omega \right] ^{p}  \notag
\\
&\leq &C\gamma ^{2p}\sum_{\left( t,s\right) \in \mathbb{L}^{\alpha }}\gamma
^{pt}\int \left[ \mathcal{M}_{\omega }^{dy}\left( \chi _{G_{s}^{\alpha
,t}}\sum_{(\ell ,i)\in \mathbb{I}_{s}^{\alpha ,t}}({\mathbf{P}}_{i}^{\ell
})^{\ast }\left( \chi _{G_{s}^{\alpha ,t}}\sigma \right) \right) \right]
^{p}\omega \\
&\leq &C\gamma ^{2p}\sum_{\left( t,s\right) \in \mathbb{L}^{\alpha }}\gamma
^{pt}\int_{G_{s}^{\alpha ,t}}\left[ \sum_{(\ell ,i)\in \mathbb{I}%
_{s}^{\alpha ,t}}({\mathbf{P}}_{i}^{\ell })^{\ast }\left( \chi
_{G_{s}^{\alpha ,t}}\sigma \right) \right] ^{p}\omega \\
&\leq &C\gamma ^{2p}\mathfrak{M}_{\ast }^{p}\sum_{\left( t,s\right) \in 
\mathbb{L}^{\alpha }}\gamma ^{pt}\left\vert G_{s}^{\alpha ,t}\right\vert
_{\sigma },  \notag
\end{eqnarray}%
which is bounded by $C\gamma ^{2p}\mathfrak{M}_{\ast }^{p}\int \left\vert
f\right\vert ^{p}\sigma $. In the last line we are applying \eqref{e.p'app}
with $h\equiv 1$.
\end{proof}

\subsubsection{Decomposition of $I\!I(1)$}

We note that the term $I\!I_{s}^{t}(1)$ is dominated by $I\!I_{s}^{t}(1)\leq
I\!I\!I_{s}^{t}+I\!V_{s}^{t}$, where%
\begin{equation*}
I\!I\!I_{s}^{t}=\sum_{(k,j)\in \mathbb{I}_{s}^{\alpha
,t}}R_{j}^{k}\left\vert \sum_{r\in \mathcal{K}_{s}^{\alpha
,t}}\int_{G_{r}^{\alpha ,t+1}\setminus \Omega _{k+2}}\left( L^{\ast }\chi
_{E_{j}^{k}\cap T_{\ell \left( r\right) }}\omega \right) b_{r}\sigma
\right\vert ^{p},
\end{equation*}%
\begin{equation}
I\!V_{s}^{t}=\sum_{(k,j)\in \mathbb{I}_{s}^{\alpha ,t}}R_{j}^{k}\left\vert
\sum_{r\in \mathcal{K}_{s}^{\alpha ,t}}\int_{G_{r}^{\alpha ,t+1}\cap \Omega
_{k+2}}\left( L^{\ast }\chi _{E_{j}^{k}\cap T_{\ell \left( r\right) }}\omega
\right) b_{r}\sigma \right\vert ^{p}.  \label{e.IVstDef}
\end{equation}%
The term $I\!I\!I_{s}^{t}$ includes that part of $b_{r}$ supported on $%
G_{r}^{\alpha ,t+1}\setminus \Omega _{k+2}$, and the term $I\!V_{s}^{t}$
includes that part of $b_{r}$ supported on $G_{r}^{\alpha ,t+1}\cap \Omega
_{k+2}$, which is the more delicate case.

\begin{remark}
\label{r.gettingDifficult} The key difference between the terms $%
I\!I\!I_{s}^{t}$ and $IV_{s}^{t}$ is the range of integration: $%
G_{r}^{\alpha ,t+1}\setminus \Omega _{k+2}$ for $I\!I\!I_{s}^{t}$ and $%
G_{r}^{\alpha ,t+1}\cap \Omega _{k+2}$ for $IV_{s}^{t}$. Just as for the
fractional integral case, it is the latter case that is harder, requiring
combinatorial facts, which we come to at the end of the argument. An
additional fact that we return to in different forms, is that the set $%
G_{r}^{\alpha ,t+1}\cap \Omega _{k+2}$ can be further decomposed using
Whitney decompositions of $\Omega _{k+2}$ in the grid $\mathcal{D}^{\alpha }$%
.
\end{remark}

Recall the definition of $\mathfrak{T}_{\ast }$ in \eqref{e.Tfrak}. We claim 
\begin{equation}
\sum_{\left( t,s\right) \in \mathbb{L}^{\alpha }}I\!I\!I_{s}^{t}\leq C%
\mathfrak{T}_{\ast }^{p}\int \left\vert f\right\vert ^{p}\sigma .
\label{e.III<}
\end{equation}

\begin{proof}
Let $\widetilde{E_{j}^{k}}=3Q_{j}^{k}\setminus \Omega _{k+2}$ (note that $%
\widetilde{E_{j}^{k}}$ is much larger than $E_{j}^{k}$). We will use the
definition of $R_{j}^{k}$ in \eqref{Rkj'}, and the fact that 
\begin{equation}
\sum_{\ell \in \mathcal{L}_{s}^{\alpha ,t}}\chi _{T_{\ell }}\leq 3^{n}
\label{finover}
\end{equation}%
provided $N\geq 9$. We will apply the form \eqref{unif'} of (\ref%
{Tsharpomega}) with $g=\chi _{E_{j}^{k}\cap T_{\ell }}$, also see %
\eqref{e.Tfrak}, and with%
\begin{equation*}
Q\equiv T_{\ell }\cap \widehat{Q_{j}^{k}}
\end{equation*}%
in the case $T_{\ell }\cap \widehat{Q_{j}^{k}}$ is a cube, and with%
\begin{equation*}
Q\equiv T_{\ell }
\end{equation*}%
in the case $T_{\ell }\cap \widehat{Q_{j}^{k}}$ is $\emph{not}$ a cube (this
is possible since $T_{\ell }$ is the \emph{triple} of a $\mathcal{D}^{\alpha
}$-cube). In each case we claim that%
\begin{equation*}
Q\subset T_{\ell }\cap 3\widehat{Q_{j}^{k}}.
\end{equation*}%
Indeed, recall that $\widehat{Q_{j}^{k}}$ is the cube in the shifted grid $%
\mathcal{D}^{\alpha }$ that is selected by $Q_{j}^{k}$ as in the definition `%
\textbf{Select a shifted grid'} above and satisfies $3\widehat{Q_{j}^{k}}%
\subset MQ_{j}^{k}\subset NQ_{j}^{k}$ where $N$ is as in Remark \ref{3N}, by
choosing $R_{W}$ sufficiently large in \eqref{Whitney}. Now $T_{\ell }$ is a
triple of a cube in the grid $\mathcal{D}^{\alpha }$ and $\widehat{Q_{j}^{k}}
$ is a cube in $\mathcal{D}^{\alpha }$. Thus if $T_{\ell }\cap \widehat{%
Q_{j}^{k}}$ is \emph{not} a cube, then we must have $T_{\ell }\subset 3%
\widetilde{Q_{j}^{k}}$ and this proves the claim. We then have 
\begin{eqnarray*}
I\!I\!I_{s}^{t} &\leq &\sum_{(k,j)\in \mathbb{I}_{s}^{\alpha ,t}}R_{j}^{k}%
\left[ \sum_{\ell \in \mathcal{L}_{s}^{\alpha ,t}}\sum_{r\in \mathcal{K}%
_{s}^{\alpha ,t}:\ell =\ell \left( r\right) }\int_{G_{r}^{\alpha ,t+1}\cap 
\widetilde{E_{j}^{k}}}\left\vert L^{\ast }\chi _{E_{j}^{k}\cap T_{\ell
\left( r\right) }}\omega \right\vert ^{p^{\prime }}\sigma \right]
^{p-1}\int_{\widetilde{E_{j}^{k}}}\left\vert h_{s}^{\alpha ,t}\right\vert
^{p}\sigma \\
&\leq &\sum_{(k,j)\in \mathbb{I}_{s}^{\alpha ,t}}R_{j}^{k}\left[ \sum_{\ell
\in \mathcal{L}_{s}^{\alpha ,t}}\int_{T_{\ell }\cap 3\widehat{Q_{j}^{k}}%
}\left\vert L^{\ast }\chi _{E_{j}^{k}\cap T_{\ell }}\omega \right\vert
^{p^{\prime }}\sigma \right] ^{p-1}\int_{\widetilde{E_{j}^{k}}}\left\vert
h_{s}^{\alpha ,t}\right\vert ^{p}\sigma \\
&\leq &\mathfrak{T}_{\ast }^{p}\sum_{(k,j)\in \mathbb{I}_{s}^{\alpha
,t}}R_{j}^{k}\left[ \sum_{\ell \in \mathcal{L}_{s}^{\alpha ,t}}\left\vert
T_{\ell }\cap 3\widehat{Q_{j}^{k}}\right\vert _{\omega }\right] ^{p-1}\int_{%
\widetilde{E_{j}^{k}}}\left\vert h_{s}^{\alpha ,t}\right\vert ^{p}\sigma \\
&\leq &\mathfrak{T}_{\ast }^{p}\sum_{(k,j)\in \mathbb{I}_{s}^{\alpha ,t}}%
\frac{\left\vert E_{j}^{k}\right\vert _{\omega }}{\left\vert
NQ_{j}^{k}\right\vert _{\omega }}\left\vert NQ_{j}^{k}\right\vert _{\omega
}^{p-1}\int_{\widetilde{E_{j}^{k}}}\left\vert h_{s}^{\alpha ,t}\right\vert
^{p}\sigma \\
&\leq &C\mathfrak{T}_{\ast }^{p}\sum_{(k,j)\in \mathbb{I}_{s}^{\alpha
,t}}\int_{\widetilde{E_{j}^{k}}}\left\vert h_{s}^{\alpha ,t}\right\vert
^{p}\sigma \leq C\mathfrak{T}_{\ast }^{p}\sum_{(k,j)\in \mathbb{G}^{\alpha
}\cap \mathbb{H}_{s}^{\alpha ,t}}\int_{\widetilde{E_{j}^{k}}}\left(
\left\vert f\right\vert ^{p}+\left\vert \mathcal{M}_{\sigma }^{\alpha
}f\right\vert ^{p}\right) \sigma .
\end{eqnarray*}%
Using%
\begin{equation}
\sum_{\left( t,s\right) \in \mathbb{L}^{\alpha }}\sum_{(k,j)\in \mathbb{G}%
^{\alpha }\cap \mathbb{H}_{s}^{\alpha ,t}}\chi _{\widetilde{E_{j}^{k}}%
}=\sum_{all\ k,j}\chi _{\widetilde{E_{j}^{k}}}\leq C,  \label{Edecomp}
\end{equation}%
we thus obtain \eqref{e.III<}.
\end{proof}

\subsection{The Analysis of the Bad Function: Part 2}

This is the most intricate and final case. We will prove%
\begin{equation}
\sum_{\left( t,s\right) \in \mathbb{L}^{\alpha }}I\!V_{s}^{t}\leq C\{\gamma
^{2p}\mathfrak{T}^{p}+\mathfrak{T}_{\ast }^{p}+\gamma ^{2p}\mathfrak{M}%
_{\ast }^{p}\}\int \left\vert f\right\vert ^{p}\sigma \,,  \label{e.IV<}
\end{equation}%
where $\mathfrak{T}$, $\mathfrak{T}_{\ast }$ and $\mathfrak{M}_{\ast }$ are
defined in \eqref{e.TfrakStar}, \eqref{e.Tfrak} and \eqref{e.Mfrak*}
respectively. The estimates \eqref{e.I<}, \eqref{e.II2<}, \eqref{e.III<}, %
\eqref{e.IV<} prove \eqref{e.T33}, and so complete the proof of assertion 1
of the strong type characterization in Theorem \ref{twoweightHaar}.
Assertions 2 and 3 of Theorem \ref{twoweightHaar} follow as in the weak-type
Theorem \ref{weaktwoweightHaar}. Finally, to prove assertion 4 we note that
Lemma \ref{dom} and condition \eqref{Tsharpsigma} imply \eqref{testmax},
which by Theorem \ref{sawyerthm} yields \eqref{M2weight}.

\subsubsection{Whitney decompositions with shifted grids}

\label{s.whitneyShift}

We now use the shifted grid $\mathcal{D}^{\alpha }$ in place of the dyadic
grid $\mathcal{D}$\ to form a Whitney decomposition of $\Omega _{k}$ in the
spirit of \eqref{Whitney}. However, in order to fit the $\mathcal{D}^{\alpha
}$-cubes $\widehat{Q_{j}^{k}}$ defined above in "\textbf{Select a shifted
grid}", it will be necessary to use a smaller constant than the constant $%
R_{W}$ already used for the Whitney decomposition of $\Omega _{k}$ into $%
\mathcal{D}$-cubes. Recall the dimensional constant $M$ defined in (\ref%
{QhatContained}): it satisfies $\widehat{Q}\subset MQ$. Define the new
constant 
\begin{equation*}
R_{W}^{\prime }=\frac{R_{W}}{M}.
\end{equation*}%
We now use the decomposition of $\Omega _{k}$ in \eqref{Whitney}, but with $%
\mathcal{D}$ replaced by $\mathcal{D}^{\alpha }$ and with $R_{W}$ replaced
by $R_{W}^{\prime }$. We have thus decomposed%
\begin{equation*}
\Omega _{k}=\overset{\cdot }{{\bigcup_{m}}}B_{m}^{k}
\end{equation*}%
into a Whitney decomposition of pairwise disjoint cubes $B_{m}^{k}$ in $%
\mathcal{D}^{\alpha }$ satisfying%
\begin{eqnarray}
R_{W}^{\prime }B_{m}^{k} &\subset &\Omega _{k},  \label{Whitney alpha} \\
3R_{W}^{\prime }B_{m}^{k}\cap \Omega _{k}^{c} &\neq &\emptyset ,  \notag
\end{eqnarray}%
and the following analogue of the nested property in \eqref{Whitney}:%
\begin{equation}
B_{j}^{k}\varsubsetneqq B_{i}^{\ell }\text{ implies }k>\ell .  \label{nested}
\end{equation}

Now we introduce yet another construction. For every pair $\left( k,j\right) 
$ let $\widetilde{Q_{j}^{k}}$ be the unique $\mathcal{D}^{\alpha }$-cube $%
B_{m}^{k}$ containing $\widehat{Q_{j}^{k}}$. Note that such a cube $%
\widetilde{Q_{j}^{k}}=B_{m}^{k}$ exists since $\widehat{Q_{j}^{k}}\subset
MQ_{j}^{k}$ by \eqref{QhatContained} and $R_{W}Q_{j}^{k}\subset \Omega _{k}$
by \eqref{Whitney} imply that $R_{W}^{\prime }\widehat{Q_{j}^{k}}\subset
\Omega _{k}$. Of course the cube $\widetilde{Q_{j}^{k}}=B_{m}^{k}$ satisfies 
\begin{equation}
R_{W}^{\prime }\widetilde{Q_{j}^{k}}\subset \Omega _{k}.  \label{def Qtilda}
\end{equation}%
Moreover, we can arrange to have%
\begin{equation}
3\widetilde{Q_{j}^{k}}\subset NQ_{j}^{k},  \label{Qtilda contained}
\end{equation}%
where $N$ is as in Remark \ref{3N}, by choosing $R_{W}$ sufficiently large
in \eqref{Whitney}. See Figure~\ref{f.4}.

%%%%%%%%%%%%%%% Figure
\begin{figure}[ptb]
\begin{center}
\includegraphics{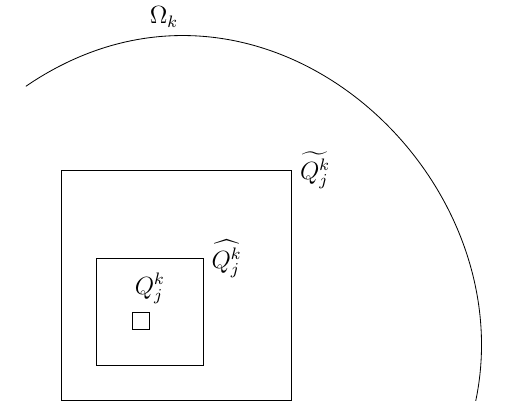}
\end{center}
\caption{The relative positions of the cubes $Q_{j}^{k}$, $\widehat{Q_{j}^{k}%
}$, and $\widetilde{Q_{j}^{k}}$ inside a set $\Omega _{k}$.}
\label{f.4}
\end{figure}
%
%%%%%%%%%%%%%%% Figure
%%%%%%%%%%%%%%% Figure

We will use this decomposition for the set $\Omega _{k+2}=\overset{\cdot }{%
\bigcup }_{m}B_{m}^{k+2}$ in our arguments below. The corresponding cubes $%
\widetilde{Q_{i}^{k+2}}$ that arise as certain of the $B_{m}^{k+2}$ satisfy
the conditions, 
\begin{equation}
3Q_{i}^{k+2}\subset \widehat{Q_{i}^{k+2}}\subset \widetilde{Q_{i}^{k+2}}%
\subset 3\widetilde{Q_{i}^{k+2}}\subset NQ_{i}^{k+2}\subset \Omega _{k+2}.
\label{modified}
\end{equation}
Note that the set of indices $m$ arising in the decomposition of $\Omega
_{k+2}$ into $\mathcal{D}^{\alpha }$ cubes $B_{m}^{k+2}$ is \emph{not} the
same as the set of indices $i$ arising in the decomposition of $\Omega
_{k+2} $ into $\mathcal{D}$ cubes $Q_{i}^{k+2}$, but this should not cause
confusion. So we will usually write $B_{i}^{k+2}$ with dummy index $i$
unless it is important to distinguish the cubes $B_{i}^{k+2}$ from the cubes 
$Q_{i}^{k+2}$. This distinction will be important in the proof of the
`Bounded Occurrence of Cubes' property in \S ~\ref{combinatorics} below.

Now use $\Omega _{k+2}=\cup B_{i}^{k+2}$ to split the term $I\!V_{s}^{t}$ in %
\eqref{e.IVstDef} into two pieces as follows:%
\begin{eqnarray}  \label{e.IV2Def}
I\!V_{s}^{t} &\leq &\sum_{(k,j)\in \mathbb{I}_{s}^{\alpha
,t}}R_{j}^{k}\left\vert \sum_{r\in \mathcal{K}_{s}^{\alpha ,t}}\sum_{i\in 
\mathcal{I}_{k}^{t}}\int_{G_{r}^{\alpha ,t+1}\cap B_{i}^{k+2}}\left( L^{\ast
}\chi _{E_{j}^{k}\cap T_{\ell \left( r\right) }}\omega \right) b_{r}\sigma
\right\vert ^{p}  \label{rewrite} \\
&&+\sum_{(k,j)\in \mathbb{I}_{s}^{\alpha ,t}}R_{j}^{k}\left\vert \sum_{r\in 
\mathcal{K}_{s}^{\alpha ,t}}\sum_{i\in \mathcal{J}_{k}^{t}}\int_{G_{r}^{%
\alpha ,t+1}\cap B_{i}^{k+2}}\left( L^{\ast }\chi _{E_{j}^{k}\cap T_{\ell
\left( r\right) }}\omega \right) b_{r}\sigma \right\vert ^{p}  \notag \\
&=&I\!V_{s}^{t}(1)+I\!V_{s}^{t}(2),  \notag
\end{eqnarray}%
where 
\begin{equation}
\mathcal{I}_{k}^{t}=\left\{ i:\mathsf{A}_{i}^{k+2}>\gamma ^{t+2}\right\} 
\text{ and }\mathcal{J}_{k}^{t}=\left\{ i:\mathsf{A}_{i}^{k+2}\leq \gamma
^{t+2}\right\} ,  \label{e.IJdef}
\end{equation}%
and where 
\begin{equation}
\mathsf{A}_{i}^{k+2}=\frac{1}{\left\vert B_{i}^{k+2}\right\vert _{\sigma }}%
\int_{B_{i}^{k+2}}\left\vert f\right\vert d\sigma   \label{e.Aik}
\end{equation}%
denotes the $\sigma $-average of $\left\vert f\right\vert $ on the cube $%
B_{i}^{k+2}$. We will occasionally write out $i:\mathsf{A}_{i}^{k+2}>\gamma
^{t+2}$ in place of $i\in \mathcal{I}_{k}^{t}$ for clarity, and the reader
may wish to do this throughout the argument, and similarly for $i\in 
\mathcal{J}_{k}^{t}$. Thus $I\!V(1)$ corresponds to the case where the
averages are `big' and $I\!V(2)$ where the averages are `small'. The
analysis of $I\!V_{s}^{t}(1)$ in \eqref{rewrite} is the hard case, taken up
later.

%%%%%%%%%%%%%%%%%%%%%%%%%%%%%% SUBSUBSECTION SUBSUBSECTION SUBSUBSECTION 

\subsubsection{A first combinatorial argument\label{combinatoric}}

\begin{description}
\item[Bounded Occurrence of Cubes] A given cube $B\in \mathcal{D}^{\alpha }$
can occur only a \emph{bounded} number of times as $B_{i}^{k+2}$ where%
\begin{equation*}
B_{i}^{k+2}\subset \widetilde{Q_{j}^{k}}\text{ with }\left( k,j\right) \in 
\mathbb{G}^{\alpha }.
\end{equation*}%
Specifically, let $(k_{1},j_{1}),\dotsc ,(k_{M},j_{M})\in \mathbb{G}^{\alpha
}$, as defined in \ref{e.GdefAlpha}, be such that $B=B_{i_{\sigma
}}^{k_{\sigma }+2}$ for some $i_{\sigma }$ and $B\subset \widetilde{%
Q_{j_{\sigma }}^{k_{\sigma }}}$ for $1\leq \sigma \leq M$. It follows that $%
M\leq C\beta ^{-1}$, where $\beta $ is the small constant chosen in the
definition of $\mathbb{G}^{\alpha }$. The constant $C$ here depends only on
dimension.
\end{description}

%%%%%%%%%%%%%%%%%%%%%%%%%%%%%% PROOF PROOF PROOF

The Whitney structure, see \eqref{Whitney}, is decisive here, as well as the
fact that $\left\vert E_{j}^{k}\right\vert _{\omega }\geq \beta \left\vert
NQ_{j}^{k}\right\vert _{\omega }$ for $\left( k,j\right) \in \mathbb{G}%
^{\alpha }$. For this proof it will be useful to use $m$ to index the cubes $%
B_{m}^{k+2}$ and to use $i$ to index the cubes $Q_{i}^{k+2}$. The following
lemma captures the main essence of the Whitney structure, and will be
applied to cubes $B_{m}^{k+2}$ satisfying (\ref{Whitney alpha}) and cubes $%
Q_{i}^{k+2}$ satisfying (\ref{Whitney}).

\begin{lemma}
\label{Whitney comparability}Suppose that $Q$ is a member of the Whitney
decomposition of $\Omega $ with respect to the grid $\mathcal{D}$ and with
Whitney constant $R_{W}$. Suppose also that a cube $B$ is a member of a
Whitney decomposition of the same open set $\Omega $ but with respect to the
grid $\mathcal{D}^{\alpha }$ and with Whitney constant $R_{W}^{\prime }$. If 
$N<\frac{1}{2}R_{W}$ and $B\subset NQ$, then the sidelengths of $Q$ and $B$
are comparable:%
\begin{equation*}
\ell \left( Q\right) \approx \ell \left( B\right) .
\end{equation*}
\end{lemma}

\begin{proof}[Proof of Lemma \protect\ref{Whitney comparability}]
Since $N<\frac{1}{2}R_{W}$ and $Q$ is a Whitney cube we have 
\begin{equation*}
\ell \left( Q\right) \approx dist\left( Q,\partial \Omega \right) \approx
\sup_{x\in NQ}dist\left( x,\partial \Omega \right) \approx \inf_{x\in
NQ}dist\left( x,\partial \Omega \right) .
\end{equation*}%
Then since $B\subset NQ$ and $B$ is a Whitney cube (for the other
decomposition) we have%
\begin{equation*}
\ell \left( Q\right) \approx dist\left( B,\partial \Omega \right) \approx
\ell \left( B\right) .
\end{equation*}
\end{proof}

\begin{proof}[Proof of Bounded Occurrence of Cubes]
So suppose that $(k_{1},j_{1}),\dotsc ,(k_{M},j_{M})\in \mathbb{G}^{\alpha }$
and $B=B_{i_{\sigma }}^{k_{\sigma }+2}\subset \widetilde{Q_{j_{\sigma
}}^{k_{\sigma }}}$ for $1\leq \sigma \leq M$, with the pairs of indices $%
(k_{\sigma },j_{\sigma })$ being distinct. Observe that the finite overlap
property in \eqref{Whitney} applies to the cubes $\widetilde{Q_{j_{\sigma
}}^{k_{\sigma }}}$ in the Whitney decompostion (\ref{Whitney alpha}) of $%
\Omega _{k_{\sigma }}$ with grid $\mathcal{D}^{\alpha }$ and Whitney
constant $R_{W}^{\prime }$. Thus for fixed $k$ the number of $(k_{\sigma
},j_{\sigma })$ with $k_{\sigma }=k$ is bounded by the finite overlap
constant since $B$ is inside each $\widetilde{Q_{j_{\sigma }}^{k_{\sigma }}}$%
. This gives us the observation that a single integer $k$ can occur only a 
\emph{bounded} number $C_{b}$ of times among the $k_{1},\dotsc ,k_{M}$.

After a relabeling, we can assume that all the $k_{\sigma }$ for $1\leq
\sigma \leq M^{\prime }$ are distinct, listed in increasing order, and that
the number $M^{\prime }$ of $k_{\sigma }$ satisfies $M\leq C_{b}M^{\prime }$%
. The nested property of \eqref{Whitney} assures us that $B$ is an element
of the Whitney decomposition (\ref{Whitney alpha}) of $\Omega _{k}$ for 
\emph{all} $k_{1}\leq k\leq k_{M^{\prime }}$.

\begin{remark}
\label{intermediate cubes}Note that the $k_{\sigma }$ are not necessarily
consecutive since we require that $(k_{\sigma },j_{\sigma })\in \mathbb{G}%
^{\alpha }$. Nevertheless, the cube $B$ \emph{does} occur among the $%
B_{i}^{k+2}$ for any $k$ that lies between $k_{\sigma }$ and $k_{\sigma +1}$%
. These latter occurrences of $B$ may be unbounded, but we are only
concerned with bounding those for which $(k_{\sigma },j_{\sigma })\in 
\mathbb{G}^{\alpha }$, and it is these occurrences that our argument is
treating.
\end{remark}

Thus for $3\leq \sigma \leq M^{\prime }$, we have $k_{1}\leq k_{\sigma
}-2\leq k_{M^{\prime }}$, and it follows from Remark \ref{intermediate cubes}
that the cube $B$ is a member of the Whitney decomposition (\ref{Whitney
alpha}) of the open set $\Omega _{k_{\sigma }}$ with grid $\mathcal{D}%
^{\alpha }$ and Whitney constant $R_{W}^{\prime }$. But we also have that $%
Q_{j_{\sigma }}^{k_{\sigma }}$ is a member of the Whitney decomposition (\ref%
{Whitney}) of $\Omega _{k_{\sigma }}$ with grid $\mathcal{D}$ and Whitney
constant $R_{W}$. Thus Lemma \ref{Whitney comparability} gives us the
equivalence of side lengths $\ell \left( Q_{j_{\sigma }}^{k_{\sigma
}}\right) \approx \ell \left( B\right) $. Combined with the containment $%
NQ_{j_{\sigma }}^{k_{\sigma }}\supset B$, we see that the number of possible
locations for the cubes $Q_{j_{\sigma }}^{k_{\sigma }}\in \mathcal{D}$ is
bounded by a constant $C_{b}^{\prime }$ depending only on dimension.

Apply the pigeonhole principle to the possible locations of the $%
Q_{j_{\sigma }}^{k_{\sigma }}$. After a relabeling, we can argue under the
assumption that all $Q_{j_{\sigma }}^{k_{\sigma }}$ equal the same cube $%
Q^{\prime }$ for all choices of $1\leq \sigma \leq M^{\prime \prime }$ where 
$M^{\prime }\leq C_{b}^{\prime }M^{\prime \prime }$. Now comes the crux of
the argument where the condition that the indices $(k_{\sigma },j_{\sigma })$
lie in $\mathbb{G}^{\alpha }$, as given in (\ref{e.GdefAlpha}), proves
critical. In particular we have $\lvert E_{j_{\sigma }}^{k_{\sigma }}\rvert
_{\omega }\geq \beta \lvert NQ^{\prime }\rvert _{\omega }$ where $N$ is as
in Remark \ref{3N}. The $k_{\sigma }$ are distinct, and the sets $%
E_{j_{\sigma }}^{k_{\sigma }}\subset Q^{\prime }$ are pairwise disjoint,
hence 
\begin{equation*}
M^{\prime \prime }\beta \lvert NQ^{\prime }\rvert _{\omega }\leq
\sum_{\sigma =1}^{M^{\prime \prime }}\lvert E_{j_{\sigma }}^{k_{\sigma
}}\rvert _{\omega }\leq \lvert Q^{\prime }\rvert _{\omega }
\end{equation*}%
implies $M^{\prime \prime }\leq \beta ^{-1}$. Thus \thinspace $M\leq
C_{b}C_{b}^{\prime }\beta ^{-1}$ and our proof of the claim is complete.
\end{proof}

%%%%%%%%%%%%%%%%%%%%%%%%%%%%%% PROOF PROOF PROOF

\subsubsection{Replace bad functions by averages}

The first task in the analysis of the terms $IV_{s}^{t}(1) $ and $%
IV_{s}^{t}(2) $ will be to replace part of the `bad functions' $b_{r}$ by
their averages over $B_{i}^{k+2}$, or more exactly the averages $\mathsf{A}%
_{i}^{k+2}$. We again appeal to the H\"{o}lder continuity of $L^{\ast }\chi
_{E_{j}^{k}\cap T_{\ell }}\omega $. By construction, $3B_{i}^{k+2}$ does not
meet $E_{j}^{k}$, so that Lemma~\ref{constant} applies. If $%
B_{i}^{k+2}\subset G_{r}^{\alpha ,t+1}$ for some $r$, then there is a
constant $c_{i}^{k+2}$ satisfying $\left\vert c_{i}^{k+2}\right\vert \leq 1$
such that%
\begin{eqnarray}
\Biggl\vert\int_{B_{i}^{k+2}}\left( L^{\ast }\chi _{E_{j}^{k}\cap T_{\ell
\left( r\right) }}\omega \right) b_{r}\sigma &-&\left\{
c_{i}^{k+2}\int_{B_{i}^{k+2}}\left( L^{\ast }\chi _{E_{j}^{k}\cap T_{\ell
\left( r\right) }}\omega \right) \sigma \right\} \left( \left\vert
A_{r}^{\alpha ,t+1}\right\vert +\mathsf{A}_{i}^{k+2}\right) \Biggr\vert
\label{e.replace} \\
&\leq &C\mathbf{P}\left( B_{i}^{k+2},\chi _{E_{j}^{k}\cap T_{\ell \left(
r\right) }}\omega \right) \int_{B_{i}^{k+2}}\left\vert b_{r}\right\vert
\sigma \,.  \notag
\end{eqnarray}%
Indeed, if $z_{i}^{k+2}$ is the center of the cube $B_{i}^{k+2}$, we have%
\begin{eqnarray*}
&&\int_{B_{i}^{k+2}}\left( L^{\ast }\chi _{E_{j}^{k}\cap T_{\ell \left(
r\right) }}\omega \right) b_{r}\sigma \\
&=&L^{\ast }\left( \chi _{E_{j}^{k}\cap T_{\ell \left( r\right) }}\omega
\right) \left( z_{i}^{k+2}\right) \int_{B_{i}^{k+2}}b_{r}\sigma +O\left\{ 
\mathbf{P}\left( B_{i}^{k+2},\chi _{E_{j}^{k}\cap T_{\ell \left( r\right)
}}\omega \right) \int_{B_{i}^{k+2}}\left\vert b_{r}\right\vert \sigma
\right\} \\
&=&\left\{ \int_{B_{i}^{k+2}}\left( L^{\ast }\chi _{E_{j}^{k}\cap T_{\ell
\left( r\right) }}\omega \right) \sigma \right\} \frac{1}{\left\vert
B_{i}^{k+2}\right\vert _{\sigma }}\int_{B_{i}^{k+2}}b_{r}\sigma \\
&&+O\left\{ \mathbf{P}\left( B_{i}^{k+2},\chi _{E_{j}^{k}\cap T_{\ell \left(
r\right) }}\omega \right) \int_{B_{i}^{k+2}}\left\vert b_{r}\right\vert
\sigma \right\} .
\end{eqnarray*}%
Now, the functions $b_{r}$ are given in \eqref{e.brDef}, and by
construction, we note that%
\begin{equation*}
\frac{1}{\left\vert B_{i}^{k+2}\right\vert _{\sigma }}\left\vert
\int_{B_{i}^{k+2}}b_{r}\sigma \right\vert \leq \left\vert \frac{1}{%
\left\vert G_{r}^{\alpha ,t+1}\right\vert _{\sigma }}\int_{G_{r}^{\alpha
,t+1}}f\sigma \right\vert +\frac{1}{\left\vert B_{i}^{k+2}\right\vert
_{\sigma }}\int_{B_{i}^{k+2}}\left\vert f\right\vert \sigma =\left\vert
A_{r}^{\alpha ,t+1}\right\vert +\mathsf{A}_{i}^{k+2}.
\end{equation*}%
So with%
\begin{equation*}
c_{i}^{k+2}=\frac{1}{\left\vert A_{r}^{\alpha ,t+1}\right\vert +\mathsf{A}%
_{i}^{k+2}}\frac{1}{\left\vert B_{i}^{k+2}\right\vert _{\sigma }}%
\int_{B_{i}^{k+2}}b_{r}\sigma ,
\end{equation*}%
we have $\left\vert c_{i}^{k+2}\right\vert \leq 1$ and%
\begin{eqnarray*}
\int_{B_{i}^{k+2}}\left( L^{\ast }\chi _{E_{j}^{k}\cap T_{\ell \left(
r\right) }}\omega \right) b_{r}\sigma &=&\left\{
c_{i}^{k+2}\int_{B_{i}^{k+2}}\left( L^{\ast }\chi _{E_{j}^{k}\cap T_{\ell
\left( r\right) }}\omega \right) \sigma \right\} \left( \left\vert
A_{r}^{\alpha ,t+1}\right\vert +\mathsf{A}_{i}^{k+2}\right) \\
&&+O\left\{ \mathbf{P}\left( B_{i}^{k+2},\chi _{E_{j}^{k}\cap T_{\ell \left(
r\right) }}\omega \right) \int_{B_{i}^{k+2}}\left\vert b_{r}\right\vert
\sigma \right\} .
\end{eqnarray*}%
In the special case where $B_{i}^{k+2}$ is equal to $G_{r}^{\alpha ,t+1}$,
then $\int_{B_{i}^{k+2}}b_{r}\sigma =\int b_{r}\sigma =0$ and the above
proof shows that%
\begin{equation}
\left\vert \int_{G_{r}^{\alpha ,t+1}}\left( L^{\ast }\chi _{E_{j}^{k}\cap
T_{\ell \left( r\right) }}\omega \right) b_{r}\sigma \right\vert \leq C%
\mathbf{P}\left( G_{r}^{\alpha ,t+1},\chi _{E_{j}^{k}\cap T_{\ell \left(
r\right) }}\omega \right) \int_{G_{r}^{\alpha ,t+1}}\left\vert f\right\vert
\sigma \ ,  \label{e.replace'}
\end{equation}%
since $\int_{G_{r}^{\alpha ,t+1}}\left\vert b_{r}\right\vert \sigma
=\int_{G_{r}^{\alpha ,t+1}}\left\vert f-A_{r}^{\alpha ,t+1}\right\vert
\sigma \leq 2\int_{G_{r}^{\alpha ,t+1}}\left\vert f\right\vert \sigma $.

Our next task is to organize the sum over the cubes $B_{i}^{k+2}$ relative
to the cubes $G_{r}^{\alpha ,t+1}$. This is necessitated by the fact that
the cubes $B_{i}^{k+2}$ are \emph{not} pairwise disjoint in $k$, and we
thank Tuomas Hyt\"{o}nen for bringing this point to our attention. The cube $%
B_{i}^{k+2}$ must intersect $\bigcup_{r\in \mathcal{K}_{s}^{\alpha
,t}}G_{r}^{\alpha ,t+1}$ since otherwise%
\begin{equation*}
\int_{G_{r}^{\alpha ,t+1}\cap B_{i}^{k+2}}\left( L^{\ast }\chi
_{E_{j}^{k}\cap T_{\ell \left( r\right) }}\omega \right) b_{r}\sigma =0,\ \
\ \ \ r\in \mathcal{K}_{s}^{\alpha ,t}.
\end{equation*}%
Thus $B_{i}^{k+2}$ satisfies exactly one of the following two cases which we
indicate by writing $i\in \QTR{up}{Case}\left( a\right) $ or $i\in \QTR{up}{%
Case}\left( b\right) $:

\textbf{Case (a):} $B_{i}^{k+2}$ \emph{strictly} contains at least one of
the cubes $G_{r}^{\alpha ,t+1}$, $r\in \mathcal{K}_{s}^{\alpha ,t}$;

\textbf{Case (b):} $B_{i}^{k+2}\subset G_{r}^{\alpha ,t+1}$ for some $r\in 
\mathcal{K}_{s}^{\alpha ,t}$.

Note that the cubes $B_{i}^{k+2}$ with $i\in \mathcal{I}_{k}^{t}$ can only
satisfy case (b), while the cubes $B_{i}^{k+2}$ with $i\in \mathcal{J}%
_{k}^{t}$ can can satisfy either of the two cases above. However, we have
the following claim.

\begin{claim}
\label{bounded triples}For each fixed $r\in \mathcal{K}_{s}^{\alpha ,t}$, we
have%
\begin{equation*}
\sum_{\left( k+2,i,j\right) \text{ admissible}}\chi _{B_{i}^{k+2}}\leq C,
\end{equation*}%
where the sum is taken over all \emph{admissible} index \emph{triples} $%
\left( k+2,i,j\right) $, i.e. those for which the cube $B_{i}^{k+2}$ arises
in term $IV_{s}^{t}$ with both $B_{i}^{k+2}\subset G_{r}^{\alpha ,t+1}$ and $%
B_{i}^{k+2}\subset \widetilde{Q_{j}^{k}}$.
\end{claim}

But we first establish a containment that will be useful later as well.
Recall that $\Omega _{k+2}$ decomposes as a pairwise disjoint union of cubes 
$B_{i}^{k+2}$, and thus we have%
\begin{equation*}
\int_{G_{r}^{\alpha ,t+1}\cap \Omega _{k+2}}\left( L^{\ast }\chi
_{E_{j}^{k}\cap T_{\ell \left( r\right) }}\omega \right) b_{r}\sigma
=\sum_{i:B_{i}^{k+2}\cap \widetilde{Q_{j}^{k}}\neq \emptyset
}\int_{B_{i}^{k+2}}\left( L^{\ast }\chi _{E_{j}^{k}\cap T_{\ell \left(
r\right) }}\omega \right) b_{r}\sigma ,
\end{equation*}%
since the support of $L^{\ast }\chi _{E_{j}^{k}\cap T_{\ell \left( r\right)
}}\omega $ is contained in $2Q_{j}^{k}\subset \widehat{Q_{j}^{k}}\subset 
\widetilde{Q_{j}^{k}}$ by \eqref{loc}. Since both $B_{i}^{k+2}$ and $%
\widetilde{Q_{j}^{k}}$ lie in the grid $\mathcal{D}^{\alpha }$ and have
nonempty intersection, one of these cubes is contained in the other. Now $%
B_{i}^{k+2}$ cannot \emph{strictly} contain $\widetilde{Q_{j}^{k}}$ since $%
\widetilde{Q_{j}^{k}}=B_{\ell }^{k}$ for some $\ell $ and the cubes $\left\{
B_{j}^{k}\right\} _{k,j}$ satisfy the nested property \eqref{nested}. It
follows that we must have%
\begin{equation}
B_{i}^{k+2}\subset \widetilde{Q_{j}^{k}}\text{ whenever }B_{i}^{k+2}\cap 
\widetilde{Q_{j}^{k}}\neq \emptyset .  \label{Bhat}
\end{equation}

Now we return to Claim \ref{bounded triples}, and note that for a fixed
index pair $\left( k+2,i\right) $, the bounded overlap condition in (\ref%
{Whitney}) shows that there are only a bounded number of indices $j$ such
that $B_{i}^{k+2}\subset \widetilde{Q_{j}^{k}}\subset NQ_{j}^{k}$ - see (\ref%
{Qtilda contained}). We record this observation here:%
\begin{equation}
\#\left\{ j:B_{i}^{k+2}\subset \widetilde{Q_{j}^{k}}\right\} \leq C,\ \ \ \
\ \text{for each pair }\left( k+2,i\right) .  \label{record}
\end{equation}%
Thus Claim \ref{bounded triples} is reduced to this one.

\begin{claim}
\label{bounded pairs}%
\begin{eqnarray}
&&\sum \left\{ \chi _{B_{i}^{k+2}}:B_{i}^{k+2}\subset G_{r}^{\alpha ,t+1}%
\text{ for \emph{some} }\left( k,j\right) \in \mathbb{I}_{s}^{\alpha ,t}%
\text{ with }B_{i}^{k+2}\subset \widetilde{Q_{j}^{k}}\right\}
\label{k bounded} \\
&&\ \ \ \ \ \ \ \ \ \ \ \ \ \ \ \leq C\text{ for each }r\in \mathcal{K}%
_{s}^{\alpha ,t}.  \notag
\end{eqnarray}
\end{claim}

{As is the case with similar assertations in this argument, a central
obstacle is that a given cube $B$ can arise in many different ways as a $B_i
^{k+2}$.}

\begin{proof}[Proof of Claim \protect\ref{bounded pairs}]
We will appeal to the `Bounded Occurrence of Cubes' in \S \ref{combinatoric}
above. This principle relies upon the definition of $\mathbb{G}^{\alpha }$
in (\ref{e.GdefAlpha}), and applies in this setting due to the definition of 
$\mathbb{I}_{s}^{\alpha ,t}$ in \eqref{e.yyI}. We also appeal to the
following fact:%
\begin{equation}
G_{r}^{\alpha ,t+1}\subset \widetilde{Q_{j}^{k}}\text{ whenever }%
B_{i}^{k+2}\subset G_{r}^{\alpha ,t+1}\cap \widetilde{Q_{j}^{k}}\text{ with }%
\left( k,j\right) \in \mathbb{I}_{s}^{\alpha ,t}.  \label{follows easily}
\end{equation}%
To see (\ref{follows easily}), we note that both of the cubes $G_{r}^{\alpha
,t+1}$ and $\widetilde{Q_{j}^{k}}$ lie in the grid $\mathcal{D}^{\alpha }$
and have nonempty intersection (they contain $B_{i}^{k+2}$), so that one of
these cubes must be contained in the other. However, if $\widetilde{Q_{j}^{k}%
}\subset G_{r}^{\alpha ,t+1}$, then $3Q_{j}^{k}\subset \widehat{Q_{j}^{k}}%
\subset \widetilde{Q_{j}^{k}}$ implies $\mathcal{A}\left( Q_{j}^{k}\right)
\subset G_{r}^{\alpha ,t+1}$, which contradicts $\left( k,j\right) \in 
\mathbb{I}_{s}^{\alpha ,t}$. So we must have $G_{r}^{\alpha ,t+1}\subset 
\widetilde{Q_{j}^{k}}$ as asserted in (\ref{follows easily}).

So to see that \eqref{k bounded} holds, suppose that $\mathcal{A}\left(
Q_{j}^{k}\right) =G_{s}^{\alpha ,t}$ and $B_{i}^{k+2}\subset G_{r}^{\alpha
,t+1}$ with an associated cube $\widetilde{Q_{j}^{k}}$ as in 
\eqref{follows
easily}. Then by (\ref{follows easily}) and (\ref{modified}) the side length 
$\ell \left( Q_{j}^{k}\right) $ of $Q_{j}^{k}$ satisfies 
\begin{equation}
\ell \left( Q_{j}^{k}\right) =\frac{1}{N}\ell \left( NQ_{j}^{k}\right) \geq 
\frac{1}{N}\ell \left( \widetilde{Q_{j}^{k}}\right) \geq \frac{1}{N}\ell
\left( G_{r}^{\alpha ,t+1}\right) .  \label{lb}
\end{equation}%
Also, if $B_{\ell }^{k}$ is \emph{any} Whitney cube at level $k$ that is
contained in $G_{r}^{\alpha ,t+1}$, then by (\ref{follows easily}) and (\ref%
{modified}) we have 
\begin{equation*}
B_{\ell }^{k}\subset G_{r}^{\alpha ,t+1}\subset \widetilde{Q_{j}^{k}}\subset
NQ_{j}^{k},
\end{equation*}%
so that Lemma \ref{Whitney comparability} shows that $B_{\ell }^{k}$ and $%
Q_{j}^{k}$ have comparable side lengths:%
\begin{equation}
\ell \left( B_{\ell }^{k}\right) \approx \ell \left( Q_{j}^{k}\right) .
\label{csl}
\end{equation}%
Moreover, if $B_{\ell ^{\prime }}^{k^{\prime }}$ is any Whitney cube at
level $k^{\prime }<k$ that is contained in $G_{r}^{\alpha ,t+1}$, then there
is some Whitney cube $B_{\ell }^{k}$ at level $k$ such that $B_{\ell
}^{k}\subset B_{\ell ^{\prime }}^{k^{\prime }}$. Thus we have the
containments $B_{\ell }^{k}\subset B_{\ell ^{\prime }}^{k^{\prime }}\subset
NQ_{j}^{k}$, and it follows from (\ref{csl}) that%
\begin{equation}
\ell \left( B_{\ell ^{\prime }}^{k^{\prime }}\right) \approx \ell \left(
Q_{j}^{k}\right) .  \label{csl'}
\end{equation}

Now momentarily \emph{fix} $k_{0}$ such that there is a cube $%
B_{i}^{k_{0}+2} $ satisfying the conditions in (\ref{k bounded}). Then \emph{%
all} of the cubes $B_{\ell }^{k+2}$ that arise in (\ref{k bounded}) with $%
k\leq k_{0}-2$ satisfy 
\begin{equation*}
\ell \left( B_{\ell }^{k+2}\right) \approx \ell \left( Q_{j}^{k_{0}}\right)
\geq \frac{1}{N}\ell \left( G_{r}^{\alpha ,t+1}\right) .
\end{equation*}%
Thus all of the cubes $B_{\ell }^{k+2}$ with $k\leq k_{0}$, except perhaps
those with $k\in \left\{ k_{0}-1,k_{0}\right\} $, have side lengths bounded
below by $c\ \ell \left( G_{r}^{\alpha ,t+1}\right) $, which bounds the
number of possible locations for these cubes by a dimensional constant.
However, those cubes $B_{i}^{k_{0}+1}$ at level $k_{0}+1$ are pairwise
disjoint, as are those cubes $B_{i}^{k_{0}+2}$ at level $k_{0}+2$.
Consequently, we can apply the `Bounded Occurrence of Cubes' to show that
the sum in (\ref{k bounded}), when restricted to $k\leq k_{0}$, is bounded
by a constant $C$ independent of $k_{0}$. Since $k_{0}$ is arbitrary, this
completes the proof of Claim \ref{bounded pairs}.
\end{proof}

As a result of \eqref{k bounded}, for those $i$ in either $\mathcal{I}%
_{k}^{t}$ or $\mathcal{J}_{k}^{t}$ that satisfy case (b), we will be able to
apply below the Poisson argument used to estimate term $I\!I_{s}^{t}\left(
2\right) $ in \eqref{e.II2<} above.

We now further split the sum over $i\in \mathcal{J}_{k}^{t}$ in term $%
I\!V_{s}^{t}(2)$\ into two sums according to the cases (a) and (b) above:%
\begin{eqnarray}
I\!V_{s}^{t}(2) &\leq &\sum_{(k,j)\in \mathbb{I}_{s}^{\alpha
,t}}R_{j}^{k}\left\vert \sum_{r\in \mathcal{K}_{s}^{\alpha ,t}}\sum_{i\in 
\mathcal{J}_{k}^{t}\text{ and }i\in \QTR{up}{Case}\left( a\right)
}\int_{G_{r}^{\alpha ,t+1}\cap B_{i}^{k+2}}\left( L^{\ast }\chi
_{E_{j}^{k}\cap T_{\ell \left( r\right) }}\omega \right) b_{r}\sigma
\right\vert ^{p}  \label{e.iv2b} \\
&&+\sum_{(k,j)\in \mathbb{I}_{s}^{\alpha ,t}}R_{j}^{k}\left\vert \sum_{r\in 
\mathcal{K}_{s}^{\alpha ,t}}\sum_{i\in \mathcal{J}_{k}^{t}\text{ and }i\in
Case\left( b\right) }\int_{G_{r}^{\alpha ,t+1}\cap B_{i}^{k+2}}\left(
L^{\ast }\chi _{E_{j}^{k}\cap T_{\ell \left( r\right) }}\omega \right)
b_{r}\sigma \right\vert ^{p}  \notag \\
&\equiv &I\!V_{s}^{t}(2)[a]+I\!V_{s}^{t}\left( 2\right) [b],  \notag
\end{eqnarray}%
where by $i\in \mathcal{J}_{k}^{t}$ and $i\in \QTR{up}{Case}\left( b\right) $
we mean $i\in \mathcal{J}_{k}^{t}:B_{i}^{k+2}\subset G_{r}^{\alpha ,t+1}$,
with a similar explanation for $\QTR{up}{Case}\left( a\right) $.

We apply the definition of Case (b) and \eqref{e.replace}, to decompose $%
I\!V_{s}^{t}(2)[b]$ as follows. 
\begin{eqnarray}
I\!V_{s}^{t}(2)[b] &=&\sum_{(k,j)\in \mathbb{I}_{s}^{\alpha
,t}}R_{j}^{k}\left\vert \sum_{r\in \mathcal{K}_{s}^{\alpha ,t}}\sum_{i\in 
\mathcal{J}_{k}^{t}:B_{i}^{k+2}\subset G_{r}^{\alpha ,t+1}}\int_{B_{i}^{k+2}}%
\left[ L^{\ast }\chi _{E_{j}^{k}\cap T_{\ell \left( r\right) }}\omega \right]
b_{r}\sigma \right\vert ^{p}  \notag  \label{e.V2def} \\
&\leq &\sum_{(k,j)\in \mathbb{I}_{s}^{\alpha ,t}}R_{j}^{k}\left\vert
\sum_{r\in \mathcal{K}_{s}^{\alpha ,t}}\sum_{i\in \mathcal{J}%
_{k}^{t}:B_{i}^{k+2}\subset G_{r}^{\alpha ,t+1}}\left[ \int_{B_{i}^{k+2}}%
\left( L^{\ast }\chi _{E_{j}^{k}\cap T_{\ell \left( r\right) }}\omega
\right) \sigma \right] \right.   \label{e.3456} \\
&&\qquad \left. \times c_{i}^{k+2}\left( \left\vert A_{r}^{\alpha
,t+1}\right\vert +\mathsf{A}_{i}^{k+2}\right) \right\vert ^{p}  \notag \\
&&+\sum_{(k,j)\in \mathbb{I}_{s}^{\alpha ,t}}R_{j}^{k}\left\vert \sum_{r\in 
\mathcal{K}_{s}^{\alpha ,t}}\sum_{i\in \mathcal{J}_{k}^{t}:B_{i}^{k+2}%
\subset G_{r}^{\alpha ,t+1}}\mathbf{P}\left( B_{i}^{k+2},\chi
_{E_{j}^{k}\cap T_{\ell \left( r\right) }}\omega \right)
\int_{B_{i}^{k+2}}\left\vert b_{r}\right\vert \sigma \right\vert ^{p} \\
&=&V_{s}^{t}(1)+V_{s}^{t}(2).  \notag
\end{eqnarray}

\subsubsection{The bound for $V(2)$}

We claim that 
\begin{equation}
\sum_{\left( t,s\right) \in \mathbb{L}^{\alpha }}V_{s}^{t}(2) \leq C\gamma
^{2p}\mathfrak{M}_{\ast }^{p}\lVert f\rVert _{L^{p}(\sigma )}^{p}\,.
\label{e.V2<}
\end{equation}%
Here, $\mathfrak{M}_{\ast }$ is defined in \eqref{e.Mfrak*}, and $%
V_{s}^{t}(2) $ is defined in \eqref{e.V2def}.

\begin{proof}
The estimate for term $V_{s}^{t}(2)$ is similar to that of $I\!I_{s}^{t}(2)$
above, see \eqref{e.II2<}, except that this time we use Claim \ref{bounded
triples} to handle a complication arising from the extra sum in the cubes $%
B_{i}^{k+2}$. We define 
\begin{equation}
\mathbf{P}_{j}^{k}\left( \mu \right) \equiv \sum_{\ell }\sum_{r\in \mathcal{K%
}_{s}^{\alpha ,t}:\ell \left( r\right) =\ell }\sum_{i\in \mathcal{J}%
_{k}^{t}:B_{i}^{k+2}\subset G_{r}^{\alpha ,t+1}}\mathbf{P}\left(
B_{i}^{k+2},\chi _{E_{j}^{k}\cap T_{\ell }}\mu \right) \chi _{B_{i}^{k+2}}\,.
\label{e.xPjk}
\end{equation}%
We observe that by \eqref{k bounded} the sum of these operators satisfies 
\begin{equation}
\sum_{(k,j)\in \mathbb{I}_{s}^{\alpha ,t}}\mathbf{P}_{j}^{k}\left( \mu
\right) \leq C\chi _{G_{s}^{\alpha ,t}}\mathcal{M}(\chi _{G_{s}^{\alpha
,t}}\mu ),  \label{sum of Pkj}
\end{equation}%
and hence the analogue of \eqref{e.p'app} holds with $\mathbf{P}_{j}^{k}$
defined as above:%
\begin{equation}
\bigl\lVert\chi _{G_{s}^{\alpha ,t}}\sum_{\left( k,j\right) \in \mathbb{I}%
_{s}^{\alpha ,t}}({\mathbf{P}}_{j}^{k})^{\ast }(\lvert h\rvert \sigma )%
\bigr\rVert_{L^{p}(w)}\leq C\mathfrak{M}_{\ast }\lVert \chi _{G_{s}^{\alpha
,t}}h\rVert _{L^{p}(\sigma )}\ .  \label{e.p'appsecond}
\end{equation}%
{\ For our use below, we note that this conclusion holds independent of the
assumption, imposed in \eqref{e.xPjk}, that $i\in \mathcal{J}_{k}^{t}$.}

With this notation, and using%
\begin{equation*}
\frac{1}{\left\vert B_{i}^{k+2}\right\vert _{\sigma }}\int_{B_{i}^{k+2}}%
\left\vert b_{r}\right\vert \sigma \leq \frac{1}{\left\vert
B_{i}^{k+2}\right\vert _{\sigma }}\int_{B_{i}^{k+2}}\left\vert f\right\vert
\sigma +\left\vert A_{G_{r}^{\alpha ,t+1}}\right\vert \lesssim \gamma ^{t+2},
\end{equation*}%
the summands in the definition of $V_{s}^{t}(2)$, as given in \eqref{e.V2def}%
, are%
\begin{eqnarray}
&&\sum_{\ell }\sum_{r\in \mathcal{K}_{s}^{\alpha ,t}:\ell \left( r\right)
=\ell }\sum_{i\in \mathcal{J}_{k}^{t}:B_{i}^{k+2}\subset G_{r}^{\alpha ,t+1}}%
\mathbf{P}\left( B_{i}^{k+2},\chi _{E_{j}^{k}\cap T_{\ell }}\omega \right)
\left( \int_{B_{i}^{k+2}}\sigma \right) \left\{ \frac{1}{\left\vert
B_{i}^{k+2}\right\vert _{\sigma }}\int_{B_{i}^{k+2}}\left\vert
b_{r}\right\vert \sigma \right\}   \label{e.V23} \\
&\leq &\gamma ^{t+2}\sum_{\ell }\int \sum_{r\in \mathcal{K}_{s}^{\alpha
,t}:\ell \left( r\right) =\ell }\sum_{i\in \mathcal{J}_{k}^{t}:B_{i}^{k+2}%
\subset G_{r}^{\alpha ,t+1}}\mathbf{P}\left( B_{i}^{k+2},\chi
_{E_{j}^{k}\cap T_{\ell }}\omega \right) \chi _{B_{i}^{k+2}}\sigma \ \ \ \ \ 
\text{(since }i\in \mathcal{J}_{k}^{t}\text{)}  \notag \\
&\leq &\gamma ^{t+2}\int_{G_{s}^{\alpha ,t}}\mathbf{P}_{j}^{k}\left( \omega
\right) \;\sigma =\gamma ^{t+2}\int_{E_{j}^{k}}\left( \mathbf{P}%
_{j}^{k}\right) ^{\ast }\left( \chi _{G_{s}^{\alpha ,t}}\sigma \right)
\;\omega .  \notag
\end{eqnarray}

We then have from \eqref{e.V2def} and \eqref{e.V23} by the argument for term 
$I\!I_{s}^{t}(2)$, 
\begin{eqnarray*}
\sum_{\left( t,s\right) \in \mathbb{L}^{\alpha }}V_{s}^{t}(2) &\leq &C\gamma
^{2p}\sum_{\left( t,s\right) \in \mathbb{L}^{\alpha }}\gamma
^{pt}\sum_{(k,j)\in \mathbb{I}_{s}^{\alpha ,t}}R_{j}^{k}\left\vert
\int_{Q_{j}^{k}}\left( \mathbf{P}_{j}^{k}\right) ^{\ast }\left( \chi
_{G_{s}^{\alpha ,t}}\sigma \right) \;\omega \right\vert ^{p} \\
&\leq &C\gamma ^{2p}\sum_{\left( t,s\right) \in \mathbb{L}^{\alpha }}\gamma
^{pt}\int \left\vert \mathcal{M}_{\omega }\left( \chi _{G_{s}^{\alpha
,t}}\sum_{(\ell ,i)\in \mathbb{I}_{s}^{\alpha ,t}}\left( \mathbf{P}%
_{i}^{\ell }\right) ^{\ast }\left( \chi _{G_{s}^{\alpha ,t}}\sigma \right)
\right) \right\vert ^{p}\;\omega \\
&\leq &C\gamma ^{2p}\sum_{\left( t,s\right) \in \mathbb{L}^{\alpha }}\gamma
^{pt}\int_{G_{s}^{\alpha ,t}}\left[ \sum_{(\ell ,i)\in \mathbb{I}%
_{s}^{\alpha ,t}}\left( \mathbf{P}_{i}^{\ell }\right) ^{\ast }\left( \chi
_{G_{s}^{\alpha ,t}}\sigma \right) \right] ^{p}\;\omega \\
&\leq &C\gamma ^{2p}\mathfrak{M}_{\ast }^{p}\sum_{\left( t,s\right) \in 
\mathbb{L}^{\alpha }}\gamma ^{pt}\sum_{\ell }\left\vert G_{s}^{\alpha
,t}\right\vert _{\sigma }\leq C\gamma ^{2p}\mathfrak{M}_{\ast }^{p}\int
\lvert f\rvert ^{p}\;\sigma \,.
\end{eqnarray*}%
In last lines we are using the boundedness \eqref{M2weightdual} of the
maximal operator.
\end{proof}

%%%%%%%%%%%%%%%%%%%%%%%%%%%%%% SUBSUBSECTION SUBSUBSECTION SUBSUBSECTION 

We will use the same method to treat term $V(1) $ and term $VI(1) $ below,
and we postpone the argument for now.

\subsubsection{The bound for $I\!V(2)[a] $}

We turn to the term defined in \eqref{e.iv2b}. In case (a) the cubes $%
B_{i}^{k+2}$ satisfy%
\begin{equation*}
G_{r}^{\alpha ,t+1}\subset B_{i}^{k+2}\text{ whenever }G_{r}^{\alpha
,t+1}\cap B_{i}^{k+2}\neq \emptyset .
\end{equation*}%
and so recalling that $i\in \mathcal{J}_{k}^{t}$ and $i\in \QTR{up}{Case}%
\left( a\right) $, we obtain from \eqref{e.replace'} that%
\begin{eqnarray*}
I\!V_{s}^{t}(2)[a] &=&\sum_{(k,j)\in \mathbb{I}_{s}^{\alpha
,t}}R_{j}^{k}\left\vert \sum_{i\in \mathcal{J}_{k}^{t}\text{ and }i\in 
\QTR{up}{Case}\left( a\right) }\sum_{r:G_{r}^{\alpha ,t+1}\subset
B_{i}^{k+2}}\int_{G_{r}^{\alpha ,t+1}}\left( L^{\ast }\chi
_{E_{j}^{k}}\omega \right) b_{r}\sigma \right\vert ^{p} \\
&\leq &C\sum_{(k,j)\in \mathbb{I}_{s}^{\alpha ,t}}R_{j}^{k}\left\vert
\sum_{i\in \QTR{up}{Case}\left( a\right) }\sum_{r:G_{r}^{\alpha ,t+1}\subset
B_{i}^{k+2}}\mathbf{P}\left( G_{r}^{\alpha ,t+1},\chi _{E_{j}^{k}}\omega
\right) \int_{G_{r}^{\alpha ,t+1}}\left\vert f\right\vert \sigma \right\vert
^{p} \\
&\leq &C\gamma ^{p\left( t+2\right) }\sum_{(k,j)\in \mathbb{I}_{s}^{\alpha
,t}}R_{j}^{k}\left\vert \sum_{i\in \QTR{up}{Case}\left( a\right)
}\sum_{r:G_{r}^{\alpha ,t+1}\subset B_{i}^{k+2}}\mathbf{P}\left(
G_{r}^{\alpha ,t+1},\chi _{E_{j}^{k}}\omega \right) \left\vert G_{r}^{\alpha
,t+1}\right\vert _{\sigma }\right\vert ^{p}.
\end{eqnarray*}%
But this last sum is identical to the estimate for the term $I\!I_{s}^{t}(2)$
used in \eqref{sum in ts} above. The estimate there thus gives%
\begin{equation}
\sum_{\left( t,s\right) \in \mathbb{L}^{\alpha }}I\!V_{s}^{t}(2)\left[ a%
\right] \leq C\gamma ^{2p}\mathfrak{M}_{\ast }^{p}\sum_{\left( t,s\right)
\in \mathbb{L}^{\alpha }}\gamma ^{pt}\left\vert G_{s}^{\alpha ,t}\right\vert
_{\sigma }\leq C\gamma ^{2p}\mathfrak{M}_{\ast }^{p}\int \left\vert
f\right\vert ^{p}\sigma ,  \label{e.IVts2a<}
\end{equation}%
which is the desired estimate.

\subsubsection{The Decomposition of $I\!V(1) $}

This term is the first term on the right hand side of \eqref{rewrite}.
Recall that for $i\in \mathcal{I}_{k}^{t}$ we have $i\in \QTR{up}{Case}(b)$
and so $B_{i}^{k+2}\subset G_{r}^{\alpha ,t+1}\subset T_{\ell \left(
r\right) }$ for some $r\in \mathcal{K}_{s}^{\alpha ,t}$. From \eqref{Bhat}
we also have $B_{i}^{k+2}\subset \widetilde{Q_{j}^{k}}$. To estimate $%
I\!V_{s}^{t}(1)$ in \eqref{rewrite}, we again apply \eqref{e.replace} to be
able to write 
\begin{eqnarray}
I\!V_{s}^{t}(1) &\leq &C\sum_{(k,j)\in \mathbb{I}_{s}^{\alpha ,t}}R_{j}^{k}%
\left[ \sum_{\ell }\sum_{i\in \mathcal{I}_{k}^{t}:B_{i}^{k+2}\subset T_{\ell
}\cap \widetilde{Q_{j}^{k}}}\left[ \int_{B_{i}^{k+2}}\left\vert L^{\ast
}\chi _{E_{j}^{k}\cap T_{\ell }}\omega \right\vert \sigma \right] \mathsf{A}%
_{i}^{k+2}\right] ^{p}  \label{e.VI2def} \\
&&{}+{}C\sum_{(k,j)\in \mathbb{I}_{s}^{\alpha ,t}}R_{j}^{k}\left[ \sum_{\ell
}\sum_{i\in \mathcal{I}_{k}^{t}:B_{i}^{k+2}\subset T_{\ell }\cap \widetilde{%
Q_{j}^{k}}}\mathbf{P}\left( B_{i}^{k+2},\chi _{E_{j}^{k}\cap T_{\ell
}}\omega \right) \int_{B_{i}^{k+2}}\left\vert f\right\vert \sigma \right]
^{p} \\
&=&V\!I_{s}^{t}(1)+V\!I_{s}^{t}(2).  \notag
\end{eqnarray}%
We comment that we are able to dominate the averages on $B_{i}^{k+2}$ of the
bad function $b_{r}$ by $\mathsf{A}_{i}^{k+2}+\left\vert A_{r}^{\alpha
,t+1}\right\vert \leq 2\mathsf{A}_{i}^{k+2}$, since in this case $i\in 
\mathcal{I}_{k}^{t}$, see \eqref{e.IJdef}, and this implies that the average
of $\left\vert b_{r}\right\vert =\left\vert f-A_{r}^{\alpha ,t+1}\right\vert 
$ over the cube $B_{i}^{k+2}$ is dominated by 
\begin{equation*}
\mathsf{A}_{i}^{k+2}+\left\vert A_{r}^{\alpha ,t+1}\right\vert \leq \mathsf{A%
}_{i}^{k+2}+\gamma ^{t+2}<2\mathsf{A}_{i}^{k+2}.
\end{equation*}

\subsubsection{The bound for $V\!I(2)$}

We claim that 
\begin{equation}
V\!I_{s}^{t}(2)\leq C\mathfrak{M}_{\ast }^{p}\sideset {} {^{s,t,\mathcal I} }%
\sum_{k,i}\left\vert B_{i}^{k+2}\right\vert _{\sigma }\left( \mathsf{A}%
_{i}^{k+2}\right) ^{p}.  \label{e.VI2<}
\end{equation}%
{\ In this display, the sum on the right is over all pairs of integers $%
k,i\in \mathcal{I}_{k}^{t}$ such that $B_{i}^{k+2}\subset T_{\ell }\cap 
\widetilde{Q_{j}^{k}}$, for some $\ell ,j$, with $(k,j)\in \mathbb{I}%
_{s}^{\alpha ,t}$. (Below, we will need a similar sum, with the condition $%
i\in \mathcal{I}_{k}^{t}$ replaced by $i\in \mathcal{J}_{k}^{t}$ and $i\in 
\QTR{up}{Case}(b)$.) This is a provisional bound, one that requires
additional combinatorial assertations in \S \ref{combinatorics} to control.}

\begin{proof}
The term $V\!I_{s}^{t}(2)$ can be handled the same way as the term $%
V_{s}^{t}(2)$, see \eqref{e.V2<}, with these two changes. {First, in the
definition of $\mathbf{P}_{j}^{k}$, we replace $\mathcal{J}_{k}^{t}$ by $%
\mathcal{I}_{k}^{t}$, and second, we use the function } 
\begin{equation*}
h={\sideset {} {^{s,t,\mathcal I} }\sum_{k,i}}\mathsf{A}_{i}^{k+2}\chi
_{B_{i}^{k+2}}
\end{equation*}%
in \eqref{e.p'appsecond}. That argument then obtains 
\begin{equation}
\left\Vert \chi _{G_{s}^{\alpha ,t}}\sum_{k,j}(\mathbf{P}_{j}^{k})^{\ast
}(\chi _{G_{s}^{\alpha ,t}}h\sigma )\right\Vert _{L^{p}(\omega )}^{p}\leq C%
\mathfrak{M}_{\ast }^{p}{\sideset {} {^{s,t,\mathcal I} }\sum_{k,i}}%
\left\vert B_{i}^{k+2}\right\vert _{\sigma }\left( \mathsf{A}%
_{i}^{k+2}\right) ^{p}.  \label{e.bnm}
\end{equation}%
Here we are using the bounded overlap of the cubes $B_{i}^{k+2}$ given in
Claim \ref{bounded triples}, along with the fact recorded in (\ref{record})
that for fixed $\left( k+2,i\right) $, only a bounded number of $j$ satisfy $%
B_{i}^{k+2}\subset \widetilde{Q_{j}^{k}}$. {\ Claim \ref{bounded triples}
applies in this setting, as we are in a subcase of the analysis of $I\!V$.}
We then use the {\ universal Maximal Function bound.} 
\begin{eqnarray*}
V\!I_{s}^{t}(2) &=&\sum_{(k,j)\in \mathbb{I}_{s}^{\alpha ,t}}R_{j}^{k}\left[
\sum_{\ell }\sum_{i\in \mathcal{I}_{k}^{t}:B_{i}^{k+2}\subset T_{\ell }\cap 
\widetilde{Q_{j}^{k}}}\mathbf{P}\left( B_{i}^{k+2},\chi _{E_{j}^{k}\cap
T_{\ell }}\omega \right) \left( \int_{B_{i}^{k+2}}\sigma \right) \mathsf{A}%
_{i}^{k+2}\right] ^{p} \\
&=&C\sum_{(k,j)\in \mathbb{I}_{s}^{\alpha ,t}}R_{j}^{k}\Bigl\lvert%
\int_{Q_{j}^{k}}(\mathbf{P}_{j}^{k})^{\ast }(h\sigma )\;\omega \Bigr\rvert%
^{p} \\
&\leq &C\int \Bigl[\mathcal{M}_{\omega }\Bigl(\chi _{G_{s}^{\alpha
,t}}\sum_{(k,j)\in \mathbb{I}_{s}^{\alpha ,t}}(\mathbf{P}_{j}^{k})^{\ast
}(\chi _{G_{s}^{\alpha ,t}}h\sigma )\Bigr)\Bigr]^{p}\;\omega  \\
&\leq &C\int \left[ \chi _{G_{s}^{\alpha ,t}}\sum_{(k,j)\in \mathbb{I}%
_{s}^{\alpha ,t}}(\mathbf{P}_{j}^{k})^{\ast }(\chi _{G_{s}^{\alpha
,t}}h\sigma )\right] ^{p}\;\omega 
\end{eqnarray*}%
{\ In view of \eqref{e.bnm}, this completes the proof of the provisional
estimate \eqref{e.VI2<}.}
\end{proof}

%%%%%%%%%%%%%%%%%%%%%%%%%%%%%% PROOF PROOF PROOF

%%%%%%%%%%%%%%%%%%%%%%%%%%%%%% SUBSUBSECTION SUBSUBSECTION SUBSUBSECTION 

\subsubsection{The bound for $V\!I(1)$}

Recall the definition of $V\!I(1)$ from \eqref{e.VI2def}, and also from %
\eqref{Bhat} that $B_{i}^{k+2}\subset \widetilde{Q_{j}^{k}}$ whenever $%
B_{i}^{k+2}\cap \widetilde{Q_{j}^{k}}\neq \emptyset $. We claim that 
\begin{equation}
V\!I_{s}^{t}(1)\leq C\mathfrak{T}_{\ast }^{p}\sideset {} {^{s,t,\mathcal I} }%
\sum_{k,i}\left\vert B_{i}^{k+2}\right\vert _{\sigma }\left( \mathsf{A}%
_{i}^{k+2}\right) ^{p}.  \label{e.VI1<}
\end{equation}%
{\ The notation here is as in \eqref{e.VI2<}, but since $i\in \mathcal{I}%
_{k}^{t}$ implies $i$ belong to $\QTR{up}{Case}(b)$, the sum over the right
is over $k,i\in \mathcal{I}_{k}^{t}$ such that $B_{i}^{k+2}\subset
G_{r}^{\alpha ,t+1}\subset T_{\ell \left( r\right) }\cap \widetilde{Q_{j}^{k}%
}$, for some integers $j,r$, with $(k,j)\in \mathbb{I}_{s}^{\alpha ,t}$. As
with \eqref{e.VI2<}, this is a provisional estimate. }

\begin{proof}
We first estimate the sum in $i$ inside term $V\!I_{s}^{t}(1)$. Recall that
the sum in $i$ is over those $i$ such that $B_{i}^{k+2}\subset G_{r}^{\alpha
,t+1}\subset T_{\ell }$ for some $r$ with $\ell =\ell \left( r\right) $, and
where $\left\{ T_{\ell }\right\} _{\ell }$ is the set of maximal cubes in
the collection $\left\{ 3G_{r}^{\alpha ,t+1}:r\in \mathcal{K}_{s}^{\alpha
,t}\right\} $. See the discussion at \eqref{e.K...}, and \eqref{finover}. We
will write $\ell \left( i\right) =\ell \left( r\right) $ when $%
B_{i}^{k+2}\subset G_{r}^{\alpha ,t+1}$. It is also important to note that
the sum in $i$ deriving from term $I\!V_{s}^{t}$ is also restricted to those 
$i$ such that $B_{i}^{k+2}\subset \widetilde{Q_{j}^{k}}$ by \eqref{Bhat}, so
that altogether, $B_{i}^{k+2}\subset T_{\ell }\cap \widetilde{Q_{j}^{k}}$.
We have 
\begin{eqnarray*}
&&\left\vert \sum_{i}\left[ \int_{B_{i}^{k+2}}\left\vert L^{\ast }\chi
_{E_{j}^{k}\cap T_{\ell \left( i\right) }}\omega \right\vert \sigma \right] 
\mathsf{A}_{i}^{k+2}\right\vert ^{p} \\
&&\qquad \leq \sum_{i}\left\vert B_{i}^{k+2}\right\vert _{\sigma }(\mathsf{A}%
_{i}^{k+2})^{p}\left[ \sum_{i}\left\vert B_{i}^{k+2}\right\vert _{\sigma
}^{1-p^{\prime }}\left[ \int_{B_{i}^{k+2}}\left\vert L^{\ast }\chi
_{E_{j}^{k}\cap T_{\ell \left( i\right) }}\omega \right\vert \sigma \right]
^{p^{\prime }}\right] ^{p-1} \\
&&\qquad \leq \sum_{i}\left\vert B_{i}^{k+2}\right\vert _{\sigma }(\mathsf{A}%
_{i}^{k+2})^{p}\left[ \sum_{i}\int_{B_{i}^{k+2}}\left\vert L^{\ast }\chi
_{E_{j}^{k}\cap T_{\ell \left( i\right) }}\omega \right\vert ^{p^{\prime
}}\sigma \right] ^{p-1} \\
&&\qquad \leq C\sum_{i}\left\vert B_{i}^{k+2}\right\vert _{\sigma }(\mathsf{A%
}_{i}^{k+2})^{p}\left[ \sum_{\ell }\sum_{i:\ell \left( i\right) =\ell
}\int_{B_{i}^{k+2}}\left\vert L^{\ast }\chi _{E_{j}^{k}\cap T_{\ell \left(
i\right) }}\omega \right\vert ^{p^{\prime }}\sigma \right] ^{p-1}.
\end{eqnarray*}%
Now we will apply the form \eqref{unif'} of (\ref{Tsharpomega}) with $g=\chi
_{E_{j}^{k}\cap T_{\ell }}$ and $Q$ chosen to be either $T_{\ell }$ or $%
\widetilde{Q_{j}^{k}}$ depending on the relative positions of $T_{\ell }$
and $\widetilde{Q_{j}^{k}}$. Since $T_{\ell }$ is a triple of a cube in the
grid $\mathcal{D}^{\alpha }$ and $\widetilde{Q_{j}^{k}}$ is a cube in the
grid $\mathcal{D}^{\alpha }$, we must have either%
\begin{equation*}
\widetilde{Q_{j}^{k}}\subset T_{\ell }\text{ or }T_{\ell }\subset 3%
\widetilde{Q_{j}^{k}}.
\end{equation*}%
If $\widetilde{Q_{j}^{k}}\subset T_{\ell }$ we choose $Q$ in \eqref{unif'}
to be $\widetilde{Q_{j}^{k}}$ and note that\ by bounded overlap of Whitney
cubes, there are only a bounded number of such cases. If on the other hand $%
T_{\ell }\subset 3\widetilde{Q_{j}^{k}}$, then we choose $Q$ to be $T_{\ell }
$ and note that the cubes $T_{\ell }$ have bounded overlap. This gives%
\begin{equation*}
\sum_{\ell }\sum_{i:\ell \left( i\right) =\ell }\int_{B_{i}^{k+2}}\left\vert
L^{\ast }\chi _{E_{j}^{k}\cap T_{\ell \left( i\right) }}\omega \right\vert
^{p^{\prime }}\sigma \lesssim \mathfrak{T}_{\ast }^{p}\left\vert 3\widetilde{%
Q_{j}^{k}}\right\vert _{\omega },
\end{equation*}%
and hence%
\begin{equation*}
\left\vert \sum_{i}\left[ \int_{B_{i}^{k+2}}\left\vert L^{\ast }\chi
_{E_{j}^{k}\cap T_{\ell \left( i\right) }}\omega \right\vert \sigma \right] 
\mathsf{A}_{i}^{k+2}\right\vert ^{p}\leq C\mathfrak{T}_{\ast
}^{p}\sum_{i}\left\vert B_{i}^{k+2}\right\vert _{\sigma }(\mathsf{A}%
_{i}^{k+2})^{p}\left\vert NQ_{j}^{k}\right\vert _{\omega }^{p-1},
\end{equation*}%
since $3\widetilde{Q_{j}^{k}}\subset NQ_{j}^{k}$ by (\ref{Qtilda contained}%
). With this we obtain,%
\begin{eqnarray}
V\!I_{s}^{t}(1) &\leq &C\mathfrak{T}_{\ast }^{p}\sum_{(k,j)\in \mathbb{I}%
_{s}^{\alpha ,t}}R_{j}^{k}\sum_{i\in \mathcal{I}_{k}^{t}}\left\vert
B_{i}^{k+2}\right\vert _{\sigma }(\mathsf{A}_{i}^{k+2})^{p}\left\vert
NQ_{j}^{k}\right\vert _{\omega }^{p-1}  \label{e.VI.} \\
&\leq &C\mathfrak{T}_{\ast }^{p}\sideset {^\ast} {^{s,t,\mathcal I} }%
\sum_{k,i}\left\vert B_{i}^{k+2}\right\vert _{\sigma }(\mathsf{A}%
_{i}^{k+2})^{p},  \notag
\end{eqnarray}%
where we are using $R_{j}^{k}\left\vert NQ_{j}^{k}\right\vert _{\omega
}^{p-1}\leq 1$ and (\ref{record}) in the final line.
\end{proof}

\subsubsection{The bound for $V(1) $}

We will use the same method as in the estimate for term $V\!I(1)$ above to
obtain 
\begin{equation}
\sum_{\left( t,s\right) \in \mathbb{L}^{\alpha }}V_{s}^{t}(1)\leq C\mathfrak{%
T}_{\ast }^{p}\gamma ^{2p}\lVert f\rVert _{L^{p}(\sigma )}^{p}.
\label{e.V1<}
\end{equation}%
Recall {from \eqref{e.3456},} that $V_{s}^{t}(1)$ is given by 
\begin{equation}
\sum_{(k,j)\in \mathbb{I}_{s}^{\alpha ,t}}R_{j}^{k}\left\vert \sum_{r\in 
\mathcal{K}_{s}^{\alpha ,t}}\sum_{i\in \mathcal{J}_{k}^{t}:B_{i}^{k+2}%
\subset G_{r}^{\alpha ,t+1}}\left[ \int_{B_{i}^{k+2}}\left( L^{\ast }\chi
_{E_{j}^{k}\cap T_{\ell \left( r\right) }}\omega \right) \sigma \right]
c_{i}^{k+2}\left( \left\vert A_{r}^{\alpha ,t+1}\right\vert +\mathsf{A}%
_{i}^{k+2}\right) \right\vert ^{p}.  \notag
\end{equation}%
The main difference here, as opposed to the previous estimate, is that $i\in 
\mathcal{J}_{k}^{t}$ rather than in $\mathcal{I}_{k}^{t}$, see %
\eqref{e.IJdef}. As a result, we have the estimate%
\begin{equation}
\left\vert A_{r}^{\alpha ,t+1}\right\vert +\mathsf{A}_{i}^{k+2}\lesssim
\gamma ^{t+2},  \label{new estimate}
\end{equation}%
instead of $\left\vert A_{r}^{\alpha ,t+1}\right\vert +\mathsf{A}%
_{i}^{k+2}\lesssim \mathsf{A}_{i}^{k+2}$, which holds when $i\in \mathcal{I}%
_{k}^{t}$.

\begin{proof}[Proof of (\protect\ref{e.V1<})]
We follow the argument leading up to and including (\ref{e.VI.}) in the
estimate for term $VI(1)$ above, but using instead (\ref{new estimate}). The
result is as below, {where we are using the notation of \eqref{e.VI2<}, with
the condition $i\in \mathcal{I}_{k}^{t}$ replaced by $i\in \mathcal{J}%
_{k}^{t}$ and $i\in \QTR{up}{Case}(b)$, and so we use an asterix and $%
\mathcal{J}$ in the notation below.} 
\begin{equation*}
V_{s}^{t}(1)\leq C\mathfrak{T}_{\ast }^{p}{\sideset {^\ast} {^{s,t,\mathcal
J} }\sum_{k,i}}\left\vert B_{i}^{k+2}\right\vert _{\sigma }\left( \gamma
^{t+2}\right) ^{p}.
\end{equation*}%
Now we collect those cubes $B_{i}^{k+2}$ that lie in a given cube $%
G_{r}^{\alpha ,t+1}$ and write the right hand side above as a constant times 
\begin{equation*}
\mathfrak{T}_{\ast }^{p}\gamma ^{\left( t+2\right) p}\sum_{r\in \mathcal{K}
_{s}^{\alpha ,t}} \ \quad {\sideset {^\ast} {^{s,t,\mathcal J} }\sum_{k,i \;:\; 
B_{i}^{k+2}\subset G_{r}^{\alpha ,t+1}}}\left\vert B_{i}^{k+2}\right\vert
_{\sigma }:=\mathfrak{T}_{\ast }^{p}\gamma ^{\left( t+2\right) p}\sum_{r\in 
\mathcal{K}_{s}^{\alpha ,t}}\mathcal{S}_{s,r}^{\alpha ,t}.
\end{equation*}%
By Claim \ref{bounded triples}, {\ which applies as we are in a subcase of $%
I\!V$,} we have $\mathcal{S}_{s,r}^{\alpha ,t}\leq C\left\vert G_{r}^{\alpha
,t+1}\right\vert _{\sigma }$, and it follows that%
\begin{equation*}
V_{s}^{t}(1)\leq C\mathfrak{T}_{\ast }^{p}\gamma ^{(t+2)p}\sum_{r\in 
\mathcal{K}_{s}^{\alpha ,t}}\left\vert G_{r}^{\alpha ,t+1}\right\vert
_{\sigma }\leq C\mathfrak{T}_{\ast }^{p}\gamma ^{(t+2)p}\left\vert
G_{s}^{\alpha ,t}\right\vert _{\sigma },
\end{equation*}%
and hence from (\ref{e.MGst}) that%
\begin{equation*}
\sum_{\left( t,s\right) \in \mathbb{L}^{\alpha }}V_{s}^{t}(1)\leq C\mathfrak{%
T}_{\ast }^{p}{\gamma ^{2p}}\sum_{\left( t,s\right) \in \mathbb{L}^{\alpha
}}\gamma ^{tp}\left\vert G_{s}^{\alpha ,t}\right\vert _{\sigma }\leq C%
\mathfrak{T}_{\ast }^{p}{\gamma ^{2p}}\rVert f\rVert _{L^{p}(\sigma )}^{p}.
\end{equation*}
\end{proof}

\subsubsection{The final combinatorial arguments\label{combinatorics}}

Our final estimate in the proof of \eqref{e.IV<} is to dominate by $C\int
\left\vert f\right\vert ^{p}d\sigma $ the sum of the right hand sides of (%
\ref{e.VI2<}) and \eqref{e.VI1<} over $\left( t,s\right) \in \mathbb{L}
^{\alpha }$, namely 
\begin{equation}
\sum_{\left( t,s\right) \in \mathbb{L}^{\alpha }} {\sideset {}
{^{s,t,\mathcal I} }\sum_{k,i} } \leq C\int \left\vert f\right\vert
^{p}d\sigma .  \label{final}
\end{equation}%
The proof of \eqref{final} will require combinatorial facts related to the
principal cubes, and the definition of the collection $\mathbb{G}^{\alpha }$
in \eqref{e.GdefAlpha}. Also essential is the implementation of the shifted
dyadic grids. We now detail the arguments.

%%%%%%%%%%%%%%%%%%%%%%%%%%%%%% SUBSUBSECTION SUBSUBSECTION SUBSUBSECTION 

\begin{definition}
\label{type}We say that a cube $B_{i}^{k+2}$ satisfying the defining
condition in ${V\!I_{s}^{t}}(1)$, namely 
\begin{eqnarray*}
&&\text{there is }\left( k,j\right) \in \mathbb{I}_{s}^{\alpha ,t}=\mathbb{G}%
^{\alpha }\cap \mathbb{H}_{s}^{\alpha ,t}\text{ such that} \\
&&B_{i}^{k+2}\subset \widetilde{Q_{j}^{k}}\text{ and} \\
&&B_{i}^{k+2}\subset \text{ some }G_{r}^{\alpha ,t+1}\subset G_{s}^{\alpha
,t}\text{ satisfying }\mathsf{A}_{i}^{k+2}>\gamma ^{t+2},
\end{eqnarray*}%
is a \emph{final type} cube for the pair $\left( t,s\right) \in \mathbb{L}%
^{\alpha }$ generated from $Q_{j}^{k}$.
\end{definition}

The collection of cubes 
\begin{eqnarray*}
\mathcal{F} &\equiv &\left\{ B_{i}^{k+2}:B_{i}^{k+2}\text{ is a \emph{final
type} cube generated from some }Q_{j}^{k}\right. \\
&&\left. \text{with }\left( k,j\right) \in \mathbb{I}_{s}^{\alpha ,t}\text{\
for some pair }\left( t,s\right) \in \mathbb{L}^{\alpha }\right\} .
\end{eqnarray*}%
satisfies the following three properties:

\begin{description}
\item[Property 1] $\mathcal{F }$ is a nested grid in the sense that given
any two \emph{distinct} cubes in $\mathcal{F }$, either one is strictly
contained in the other, or they are disjoint (ignoring boundaries).

\item[Property 2] If $B_{i}^{k+2}$ and $B_{i^{\prime }}^{k^{\prime }+2}$ are
two \emph{distinct} cubes in $\mathcal{F}$ with $B_{i^{\prime }}^{k^{\prime
}+2}\subsetneqq B_{i}^{k+2}$, and $k,k^{\prime }$ have the same parity, then%
\begin{equation*}
\mathsf{A}_{i^{\prime }}^{k^{\prime }+2}>\gamma \mathsf{A}_{i}^{k+2}.
\end{equation*}

\item[Property 3] A given cube $B_{i}^{k+2}$ can occur at most a bounded
number of times in the grid $\mathcal{F}$.
\end{description}

\begin{proof}[Proof of Properties 1, 2 and 3]
Property 1 is obvious from the properties of the dyadic shifted grid $%
\mathcal{D}^{\alpha }$. Property 3 follows from the `Bounded Occurrence of
Cubes' noted above. So we turn to Property 2. It is this Property that
prompted the use of the shifted dyadic grids.

Indeed, since $B_{i^{\prime }}^{k^{\prime }+2}\subsetneqq B_{i}^{k+2}$, it
follows from the nested property \eqref{nested} that $k^{\prime }>k$. By
Definition \ref{type} there are cubes $Q_{j^{\prime }}^{k^{\prime }}$ and $%
Q_{j}^{k}$ satisfying%
\begin{equation*}
B_{i^{\prime }}^{k^{\prime }+2}\subset \widetilde{Q_{j^{\prime }}^{k^{\prime
}}}\qquad \hbox{and}\qquad B_{i}^{k+2}\subset \widetilde{Q_{j}^{k}},
\end{equation*}%
and also cubes $G_{s^{\prime }}^{\alpha ,t^{\prime }}\subset G_{s}^{\alpha
,t}$ such that $\left( k^{\prime },j^{\prime }\right) \in \mathbb{I}%
_{s^{\prime }}^{\alpha ,t^{\prime }}$ and $\left( k,j\right) \in \mathbb{I}%
_{s}^{\alpha ,t}$ with $\left( t^{\prime },s^{\prime }\right) ,\left(
t,s\right) \in \mathbb{L}^{\alpha }$, so that in particular,%
\begin{equation*}
\widetilde{Q_{j^{\prime }}^{k^{\prime }}}\subset G_{s^{\prime }}^{\alpha
,t^{\prime }}\text{ and }\widetilde{Q_{j}^{k}}\subset G_{s}^{\alpha ,t}.
\end{equation*}%
Now $k^{\prime }\geq k+2$ and in the extreme case where $k^{\prime }=k+2$,
it follows that the $\mathcal{D}^{\alpha }$-cube $\widetilde{Q_{j^{\prime
}}^{k^{\prime }}}$ is one of the cubes $B_{\ell }^{k+2}$, so in fact it must
be $B_{i}^{k+2}$ since $B_{i^{\prime }}^{k^{\prime }+2}\subset B_{i}^{k+2}$.
Thus we have%
\begin{equation*}
B_{i^{\prime }}^{k^{\prime }+2}\subset \widetilde{Q_{j^{\prime }}^{k^{\prime
}}}=B_{i}^{k+2}.
\end{equation*}%
In the general case $k^{\prime }\geq k+2$ we have instead%
\begin{equation*}
B_{i^{\prime }}^{k^{\prime }+2}\subset \widetilde{Q_{j^{\prime }}^{k^{\prime
}}}\subset B_{i}^{k+2}.
\end{equation*}

Now $\mathsf{A}_{i}^{k+2}>\gamma ^{t+2}$ by Definition \ref{type}, and so
there is $t_{0}\geq t+2$ determined by the condition%
\begin{equation}
\gamma ^{t_{0}}<\mathsf{A}_{i}^{k+2}\leq \gamma ^{t_{0}+1},  \label{detcond}
\end{equation}%
and also $s_{0}$ such that%
\begin{equation*}
B_{i}^{k+2}\subset G_{s_{0}}^{\alpha ,t_{0}}\subset G_{s}^{\alpha ,t},
\end{equation*}%
where the label $\left( t_{0},s_{0}\right) $ need not be principal.
Combining inclusions we have%
\begin{equation*}
\widetilde{Q_{j^{\prime }}^{k^{\prime }}}\subset B_{i}^{k+2}\subset
G_{s_{0}}^{\alpha ,t_{0}},
\end{equation*}%
and since $\left( k^{\prime },j^{\prime }\right) \in \mathbb{I}_{s^{\prime
}}^{\alpha ,t^{\prime }}$, we obtain $G_{s^{\prime }}^{\alpha ,t^{\prime
}}\subset G_{s_{0}}^{\alpha ,t_{0}}$. Since $\left( t^{\prime },s^{\prime
}\right) \in \mathbb{L}^{\alpha }$ is a principal label, we have the key
property that%
\begin{equation}
t^{\prime }\geq t_{0}.  \label{keyprop}
\end{equation}%
Indeed, if $G_{s^{\prime }}^{\alpha ,t^{\prime }}=G_{s_{0}}^{\alpha ,t_{0}}$
then \eqref{keyprop} holds because $\left( t^{\prime },s^{\prime }\right)
\in \mathbb{L}^{\alpha }$ is a principal label, and otherwise the maximality
of $G_{s^{\prime }}^{\alpha ,t^{\prime }}$ shows that%
\begin{equation*}
\gamma ^{t_{0}}<\frac{1}{\left\vert G_{s_{0}}^{\alpha ,t_{0}}\right\vert
_{\sigma }}\int_{G_{s_{0}}^{\alpha ,t_{0}}}\left\vert f\right\vert d\sigma
\leq \gamma ^{t^{\prime }+1},\ \ \ \ \ \text{i.e. }t_{0}<t^{\prime }+1.
\end{equation*}%
Thus using \eqref{keyprop} and \eqref{detcond} we obtain%
\begin{equation*}
\mathsf{A}_{i^{\prime }}^{k^{\prime }+2}>\gamma ^{t^{\prime }+2}\geq \gamma
^{t_{0}+2}\geq \gamma \mathsf{A}_{i}^{k+2},
\end{equation*}%
which is Property 2.
\end{proof}

\bigskip

\begin{proof}[Proof of (\protect\ref{final})]
Now for $Q=B_{i}^{k+2}\in \mathcal{F}$ set 
\begin{equation*}
\mathsf{A}\left( Q\right) =\frac{1}{\left\vert Q\right\vert _{\sigma }}%
\int_{Q}\left\vert f\right\vert \sigma =\mathsf{A}_{i}^{k+2}=\frac{1}{%
\left\vert B_{i}^{k+2}\right\vert _{\sigma }}\int_{B_{i}^{k+2}}\left\vert
f\right\vert \sigma .
\end{equation*}

With the above three properties we can now prove \eqref{final} as follows.
Recall that in term $IV(1)$ we have $i\in \mathcal{I}_{k}^{t}$ which implies 
$B_{i}^{k+2}$ satisfies case (b). In the display below by $\sideset {} {
^\ast }\sum_{i}$ we mean the sum over $i$ such that $B_{i}^{k+2}$ is
contained in some $G_{r}^{\alpha ,t+1}\subset G_{s}^{\alpha ,t}$, and also
in some $\widetilde{Q_{j}^{k}}$ with $\left( k,j\right) \in \mathbb{I}%
_{s}^{\alpha ,t}$, and satisfying $\mathsf{A}_{i}^{k+2}>2^{t+2}$. The left
side of \eqref{final} is dominated by 
\begin{eqnarray*}
&&\sum_{\left( t,s\right) \in \mathbb{L}^{\alpha }}\sum_{(k,j)\in \mathbb{I}%
_{s}^{\alpha ,t}}\sideset {} { ^\ast }\sum_{i}\left\vert
B_{i}^{k+2}\right\vert _{\sigma }(\mathsf{A}_{i}^{k+2})^{p} \\
&=&\sum_{Q\in \mathcal{F}}\left\vert Q\right\vert _{\sigma }\mathsf{A}\left(
Q\right) ^{p}=\sum_{Q\in \mathcal{F}}\left\vert Q\right\vert _{\sigma }\left[
\frac{1}{\left\vert Q\right\vert _{\sigma }}\int_{Q}\left\vert f\right\vert
\sigma \right] ^{p} \\
&=&\int_{\mathbb{R}^{n}}\sum_{Q\in \mathcal{F}}\chi _{Q}\left( x\right) %
\left[ \frac{1}{\left\vert Q\right\vert _{\sigma }}\int_{Q}\left\vert
f\right\vert \sigma \right] ^{p}d\sigma (x) \\
&\leq &C\int_{\mathbb{R}^{n}}\sup_{x\in Q:Q\in \mathcal{F}}\left[ \frac{1}{%
\left\vert Q\right\vert _{\sigma }}\int_{Q}\left\vert f\right\vert \sigma %
\right] ^{p}d\sigma (x) \\
&\leq &C\int_{\mathbb{R}^{n}}\mathcal{M}_{\sigma }^{\alpha }f\left( x\right)
^{p}\sigma (dx)\leq C\int_{\mathbb{R}^{n}}\left\vert f\left( x\right)
\right\vert ^{p}d\sigma (x),
\end{eqnarray*}%
where the second to last line follows since for fixed $x\in \mathbb{R}^{n}$,
the sum 
\begin{equation*}
\sum_{Q\in \mathcal{F}}\chi _{Q}\left( x\right) \left[ \frac{1}{\left\vert
Q\right\vert _{\sigma }}\int_{Q}\left\vert f\right\vert \sigma \right] ^{p}
\end{equation*}%
is super-geometric by properties 1, 2 and 3 above, i.e. for any two distinct
cubes $Q$ and $Q^{\prime }$ in $\mathcal{F}$\ each containing $x$, the ratio
of the corresponding values is bounded away from $1$, more precisely,%
\begin{equation*}
\frac{\left[ \frac{1}{\left\vert Q\right\vert _{\sigma }}\int_{Q}\left\vert
f\right\vert \sigma \right] ^{p}}{\left[ \frac{1}{\left\vert Q^{\prime
}\right\vert _{\sigma }}\int_{Q^{\prime }}\left\vert f\right\vert \sigma %
\right] ^{p}}\notin \lbrack {\gamma ^{-p}},\gamma ^{p}),\ \ \ \ \ \gamma
\geq 2.
\end{equation*}%
This completes the proof of \eqref{final}.
\end{proof}

\section{The proof of Theorem \protect\ref{improved} on the strongly maximal
Hilbert transform}

To prove Theorem \ref{improved} we first show that in the proof of Theorem %
\ref{twoweightHaar} above, we can replace the use of the dual maximal
function inequality \eqref{M2weightdual} with the dual weighted Poisson
inequality \eqref{poissonweightedineq'} defined below. After that we will
show that in the case of standard kernels satisfying (\ref{sizeandsmoothness}%
) with $\delta \left( s\right) =s$ in dimension $n=1$, the dual weighted
Poisson inequality \eqref{poissonweightedineq'} is implied by the \emph{%
half-strengthened} $A_{p}$ condition%
\begin{equation}
\left( \int_{\mathbb{R}}\left( \frac{\left\vert Q\right\vert }{\left\vert
Q\right\vert +\left\vert x-x_{Q}\right\vert }\right) ^{p^{\prime }}d\sigma
\left( x\right) \right) ^{\frac{1}{p^{\prime }}}\left( \int_{Q}d\omega
\left( x\right) \right) ^{\frac{1}{p}}\leq \mathcal{A}_{p}\left( \omega
,\sigma \right) \left\vert Q\right\vert ,  \label{Ap half skirt}
\end{equation}%
for all intervals $Q$, together with the dual pivotal condition 
\eqref{dual
pivotal condition} of Nazarov, Treil and Volberg \cite{NTV3}, namely that%
\begin{equation}
\sum_{r=1}^{\infty }\left\vert Q_{r}\right\vert _{\sigma }\mathsf{P}\left(
Q_{r},\chi _{Q_{0}}\omega \right) ^{p^{\prime }}\leq \mathfrak{C}_{\ast
}^{p^{\prime }}\left\vert Q_{0}\right\vert _{\omega },
\label{dual pivotal condition}
\end{equation}%
holds for all decompositions of an interval $Q_{0}$ into a union of pairwise
disjoint intervals $Q_{0}=\bigcup_{r=1}^{\infty }Q_{r}$. We will assume $%
1<p\leq 2$ for this latter implication. Finally, for $p>2$, we show that (%
\ref{poissonweightedineq'}) is implied by \eqref{Ap half skirt}, 
\eqref{dual
pivotal condition} and the Poisson condition \eqref{Poisson condition}.

It follows from work in \cite{NTV3} and \cite{LaSaUr} that the strengthened $%
A_{2}$ condition \eqref{skirt Ap} is necessary for the two weight inequality
for the Hilbert transform, and also from \cite{LaSaUr} that the dual pivotal
condition \eqref{dual pivotal condition} is necessary for the dual testing
condition%
\begin{equation*}
\int_{Q}T\left( \chi _{Q}\omega \right) ^{2}d\sigma \leq C\int_{Q}d\omega ,
\end{equation*}%
for $T$ when $p=2$ and $\sigma $ is doubling. We show below that these
results extend to $1<p<\infty $. A slightly weaker result was known earlier
from work of Nazarov, Treil and Volberg - namely that the pivotal conditions
are necessary for the Hilbert transform $H$ when \emph{both} of the weights
are $\omega $ and $\sigma $ are doubling and $p=2$. However, in \cite{LaSaUr}%
, an example is given to show that \eqref{dual pivotal condition} is \emph{%
not} in general necessary for boundedness of the Hilbert transform $T$ when $%
p=2$.

Finally, we show below that when $\sigma $ is doubling, the dual weighted
Poisson inequality \eqref{poissonweightedineq'} is implied by the two weight
inequality for the Hilbert transform. Since the Poisson condition (\ref%
{Poisson condition}) is a special case of the inequality dual to (\ref%
{poissonweightedineq'}), we obtain the necessity of (\ref{Poisson condition}%
) for the two weight inequality for the Hilbert transform.

\subsection{The Poisson inequalities}

We begin working in $\mathbb{R}^{n}$ with $1<p<\infty $. Recall the
definition of the Poisson integral $\mathsf{P}\left( Q,\nu \right) $ of a
measure $\nu $ relative to a cube $Q$ given by,%
\begin{equation}
\mathsf{P}\left( Q,\nu \right) \equiv \sum_{\ell =0}^{\infty }\frac{\delta
\left( 2^{-\ell }\right) }{\left\vert 2^{\ell }Q\right\vert }\int_{2^{\ell
}Q}d\left\vert \nu \right\vert .  \label{PQnu}
\end{equation}%
We will consider here only the standard Poisson integral with $\delta \left(
s\right) =s$ in \eqref{PQnu}, and so we also suppose that $\delta \left(
s\right) =s$ in \eqref{sizeandsmoothness} above. We now fix a cube $Q_{0}$
and a collection of pairwise disjoint subcubes $\left\{ Q_{r}\right\}
_{r=1}^{\infty }$. Corresponding to these cubes we define a positive linear
operator%
\begin{equation}
\mathbb{P}\nu \left( x\right) =\sum_{r=1}^{\infty }\mathsf{P}\left(
Q_{r},\nu \right) \chi _{Q_{r}}\left( x\right) .  \label{Poisson operator}
\end{equation}

We wish to obtain \emph{sufficient} conditions for the following `dual'
weighted Poisson inequality, 
\begin{equation}
\int_{\mathbb{R}^{n}}\mathbb{P}\left( f\omega \right) \left( x\right)
^{p^{\prime }}d\sigma \left( x\right) \leq C\int_{\mathbb{R}%
^{n}}f^{p^{\prime }}d\omega \left( x\right) ,\ \ \ \ \ f\geq 0.
\label{poissonweightedineq'}
\end{equation}%
uniformly in $Q_{0}$ and pairwise disjoint subcubes $\left\{ Q_{r}\right\}
_{r=1}^{\infty }$. As we will see below, this inequality is necessary for
the two weight Hilbert transform inequality when $\sigma $ is doubling.

The reason for wanting the dual Poisson inequality (\ref%
{poissonweightedineq'}) is that in Theorem \ref{twoweightHaar} above, we can
replace the assumption \eqref{M2weightdual} on dual boundedness of the
maximal operator $\mathcal{M}$ by the dual Poisson inequality (\ref%
{poissonweightedineq'}). Indeed, this will be revealed by simple
modifications of the proof of Theorem \ref{twoweightHaar} above. In fact (%
\ref{poissonweightedineq'}) can replace \eqref{M2weightdual} in estimating
term $I\!I_{s}^{t}(2)$, as well as in the similar estimates for terms $%
V_{s}^{t}(2)$ and $V\!I_{s}^{t}(2)$. We turn now to the proofs of these
assertions before addressing the question of sufficient conditions for the
dual Poisson inequality \eqref{poissonweightedineq'}.

\subsubsection{Sufficiency of the dual Poisson inequality\label{suff P}}

We begin by demonstrating that the term $I\!I_{s}^{t}(2)$ in \eqref{e.II2<}
can be handled using the `dual' Poisson inequality (\ref%
{poissonweightedineq'}) in place of the maximal inequality (\ref%
{M2weightdual}). We are working here in $\mathbb{R}^{n}$ with $1<p<\infty $.
In fact we claim that%
\begin{equation}
\sum_{\left( t,s\right) \in \mathbb{L}^{\alpha }}I\!I_{s}^{t}(2)\leq C\gamma
^{2p}\mathfrak{P}_{\ast }^{p}\int \left\vert f\right\vert ^{p}\sigma \,,
\label{new IIts}
\end{equation}%
where $\mathfrak{P}_{\ast }$ is the norm of the dual Poisson inequality (\ref%
{poissonweightedineq'}) if we take $Q_{0}$ and its collection of pairwise
disjoint subcubes $\left\{ Q_{r}\right\} _{r=1}^{\infty }$ to be $%
G_{s}^{\alpha ,t}$ and $\left\{ G_{r}^{\alpha ,t+1}\right\} _{r\in \mathcal{K%
}_{s}^{\alpha ,t}}$. Now the maximal inequality \eqref{M2weightdual} was
used in the proof of \eqref{e.II2<} only in establishing 
\eqref{dual pivotal
substitute},%
\begin{equation*}
\lVert \chi _{G_{s}^{\alpha ,t}}\sum_{\left( k,j\right) \in \mathbb{I}%
_{s}^{\alpha ,t}}{\mathbf{P}}_{j}^{k}(\lvert g\rvert \omega )\rVert
_{L^{p^{\prime }}(\sigma )}\leq C\mathfrak{M}_{\ast }\lVert \chi
_{G_{s}^{\alpha ,t}}g\rVert _{L^{p^{\prime }}(\omega )},
\end{equation*}%
where%
\begin{equation*}
\mathbf{P}_{j}^{k}\left( \mu \right) \equiv \sum_{r\in \mathcal{K}%
_{s}^{\alpha ,t}}\mathbf{P}\left( G_{r}^{\alpha ,t+1},\chi _{E_{j}^{k}}\mu
\right) \chi _{G_{r}^{\alpha ,t+1}}.
\end{equation*}%
We now note that%
\begin{eqnarray*}
\sum_{\left( k,j\right) \in \mathbb{I}_{s}^{\alpha ,t}}{\mathbf{P}}%
_{j}^{k}(\lvert g\rvert \omega ) &=&\sum_{\left( k,j\right) \in \mathbb{I}%
_{s}^{\alpha ,t}}\sum_{r\in \mathcal{K}_{s}^{\alpha ,t}}\mathbf{P}\left(
G_{r}^{\alpha ,t+1},\chi _{E_{j}^{k}}\lvert g\rvert \omega \right) \chi
_{G_{r}^{\alpha ,t+1}} \\
&\leq &\sum_{r\in \mathcal{K}_{s}^{\alpha ,t}}\mathbf{P}\left( G_{r}^{\alpha
,t+1},\chi _{G_{s}^{\alpha ,t}}\lvert g\rvert \omega \right) \chi
_{G_{r}^{\alpha ,t+1}} \\
&=&\mathbb{P}\left( \chi _{G_{s}^{\alpha ,t}}\lvert g\rvert \omega \right)
\left( x\right) ,
\end{eqnarray*}%
proves 
\begin{equation*}
\lVert \chi _{G_{s}^{\alpha ,t}}\sum_{k,j}{\mathbf{P}}_{j}^{k}(\lvert
g\rvert \omega )\rVert _{L^{p^{\prime }}(\sigma )}\leq C\mathfrak{P}_{\ast
}\lVert \chi _{G_{s}^{\alpha ,t}}g\rVert _{L^{p^{\prime }}(\omega )},
\end{equation*}%
which yields \eqref{new IIts} as before.

The terms $V(2) $ and $V\!I(2) $ are handled similarly. Indeed, \eqref{k
bounded} yields the following analogue of (\ref{sum of Pkj}),%
\begin{equation*}
\sum_{(k,j)\in \mathbb{I}_{s}^{\alpha ,t}}\mathbf{P}_{j}^{k}\left( \mu
\right) \leq C\chi _{G_{s}^{\alpha ,t}}\mathbb{P}\left( \chi _{G_{s}^{\alpha
,t}}\mu \right) ,
\end{equation*}%
from which the arguments above yield both \eqref{e.V2<} and \eqref{e.VI2<}
with $\mathfrak{M}_{\ast }$ replaced by $\mathfrak{P}_{\ast }$.

\subsubsection{Sufficient conditions for Poisson inequalities}

We continue to work in $\mathbb{R}^{n}$ with $1<p<\infty $. We note that (%
\ref{poissonweightedineq'}) can be rewritten%
\begin{equation*}
\sum_{r=1}^{\infty }\left\vert Q_{r}\right\vert _{\sigma }\mathsf{P}\left(
Q_{r},f\omega \right) ^{p^{\prime }}\leq C\int_{\mathbb{R}^{n}}f^{p^{\prime
}}d\omega ,\ \ \ \ \ f\geq 0,
\end{equation*}%
and this latter inequality can then be expressed in terms of the Poisson
operator $\mathbb{P}_{+}$ in the upper half space $\mathbb{R}_{+}^{n+1}$
given by%
\begin{equation*}
\mathbb{P}_{+}\left( f\omega \right) \left( x,t\right) =\int_{\mathbb{R}%
^{n}}P_{t}\left( x-y\right) f\left( y\right) d\omega \left( y\right) .
\end{equation*}%
Indeed, let $Z_{r}=\left( x_{Q_{r}},\ell \left( Q_{r}\right) \right) $ be
the point in $\mathbb{R}_{+}^{n+1}$ that lies above the center $x_{Q_{r}}$
of $Q_{r}$ at a height equal to the side length $\ell \left( Q_{r}\right) $
of $Q_{r}$. Define an atomic measure $ds$ in $\mathbb{R}_{+}^{n+1}$ by%
\begin{equation}
ds\left( x,t\right) =\sum_{r=1}^{\infty }\left\vert Q_{r}\right\vert
_{\sigma }\delta _{Z_{r}}\left( x,t\right) .  \label{w}
\end{equation}%
Then \eqref{poissonweightedineq'} is equivalent to the inequality (this is
where we use $\delta \left( s\right) =s$),%
\begin{equation}
\int_{\mathbb{R}_{+}^{n+1}}\mathbb{P}_{+}\left( f\omega \right) \left(
x,t\right) ^{p^{\prime }}ds\left( x,t\right) \leq C\int_{\mathbb{R}%
^{n}}f^{p^{\prime }}d\omega \left( x\right) ,\ \ \ \ \ f\geq 0.
\label{upper}
\end{equation}

We can use Theorem 2 in \cite{Saw2} to characterize this latter inequality
in terms of testing conditions over $\mathbb{P}_{+}$ and its dual $\mathbb{P}%
_{+}^{\ast }$ given by%
\begin{equation*}
\mathbb{P}_{+}^{\ast }\left( gw\right) \left( x,t\right) =\int_{\mathbb{R}%
_{+}^{n+1}}P_{t}\left( y-x\right) g\left( x,t\right) dw\left( x,t\right) .
\end{equation*}%
Let $\widehat{Q}$ denote the cube in $\mathbb{R}_{+}^{n+1}$ with $Q$ as a
face. Theorem 2 in \cite{Saw2} yields the following.

\begin{theorem}
\label{testing Poisson}The Poisson inequality \eqref{poissonweightedineq'}
holds for given data $Q_{0}$ and $\left\{ Q_{r}\right\} _{r=1}^{\infty }$ 
\emph{if and only if} the measure $s$ in \eqref{w} satisfies%
\begin{eqnarray*}
\int_{\mathbb{R}_{+}^{n+1}}\mathbb{P}_{+}\left( \chi _{Q}\omega \right)
^{p^{\prime }}ds &\leq &C\int_{Q}d\omega ,\ \ \ \ \ \text{for all cubes }%
Q\in \mathcal{D}\text{,} \\
\int_{\mathbb{R}^{n}}\mathbb{P}_{+}^{\ast }\left( t^{p^{\prime }-1}\chi _{%
\widehat{Q}}ds\right) ^{p}d\omega &\leq &C\int_{\widehat{Q}}t^{p^{\prime
}}ds\ \ \ \ \ \text{for all cubes }Q\in \mathcal{D}\text{.}
\end{eqnarray*}
\end{theorem}

Note that%
\begin{equation*}
\int_{\mathbb{R}_{+}^{n+1}}\mathbb{P}_{+}\left( \chi _{Q}\omega \right)
^{p^{\prime }}ds\approx \sum_{r=1}^{\infty }\left\vert Q_{r}\right\vert
_{\sigma }\mathsf{P}\left( Q_{r},\chi _{Q}\omega \right) ^{p^{\prime }}.
\end{equation*}

\begin{claim}
\label{pivandfat}Let $n=1$ and suppose that $\sigma $ is doubling. First
assume that $1<p<\infty $. Then for the special measure $s$ in \eqref{w},
inequality \eqref{upper} follows from the dual pivotal condition 
\eqref{dual
pivotal condition}, the Poisson condition \eqref{Poisson condition}, and the
half-strengthened $A_{p}$ condition \eqref{Ap half skirt}. Now assume that $%
1<p\leq 2$. Then for the special measure $s$ in \eqref{w}, (\ref{upper})
follows from \eqref{dual pivotal condition} and \eqref{Ap half skirt}
without \eqref{Poisson condition}.
\end{claim}

With Claim \ref{pivandfat} proved, the discussion above yields the following
result.

\begin{theorem}
\label{poisson sufficiency}Let $n=1$ and suppose that $\sigma $ is doubling.
First assume that $1<p<\infty $. Then the dual Poisson inequality (\ref%
{poissonweightedineq'})\ holds uniformly in $Q_{0}$ and $\left\{
Q_{r}\right\} _{r=1}^{\infty }$ satisfying $\bigcup_{r=1}^{\infty
}Q_{r}\subset Q_{0}$ provided the half-strengthened $A_{p}$ condition (\ref%
{Ap half skirt}), the dual pivotal condition \eqref{dual pivotal condition},
and the Poisson condition \eqref{Poisson condition} all hold. Now assume
that $1<p\leq 2$. Then \eqref{poissonweightedineq'}\ holds uniformly in $%
Q_{0}$ and $\left\{ Q_{r}\right\} _{r=1}^{\infty }$ satisfying $%
\bigcup_{r=1}^{\infty }Q_{r}\subset Q_{0}$ provided \eqref{Ap half skirt}
and \eqref{dual pivotal condition} both hold.
\end{theorem}

\begin{remark}
We do not know if Claim \ref{pivandfat} and Theorem \ref{poisson sufficiency}
hold without the assumption that $\sigma $ is doubling, nor do we know if
the Poisson condition \eqref{Poisson condition} is implied by 
\eqref{Ap half
skirt} and \eqref{dual pivotal condition} when $p>2$.
\end{remark}

We work exclusively in dimension $n=1$ from now on.

\subsubsection{Proof of Claim \protect\ref{pivandfat}}

Instead of applying Theorem \ref{testing Poisson} directly, we first reduce
matters to proving that certain $\mathcal{D}^{\alpha }$-dyadic analogues
hold of the two conditions in Theorem \ref{testing Poisson}. For $\alpha \in
\left\{ 0,\tfrac{1}{3},\tfrac{2}{3}\right\} $ we use the following atomic
measures $ds_{\alpha }$ on $\mathbb{R}_{+}^{2}$, along with the following $%
\mathcal{D}^{\alpha }$-dyadic analogues of the Poisson operators $\mathbb{P}$
and $\mathbb{P}_{+}$ (with $\delta \left( s\right) =s$),%
\begin{eqnarray}
\mathbb{P}_{\alpha }^{dy}\nu \left( x\right) &=&\sum_{r=1}^{\infty }\mathsf{P%
}_{\alpha }^{dy}\left( I_{r}^{\alpha },\nu \right) \chi _{I_{r}^{\alpha
}}\left( x\right) ,  \label{dyadic analogues} \\
\mathbb{P}_{+,\alpha }^{dy}\nu \left( x,t\right) &=&\sum_{Q\in \mathcal{D}%
^{\alpha }:x\in Q\text{ and }\ell \left( Q\right) \geq t}\frac{t}{\ell
\left( Q\right) }\frac{1}{\left\vert Q\right\vert }\int_{Q}d\nu ,  \notag \\
ds_{\alpha }\left( x,t\right) &=&\sum_{r=1}^{\infty }\left\vert
I_{r}^{\alpha }\right\vert _{\sigma }\delta _{Z_{r}^{\alpha }}\left(
x,t\right) ,  \notag
\end{eqnarray}%
where

\begin{enumerate}
\item the interval $I_{r}^{\alpha }$ is chosen to be a \emph{maximal} $%
\mathcal{D}^{\alpha }$-interval contained in $Q_{r}$ with \emph{maximum}
length (there can be at most two such intervals, in which case we choose the
leftmost one),

\item the $\mathcal{D}^{\alpha }$-Poisson integral $\mathsf{P}_{\alpha
}^{dy}\left( Q,\nu \right) $ is given by%
\begin{equation*}
\mathsf{P}_{\alpha }^{dy}\left( Q,\nu \right) \equiv \sum_{\ell =0}^{\infty }%
\frac{2^{-\ell }}{\left\vert Q^{\left( \ell \right) }\right\vert }%
\int_{Q^{\left( \ell \right) }}d\nu ,\ \ \ \ \ Q\in \mathcal{D}^{\alpha },
\end{equation*}%
where $Q^{\left( \ell \right) }$ denotes the $\ell ^{th}$ dyadic parent of $%
Q $ in $\mathcal{D}^{\alpha }$,

\item the point $Z_{r}^{\alpha }=\left( x_{I_{r}^{\alpha }},\ell \left(
I_{r}^{\alpha }\right) \right) $ in $\mathbb{R}_{+}^{2}$ lies above the
center $x_{I_{r}^{\alpha }}$ of $I_{r}^{\alpha }$ at a height equal to the
side length $\ell \left( I_{r}^{\alpha }\right) $ of $I_{r}^{\alpha }$.
\end{enumerate}

We will use the following dyadic analogue of Theorem \ref{testing Poisson},
whose proof is the obvious dyadic analogue of the proof of Theorem \ref%
{testing Poisson}\ as given in \cite{Saw2}.

\begin{theorem}
\label{testing Poisson dyadic}The $\mathcal{D}^{\alpha }$-Poisson inequality%
\begin{equation*}
\int_{\mathbb{R}_{+}^{2}}\mathbb{P}_{+,\alpha }^{dy}\left( f\omega \right)
^{p^{\prime }}ds_{\alpha }\leq C\int_{Q}f^{p^{\prime }}d\omega ,\ \ \ \ \
f\geq 0,
\end{equation*}%
holds \emph{if and only if}%
\begin{eqnarray}
\int_{\mathbb{R}_{+}^{2}}\mathbb{P}_{+,\alpha }^{dy}\left( \chi _{Q}\omega
\right) ^{p^{\prime }}ds_{\alpha } &\leq &C\int_{Q}d\omega ,\ \ \ \ \ \text{%
for all intervals }Q\in \mathcal{D}^{\alpha },  \label{alpha conditions} \\
\int_{\mathbb{R}}\left( \mathbb{P}_{+,\alpha }^{dy}\right) ^{\ast }\left(
t^{p^{\prime }-1}\chi _{\widehat{Q}}ds_{\alpha }\right) ^{p}d\omega &\leq
&C\int_{\widehat{Q}}t^{p^{\prime }}ds_{\alpha }\ \ \ \ \ \text{for all
intervals }Q\in \mathcal{D}^{\alpha }.  \notag
\end{eqnarray}
\end{theorem}

We claim that for any positive measure $\nu $, the collection of shifted
dyadic grids $\left\{ \mathcal{D}^{\alpha }\right\} _{\alpha \in \left\{ 0,%
\tfrac{1}{3},\tfrac{2}{3}\right\} }$ satisfies%
\begin{equation*}
\mathsf{P}\left( Q_{r},\nu \right) =\sum_{\ell =0}^{\infty }\frac{2^{-\ell }%
}{\left\vert 2^{\ell }Q_{r}\right\vert }\int_{2^{\ell }Q_{r}}d\nu \approx
\sum_{\alpha \in \left\{ 0,\tfrac{1}{3},\tfrac{2}{3}\right\} }\sum_{\ell
=0}^{\infty }\frac{2^{-\ell }}{\left\vert \left( I_{r}^{\alpha }\right)
^{\left( \ell \right) }\right\vert }\int_{\left( I_{r}^{\alpha }\right)
^{\left( \ell \right) }}d\nu =\sum_{\alpha \in \left\{ 0,\tfrac{1}{3},\tfrac{%
2}{3}\right\} }\mathsf{P}_{\alpha }^{dy}\left( I_{r}^{\alpha },\nu \right) ,
\end{equation*}%
for all $r$. Indeed, for each interval $2^{\ell }Q_{r}$, there is $\alpha
\in \left\{ 0,\tfrac{1}{3},\tfrac{2}{3}\right\} $ and an interval $Q\in 
\mathcal{D}^{\alpha }$ containing $2^{\ell }Q_{r}$ whose length is
comparable to that of $2^{\ell }Q_{r}$. Thus $Q=\left( I_{r}^{\alpha
}\right) ^{\left( \ell +c\right) }$ for some universal positive integer $c$.
Now%
\begin{eqnarray*}
\mathbb{P}_{+}\left( \nu \right) \left( x_{Q_{r}},\ell \left( Q_{r}\right)
\right) &=&\int_{\mathbb{R}}P_{\ell \left( Q_{r}\right) }\left(
x_{Q_{r}}-y\right) d\nu \left( y\right) \\
&\approx &\sum_{\ell =0}^{\infty }2^{-\ell }\frac{1}{\left\vert 2^{\ell
}Q_{r}\right\vert }\int_{2^{\ell }Q_{r}}d\nu =\mathsf{P}\left( Q_{r},\nu
\right) .
\end{eqnarray*}%
Since $\sigma $ is \emph{doubling} and $I_{r}^{\alpha }$ is a maximal $%
\mathcal{D}^{\alpha }$-interval in $Q_{r}$ with maximum length, we thus have 
$\left\vert Q_{r}\right\vert _{\sigma }\lesssim \left\vert I_{r}^{\alpha
}\right\vert _{\sigma }$ and%
\begin{eqnarray*}
\int_{\mathbb{R}_{+}^{n+1}}\mathbb{P}_{+}\nu \left( x,t\right) ^{p}ds
&=&\sum_{r=1}^{\infty }\left\vert Q_{r}\right\vert _{\sigma }\mathbb{P}%
_{+}\nu \left( x_{Q_{r}},\ell \left( Q_{r}\right) \right) ^{p} \\
&\approx &\sum_{r=1}^{\infty }\left\vert Q_{r}\right\vert _{\sigma }\mathsf{P%
}\left( Q_{r},\nu \right) ^{p}\approx \sum_{\alpha \in \left\{ 0,\tfrac{1}{3}%
,\tfrac{2}{3}\right\} }\sum_{r=1}^{\infty }\left\vert I_{r}^{\alpha
}\right\vert _{\sigma }\mathsf{P}_{\alpha }^{dy}\left( I_{r}^{\alpha },\nu
\right) ^{p} \\
&=&\sum_{\alpha \in \left\{ 0,\tfrac{1}{3},\tfrac{2}{3}\right\} }\int_{%
\mathbb{R}_{+}^{2}}\mathbb{P}_{+,\alpha }^{dy}\nu \left( x,t\right)
^{p}ds_{\alpha }.
\end{eqnarray*}%
This together with Theorem \ref{testing Poisson dyadic}\ reduces the proof
of Claim \ref{pivandfat} to showing that \eqref{alpha conditions} holds for
all $\alpha \in \left\{ 0,\frac{1}{3},\frac{2}{3}\right\} $.

\bigskip

Now the definition of $s_{\alpha }$ in \eqref{dyadic analogues} shows that
the left side of the first line in \eqref{alpha conditions} is%
\begin{equation*}
\int_{\mathbb{R}_{+}^{2}}\mathbb{P}_{+,\alpha }^{dy}\left( \chi _{Q}\omega
\right) ^{p^{\prime }}ds_{\alpha }=\sum_{r=1}^{\infty }\left\vert
I_{r}^{\alpha }\right\vert _{\sigma }\mathsf{P}_{\alpha }^{dy}\left(
I_{r}^{\alpha },\chi _{Q}\omega \right) ^{p^{\prime }}.
\end{equation*}%
Recall that $I_{r}^{\alpha },Q\in \mathcal{D}^{\alpha }$. Now if $Q\subset
I_{r}^{\alpha }$ for some $r$, then the above sum consists of just one term
that satisfies%
\begin{equation*}
\left\vert I_{r}^{\alpha }\right\vert _{\sigma }\mathsf{P}_{\alpha
}^{dy}\left( I_{r}^{\alpha },\chi _{Q}\omega \right) ^{p^{\prime }}\leq C%
\frac{\left\vert I_{r}^{\alpha }\right\vert _{\sigma }\left\vert
Q\right\vert _{\omega }^{p^{\prime }-1}}{\left\vert I_{r}^{\alpha
}\right\vert ^{p^{\prime }}}\left\vert Q\right\vert _{\omega }\leq C\mathcal{%
A}_{p}\left( \omega ,\sigma \right) ^{p^{\prime }}\left\vert Q\right\vert
_{\omega }.
\end{equation*}%
Otherwise we have%
\begin{eqnarray*}
\int_{\mathbb{R}_{+}^{2}}\mathbb{P}_{+,\alpha }^{dy}\left( \chi _{Q}\omega
\right) ^{p^{\prime }}ds_{\alpha } &\lesssim &\sum_{I_{r}^{\alpha }\subset
Q}\left\vert I_{r}^{\alpha }\right\vert _{\sigma }\mathsf{P}_{\alpha
}^{dy}\left( I_{r}^{\alpha },\chi _{Q}\omega \right) ^{p^{\prime
}}+\sum_{I_{r}^{\alpha }\cap Q=\emptyset }\left\vert I_{r}^{\alpha
}\right\vert _{\sigma }\mathsf{P}_{\alpha }^{dy}\left( I_{r}^{\alpha },\chi
_{Q}\omega \right) ^{p^{\prime }} \\
&\leq &\mathfrak{C}_{\ast }^{p^{\prime }}\int_{Q}d\omega
+\sum_{I_{r}^{\alpha }\cap Q=\emptyset }\left\vert I_{r}^{\alpha
}\right\vert _{\sigma }\left\{ \sum_{\ell =0}^{\infty }\frac{2^{-\ell }}{%
\left\vert \left( I_{r}^{\alpha }\right) ^{\left( \ell \right) }\right\vert }%
\int_{Q\cap \left( I_{r}^{\alpha }\right) ^{\left( \ell \right) }}d\omega
\right\} ^{p^{\prime }},
\end{eqnarray*}%
where the local term has been estimated by the dual pivotal condition (\ref%
{dual pivotal condition}) applied to $Q$.

Now if $I_{r}^{\alpha }\subset Q^{\left( m\right) }\setminus Q^{\left(
m-1\right) }$, then $Q\cap Q_{r}^{\left( \ell \right) }\neq \emptyset $ only
if $Q^{\left( m\right) }\subset \left( I_{r}^{\alpha }\right) ^{\left( \ell
\right) }$. Thus the second term on the right can be estimated by%
\begin{eqnarray*}
&&\sum_{m=1}^{\infty }\sum_{I_{r}^{\alpha }\subset Q^{\left( m\right)
}\setminus Q^{\left( m-1\right) }}\left\vert I_{r}^{\alpha }\right\vert
_{\sigma }\left\{ \sum_{\ell =0}^{\infty }\frac{2^{-\ell }}{\left\vert
\left( I_{r}^{\alpha }\right) ^{\left( \ell \right) }\right\vert }%
\int_{Q\cap \left( I_{r}^{\alpha }\right) ^{\left( \ell \right) }}d\omega
\right\} ^{p^{\prime }} \\
&\leq &\sum_{m=1}^{\infty }\sum_{I_{r}^{\alpha }\subset Q^{\left( m\right)
}\setminus Q^{\left( m-1\right) }}\left\vert I_{r}^{\alpha }\right\vert
_{\sigma }\sum_{\ell =0}^{\infty }2^{-\ell }\left( \frac{\int_{Q\cap \left(
I_{r}^{\alpha }\right) ^{\left( \ell \right) }}d\omega }{\left\vert \left(
I_{r}^{\alpha }\right) ^{\left( \ell \right) }\right\vert }\right)
^{p^{\prime }} \\
&\leq &C\sum_{m=1}^{\infty }\sum_{I_{r}^{\alpha }\subset Q^{\left( m\right)
}\setminus Q^{\left( m-1\right) }}\left\vert I_{r}^{\alpha }\right\vert
_{\sigma }\sum_{\ell =0}^{\infty }2^{-\ell }\left( \frac{\int_{Q}d\omega }{%
\left\vert Q^{\left( m\right) }\right\vert }\right) ^{p^{\prime }} \\
&\leq &\left( \sum_{m=1}^{\infty }\frac{\left\vert Q^{\left( m\right)
}\right\vert _{\sigma }}{\left\vert Q^{\left( m\right) }\right\vert
^{p^{\prime }}}\right) \left\vert Q\right\vert _{\omega }^{p^{\prime
}-1}\int_{Q}d\omega \\
&=&\left\{ \frac{1}{\left\vert Q\right\vert ^{p^{\prime }}}\left( \int
s_{Q,\alpha }^{dy}\left( x\right) ^{p^{\prime }}d\sigma \left( x\right)
\right) \left\vert Q\right\vert _{\omega }^{p^{\prime }-1}\right\}
\int_{Q}d\omega \leq C\mathcal{A}_{p}\left( \omega ,\sigma \right)
^{p^{\prime }}\int_{Q}d\omega ,
\end{eqnarray*}%
where we have used%
\begin{equation*}
s_{Q,\alpha }^{dy}\left( x\right) \equiv \sum_{m=0}^{\infty }\frac{%
\left\vert Q\right\vert }{\left\vert Q^{\left( m\right) }\right\vert }\chi
_{Q^{\left( m\right) }}\left( x\right) \lesssim s_{Q}\left( x\right) ,
\end{equation*}%
and the half-strengthened $A_{p}$ condition \eqref{Ap half skirt} in the
final inequality.

\bigskip

Now we turn to showing that the second line in \eqref{alpha conditions}
holds using only the $A_{p}$ condition \eqref{Ap}. First we compute the dual
operator $\left( \mathbb{P}_{+,\alpha }^{dy}\right) ^{\ast }$. Since the
kernel of $\mathbb{P}_{+,\alpha }^{dy}$ is%
\begin{equation*}
\mathbb{P}_{+,\alpha }^{dy}\left[ \left( x,t\right) ,y\right] \equiv
\sum_{I\in \mathcal{D}^{\alpha }:\ell \left( I\right) \geq t}\chi _{I}\left(
x\right) \frac{t}{\ell \left( I\right) }\frac{1}{\left\vert I\right\vert }%
\chi _{I}\left( y\right) ,
\end{equation*}%
we have for any positive measure $\mu \left( x,t\right) $ on the upper half
space $\mathbb{R}_{+}^{2}$,%
\begin{eqnarray*}
\left( \mathbb{P}_{+,\alpha }^{dy}\right) ^{\ast }\mu \left( y\right)
&=&\int_{\mathbb{R}_{+}^{2}}\left\{ \sum_{I\in \mathcal{D}^{\alpha }:\ell
\left( I\right) \geq t}\chi _{I}\left( x\right) \frac{t}{\ell \left(
I\right) }\frac{1}{\left\vert I\right\vert }\chi _{I}\left( y\right)
\right\} d\mu \left( x,t\right) \\
&=&\sum_{I\in \mathcal{D}^{\alpha }:y\in I}\frac{1}{\left\vert I\right\vert }%
\int_{\widehat{I}}\frac{t}{\ell \left( I\right) }d\mu \left( x,t\right) .
\end{eqnarray*}%
Using the third line in \eqref{dyadic analogues} we compute that%
\begin{equation*}
\int_{\widehat{Q}}t^{p^{\prime }}ds_{\alpha }=\sum_{I_{r}^{\alpha }\subset
Q}\left\vert I_{r}^{\alpha }\right\vert _{\sigma }\left\vert I_{r}^{\alpha
}\right\vert ^{p^{\prime }},
\end{equation*}%
and%
\begin{eqnarray*}
\left( \mathbb{P}_{+,\alpha }^{dy}\right) ^{\ast }\left( t^{p^{\prime
}-1}\chi _{\widehat{Q}}ds_{\alpha }\right) \left( y\right) &=&\sum_{I\in 
\mathcal{D}^{\alpha }:y\in I}\frac{1}{\left\vert I\right\vert }\int_{%
\widehat{I}\cap \widehat{Q}}\frac{t}{\ell \left( I\right) }t^{p^{\prime
}-1}ds_{\alpha }\left( x,t\right) \\
&=&\sum_{I_{r}^{\alpha }\subset Q}\left\vert I_{r}^{\alpha }\right\vert
_{\sigma }\left\vert I_{r}^{\alpha }\right\vert ^{p^{\prime }-1}\sum_{\ell
=0}^{\infty }\frac{2^{-\ell }}{\left\vert \left( I_{r}^{\alpha }\right)
^{\left( \ell \right) }\right\vert }\chi _{\left( I_{r}^{\alpha }\right)
^{\left( \ell \right) }}\left( y\right) .
\end{eqnarray*}

Thus we must prove%
\begin{equation}
\int_{\mathbb{R}}\left( \sum_{I_{r}^{\alpha }\subset Q}\left\vert
I_{r}^{\alpha }\right\vert _{\sigma }\left\vert I_{r}^{\alpha }\right\vert
^{p^{\prime }-1}\sum_{\ell =0}^{\infty }\frac{2^{-\ell }}{\left\vert \left(
I_{r}^{\alpha }\right) ^{\left( \ell \right) }\right\vert }\chi _{\left(
I_{r}^{\alpha }\right) ^{\left( \ell \right) }}\left( y\right) \right)
^{p}d\omega \left( y\right) \leq C\mathcal{A}_{p}\left( \omega ,\sigma
\right) ^{p}\sum_{I_{r}^{\alpha }\subset Q}\left\vert I_{r}^{\alpha
}\right\vert _{\sigma }\left\vert I_{r}^{\alpha }\right\vert ^{p^{\prime }};
\label{Poisson condition'}
\end{equation}%
but this is the Poisson condition \eqref{Poisson condition} in Theorem \ref%
{improved} for the shifted dyadic grid $\mathcal{D}^{\alpha }$. This
completes the proof of the first assertion in Claim \ref{pivandfat}
regarding the case $1<p<\infty $. We now assume that $1<p\leq 2$ for the
remainder of the proof.

To obtain \eqref{Poisson condition'} it suffices to show that for each $\ell
\geq 0$:%
\begin{equation}
\int_{\mathbb{R}}\left( \sum_{I_{r}^{\alpha }\subset Q}\left\vert
I_{r}^{\alpha }\right\vert _{\sigma }\left\vert I_{r}^{\alpha }\right\vert
^{p^{\prime }-2}2^{-2\ell }\chi _{\left( I_{r}^{\alpha }\right) ^{\left(
\ell \right) }}\left( y\right) \right) ^{p}d\omega \left( y\right) \leq
C2^{-p\ell }\mathcal{A}_{p}\left( \omega ,\sigma \right)
^{p}\sum_{I_{r}^{\alpha }\subset Q}\left\vert I_{r}^{\alpha }\right\vert
_{\sigma }\left\vert I_{r}^{\alpha }\right\vert ^{p^{\prime }}.
\label{ell p estimate}
\end{equation}%
Indeed, with this in hand, Minkowski's inequality yields%
\begin{eqnarray}
\left\Vert \left( \mathbb{P}_{+,\alpha }^{dy}\right) ^{\ast }\left( t\chi _{%
\widehat{Q}}ds_{\alpha }\right) \right\Vert _{L^{p}\left( \omega \right) }
&=&\left\Vert \sum_{\ell =0}^{\infty }\sum_{I_{r}^{\alpha }\subset
Q}\left\vert I_{r}^{\alpha }\right\vert _{\sigma }\left\vert I_{r}^{\alpha
}\right\vert ^{p^{\prime }-2}2^{-2\ell }\chi _{\left( I_{r}^{\alpha }\right)
^{\left( \ell \right) }}\right\Vert _{L^{p}\left( \omega \right) }
\label{Mink} \\
&\leq &\sum_{\ell =0}^{\infty }\left\Vert \sum_{I_{r}^{\alpha }\subset
Q}\left\vert I_{r}^{\alpha }\right\vert _{\sigma }\left\vert I_{r}^{\alpha
}\right\vert ^{p^{\prime }-2}2^{-2\ell }\chi _{\left( I_{r}^{\alpha }\right)
^{\left( \ell \right) }}\right\Vert _{L^{p}\left( \omega \right) }  \notag \\
&\leq &C\sum_{\ell =0}^{\infty }2^{-\ell }\mathcal{A}_{p}\left( \omega
,\sigma \right) \left( \sum_{I_{r}^{\alpha }\subset Q}\left\vert
I_{r}^{\alpha }\right\vert _{\sigma }\left\vert I_{r}^{\alpha }\right\vert
^{p^{\prime }}\right) ^{\frac{1}{p}},  \notag
\end{eqnarray}%
as required.

Note that for $a>0$ and $p>1$, 
\begin{equation*}
h\left( x\right) \equiv \left( a+x\right) ^{p}-a^{p}-p\left( a+x\right)
^{p-1}x,
\end{equation*}%
is decreasing on $\left[ 0,\infty \right) $ since%
\begin{equation*}
h^{\prime }\left( x\right) =-p\left( p-1\right) \left( a+x\right)
^{p-2}x<0,\ \ \ \ \ x>0.
\end{equation*}%
Since $h(0) =0$ we have $h\left( x\right) \leq 0$ for $x\geq 0$, i.e.%
\begin{equation}
\left( a+x\right) ^{p}-a^{p}\leq p\left( a+x\right) ^{p-1}x,\ \ \ \ \ \text{%
for }a,x>0\text{ and }p>1.  \label{ineq}
\end{equation}%
Now fix an interval $Q$ in \eqref{ell p estimate} and arrange the intervals $%
I_{r}^{\alpha }$ that are contained in $Q$ into a sequence $\left\{
I_{r}^{\alpha }\right\} _{r=1}^{N}$ in which the lengths $\left\vert
I_{r}^{\alpha }\right\vert $ are increasing (we may suppose without loss of
generality that $N$ is finite). Recall we are now assuming $1<p\leq 2$.
Integrate by parts and apply (\ref{ineq}) to estimate the left side of %
\eqref{ell p estimate} by%
\begin{eqnarray*}
&&2^{-2p\ell }\int_{\mathbb{R}}\left( \sum_{r=1}^{N}\left\vert I_{r}^{\alpha
}\right\vert _{\sigma }\left\vert I_{r}^{\alpha }\right\vert ^{p^{\prime
}-2}\chi _{\left( I_{r}^{\alpha }\right) ^{\left( \ell \right) }}\left(
y\right) \right) ^{p}d\omega \left( y\right) \\
&=&2^{-2p\ell }\int_{\mathbb{R}}\sum_{n=1}^{N}\left\{ \left(
\sum_{r=1}^{n}\left\vert I_{r}^{\alpha }\right\vert _{\sigma }\left\vert
I_{r}^{\alpha }\right\vert ^{p^{\prime }-2}\chi _{\left( I_{r}^{\alpha
}\right) ^{\left( \ell \right) }}\left( y\right) \right) ^{p}-\left(
\sum_{r=1}^{n-1}\left\vert I_{r}^{\alpha }\right\vert _{\sigma }\left\vert
I_{r}^{\alpha }\right\vert ^{p^{\prime }-2}\chi _{\left( I_{r}^{\alpha
}\right) ^{\left( \ell \right) }}\left( y\right) \right) ^{p}\right\}
d\omega \left( y\right) \\
&\leq &2^{-2p\ell }\int_{\mathbb{R}}\sum_{n=1}^{N}\left\{ p\left(
\sum_{r=1}^{n}\left\vert I_{r}^{\alpha }\right\vert _{\sigma }\left\vert
I_{r}^{\alpha }\right\vert ^{p^{\prime }-2}\chi _{\left( I_{r}^{\alpha
}\right) ^{\left( \ell \right) }}\left( y\right) \right) ^{p-1}\left\vert
I_{n}^{\alpha }\right\vert _{\sigma }\left\vert I_{n}^{\alpha }\right\vert
^{p^{\prime }-2}\chi _{\left( I_{n}^{\alpha }\right) ^{\left( \ell \right)
}}\left( y\right) \right\} d\omega \left( y\right) \\
&\leq &2^{-2p\ell }p\sum_{n=1}^{N}\int_{\mathbb{R}}\left\{ \left(
\sum_{r=1}^{n}\left\vert I_{r}^{\alpha }\right\vert _{\sigma }\chi _{\left(
I_{r}^{\alpha }\right) ^{\left( \ell \right) }}\left( y\right) \right)
^{p-1}\left\vert I_{n}^{\alpha }\right\vert _{\sigma }\left\vert
I_{n}^{\alpha }\right\vert ^{p^{\prime }-2}\left\vert I_{n}^{\alpha
}\right\vert ^{\left( p^{\prime }-2\right) \left( p-1\right) }\chi _{\left(
I_{n}^{\alpha }\right) ^{\left( \ell \right) }}\left( y\right) \right\}
d\omega \left( y\right) ,
\end{eqnarray*}%
where we have used \eqref{ineq} with $a=\sum_{r=1}^{n-1}\left\vert
I_{r}^{\alpha }\right\vert _{\sigma }\left\vert I_{r}^{\alpha }\right\vert
^{p^{\prime }-2}\chi _{\left( I_{r}^{\alpha }\right) ^{\left( \ell \right)
}}\left( y\right) $ and $x=\left\vert I_{n}^{\alpha }\right\vert _{\sigma
}\left\vert I_{n}^{\alpha }\right\vert ^{p^{\prime }-2}\chi _{\left(
I_{n}^{\alpha }\right) ^{\left( \ell \right) }}\left( y\right) $, and then
used $\left\vert I_{r}^{\alpha }\right\vert ^{p^{\prime }-2}\leq \left\vert
I_{n}^{\alpha }\right\vert ^{p^{\prime }-2}$ for $1\leq r\leq n$, which
follows from $\left\vert I_{r}^{\alpha }\right\vert \leq \left\vert
I_{n}^{\alpha }\right\vert $ and $p^{\prime }\geq 2$. If $\left(
I_{r}^{\alpha }\right) ^{\left( \ell \right) }\cap \left( I_{n}^{\alpha
}\right) ^{\left( \ell \right) }\neq \emptyset $ and $1\leq r\leq n$, then $%
I_{r}^{\alpha }\subset \left( I_{n}^{\alpha }\right) ^{\left( \ell \right) }$
and so 
\begin{eqnarray*}
&&\int_{\mathbb{R}}\left( \sum_{I_{r}^{\alpha }\subset Q}\left\vert
I_{r}^{\alpha }\right\vert _{\sigma }\left\vert I_{r}^{\alpha }\right\vert
^{p^{\prime }-2}2^{-2\ell }\chi _{\left( I_{r}^{\alpha }\right) ^{\left(
\ell \right) }}\left( y\right) \right) ^{p}d\omega \left( y\right) \\
&\leq &2^{-2p\ell }p\sum_{n=1}^{N}\left\vert I_{n}^{\alpha }\right\vert
_{\sigma }\left\vert I_{n}^{\alpha }\right\vert ^{p^{\prime }p-2p}\int_{%
\mathbb{R}}\left( \sum_{1\leq r\leq n:I_{r}^{\alpha }\subset \left(
I_{n}^{\alpha }\right) ^{\left( \ell \right) }}^{n}\left\vert I_{r}^{\alpha
}\right\vert _{\sigma }\right) ^{p-1}\chi _{\left( I_{n}^{\alpha }\right)
^{\left( \ell \right) }}\left( y\right) d\omega \left( y\right) \\
&\leq &2^{-2p\ell }p\sum_{n=1}^{N}\left\vert I_{n}^{\alpha }\right\vert
_{\sigma }\left\vert I_{n}^{\alpha }\right\vert ^{p^{\prime }p-2p}\left\vert
\left( I_{n}^{\alpha }\right) ^{\left( \ell \right) }\right\vert _{\sigma
}^{p-1}\left\vert \left( I_{n}^{\alpha }\right) ^{\left( \ell \right)
}\right\vert _{\omega } \\
&\leq &2^{-2p\ell }p\mathcal{A}_{p}\left( \omega ,\sigma \right)
^{p}\sum_{n=1}^{N}\left\vert I_{n}^{\alpha }\right\vert _{\sigma }\left\vert
I_{n}^{\alpha }\right\vert ^{p^{\prime }p-2p}\left\vert \left( I_{n}^{\alpha
}\right) ^{\left( \ell \right) }\right\vert ^{p} \\
&=&2^{-p\ell }p\mathcal{A}_{p}\left( \omega ,\sigma \right)
^{p}\sum_{n=1}^{N}\left\vert I_{n}^{\alpha }\right\vert _{\sigma }\left\vert
I_{n}^{\alpha }\right\vert ^{p^{\prime }}=2^{-p\ell }p\mathcal{A}_{p}\left(
\omega ,\sigma \right) ^{p}\sum_{I_{r}^{\alpha }\subset Q}\left\vert
I_{n}^{\alpha }\right\vert _{\sigma }\left\vert I_{n}^{\alpha }\right\vert
^{p^{\prime }}.
\end{eqnarray*}

Thus we have proved \eqref{ell p estimate} for $p\in \left( 1,2\right] $,
which completes the proof of \eqref{alpha conditions}. This finishes the
proof of Claim \ref{pivandfat}, and hence also that of Theorem \ref{poisson
sufficiency}.

\subsection{Necessity of the conditions}

Here we consider the two weight Hilbert transform inequality for $1<p<\infty 
$. We show the necessity of the strengthened $A_{p}$ condition for general
weights, as well as the necessity of the dual pivotal condition for the dual
testing condition, and the dual Poisson inequality for the dual Hilbert
transform inequality, when $\sigma $ is doubling.

\subsubsection{The strengthened $A_{p}$ condition}

Here we derive a necessary condition for the weighted inequality (\ref%
{2weight}) but with the Hilbert transform $T$ in place of $T_{\natural }$,%
\begin{equation}
\int_{\mathbb{R}\setminus supp\ f}T(f\sigma )\left( x\right) ^{p}d\omega
\left( x\right) \leq C\int_{\mathbb{R}^{n}}\left\vert f\left( x\right)
\right\vert ^{p}d\sigma \left( x\right) ,  \label{2weightflat}
\end{equation}%
that is stronger than the two weight $A_{p}$ condition \eqref{Ap}, namely
the \emph{strengthened} $A_{p}$ condition%
\begin{equation}
\left( \int_{\mathbb{R}}\left( \frac{\left\vert Q\right\vert }{\left\vert
Q\right\vert +\left\vert x-x_{Q}\right\vert }\right) ^{p}d\omega \left(
x\right) \right) ^{\frac{1}{p}}\left( \int_{\mathbb{R}}\left( \frac{%
\left\vert Q\right\vert }{\left\vert Q\right\vert +\left\vert
x-x_{Q}\right\vert }\right) ^{p^{\prime }}d\sigma \left( x\right) \right) ^{%
\frac{1}{p^{\prime }}}\leq C\left\vert Q\right\vert ,  \label{skirt Ap}
\end{equation}%
for all intervals $Q$.

Preliminary results in this direction were obtained by Muckenhoupt and
Wheeden, and in the setting of fractional integrals by Gabidzashvili and
Kokilashvili, and here we follow the argument proving (1.9) in Sawyer and
Wheeden \cite{SaWh}, where `two-tailed' inequalities of the type 
\eqref{skirt
Ap} originated in the fractional integral setting. A somewhat different
approach to this for the conjugate operator in the disk when $p=2$ uses
conformal invariance and appears in \cite{NTV3}, and provides the first
instance of a strengthened $A_{2}$ condition being proved necessary for a
two weight inequality for a singular integral.

Fix an interval $Q$ and for $a\in \mathbb{R}$ and $r>0$ let 
\begin{eqnarray*}
s_{Q}\left( x\right) &=&\frac{\left\vert Q\right\vert }{\left\vert
Q\right\vert +\left\vert x-x_{Q}\right\vert }, \\
f_{a,r}\left( y\right) &=&\chi _{\left( a-r,a\right) }\left( y\right)
s_{Q}\left( y\right) ^{p^{\prime }-1},
\end{eqnarray*}%
where $x_{Q}$ is the center of the interval $Q$. For convenience we assume
that neither $\omega $ nor $\sigma $ have any point masses - see \cite%
{LaSaUr} for the modifications necessary when point masses are present. For $%
y<x$ we have%
\begin{eqnarray*}
\left\vert Q\right\vert \left( x-y\right) &=&\left\vert Q\right\vert \left(
x-x_{Q}\right) +\left\vert Q\right\vert \left( x_{Q}-y\right) \\
&\leq &\left( \left\vert Q\right\vert +\left\vert x-x_{Q}\right\vert \right)
\left( \left\vert Q\right\vert +\left\vert x_{Q}-y\right\vert \right) ,
\end{eqnarray*}%
and so%
\begin{equation*}
\frac{1}{x-y}\geq \left\vert Q\right\vert ^{-1}s_{Q}\left( x\right)
s_{Q}\left( y\right) ,\ \ \ \ \ y<x.
\end{equation*}%
Thus for $x>a$ we obtain that%
\begin{eqnarray*}
H\left( f_{a,r}\sigma \right) \left( x\right) &=&\int_{a-r}^{a}\frac{1}{x-y}%
s_{Q}\left( y\right) ^{p^{\prime }-1}d\sigma \left( y\right) \\
&\geq &\left\vert Q\right\vert ^{-1}s_{Q}\left( x\right)
\int_{a-r}^{a}s_{Q}\left( y\right) ^{p^{\prime }}d\sigma \left( y\right) ,
\end{eqnarray*}%
and hence by \eqref{2weightflat} for the Hilbert transform $H$,%
\begin{eqnarray*}
&&\left\vert Q\right\vert ^{-p}\int_{a}^{\infty }s_{Q}\left( x\right)
^{p}\left( \int_{a-r}^{a}s_{Q}\left( y\right) ^{p^{\prime }}d\sigma \left(
y\right) \right) ^{p}d\omega \left( x\right) \\
&\leq &\int \left\vert H\left( f_{a,r}\sigma \right) \left( x\right)
\right\vert ^{p}d\omega \left( x\right) \leq C\int \left\vert f_{a,r}\left(
y\right) \right\vert ^{p}d\sigma \left( y\right) =C\int_{a-r}^{a}s_{Q}\left(
y\right) ^{p^{\prime }}d\sigma \left( y\right) .
\end{eqnarray*}

From this we obtain%
\begin{equation*}
\left\vert Q\right\vert ^{-p}\left( \int_{a}^{\infty }s_{Q}\left( x\right)
^{p}d\omega \left( x\right) \right) \left( \int_{a-r}^{a}s_{Q}\left(
y\right) ^{p^{\prime }}d\sigma \left( y\right) \right) ^{p-1}\leq C,
\end{equation*}%
and upon letting $r\rightarrow \infty $ and taking $p^{th}$ roots, we get%
\begin{equation*}
\left( \int_{a}^{\infty }s_{Q}\left( x\right) ^{p}d\omega \left( x\right)
\right) ^{\frac{1}{p}}\left( \int_{-\infty }^{a}s_{Q}\left( y\right)
^{p^{\prime }}d\sigma \left( y\right) \right) ^{\frac{1}{p^{\prime }}}\leq
C\left\vert Q\right\vert .
\end{equation*}%
Similarly we have%
\begin{equation*}
\left( \int_{-\infty }^{a}s_{Q}\left( x\right) ^{p}d\omega \left( x\right)
\right) ^{\frac{1}{p}}\left( \int_{a}^{\infty }s_{Q}\left( y\right)
^{p^{\prime }}d\sigma \left( y\right) \right) ^{\frac{1}{p^{\prime }}}\leq
C\left\vert Q\right\vert .
\end{equation*}%
Now we choose $a$ so that 
\begin{equation*}
\int_{-\infty }^{a}s_{Q}\left( y\right) ^{p^{\prime }}d\sigma \left(
y\right) =\int_{a}^{\infty }s_{Q}\left( y\right) ^{p^{\prime }}d\sigma
\left( y\right) =\frac{1}{2}\int s_{Q}\left( y\right) ^{p^{\prime }}d\sigma
\left( y\right) ,
\end{equation*}%
and conclude that%
\begin{eqnarray*}
&&\left( \int s_{Q}\left( x\right) ^{p}d\omega \left( x\right) \right) ^{%
\frac{1}{p}}\left( \int s_{Q}\left( y\right) ^{p^{\prime }}d\sigma \left(
y\right) \right) ^{\frac{1}{p^{\prime }}} \\
&\leq &\left( \int_{-\infty }^{a}s_{Q}\left( x\right) ^{p}d\omega \left(
x\right) \right) ^{\frac{1}{p}}\left( \int s_{Q}\left( y\right) ^{p^{\prime
}}d\sigma \left( y\right) \right) ^{\frac{1}{p^{\prime }}} \\
&&+\left( \int_{a}^{\infty }s_{Q}\left( x\right) ^{p}d\omega \left( x\right)
\right) ^{\frac{1}{p}}\left( \int s_{Q}\left( y\right) ^{p^{\prime }}d\sigma
\left( y\right) \right) ^{\frac{1}{p^{\prime }}} \\
&\leq &2^{\frac{1}{p^{\prime }}}\left( \int_{-\infty }^{a}s_{Q}\left(
x\right) ^{p}d\omega \left( x\right) \right) ^{\frac{1}{p}}\left(
\int_{a}^{\infty }s_{Q}\left( y\right) ^{p^{\prime }}d\sigma \left( y\right)
\right) ^{\frac{1}{p^{\prime }}} \\
&&+2^{\frac{1}{p^{\prime }}}\left( \int_{a}^{\infty }s_{Q}\left( x\right)
^{p}d\omega \left( x\right) \right) ^{\frac{1}{p}}\left( \int_{-\infty
}^{a}s_{Q}\left( y\right) ^{p^{\prime }}d\sigma \left( y\right) \right) ^{%
\frac{1}{p^{\prime }}} \\
&\leq &2^{1+\frac{1}{p^{\prime }}}C\left\vert Q\right\vert .
\end{eqnarray*}

\subsubsection{Necessity of the dual pivotal condition and the dual Poisson
inequality for a doubling measure}

Here we show first that if $\sigma $ is a \emph{doubling} measure, then the
dual pivotal condition \eqref{dual pivotal condition} with $\delta \left(
s\right) =s$ is implied by the $A_{p}$ condition \eqref{Ap} and the dual
testing condition for the Hilbert transform $H$,%
\begin{equation}
\int_{I}\left\vert H\left( \chi _{I}\omega \right) \left( x\right)
\right\vert ^{p^{\prime }}d\sigma \left( x\right) \leq C_{\omega ,\sigma
,p}\left\vert I\right\vert _{\omega },\ \ \ \ \ \text{for all intervals }I.
\label{Hilbert test}
\end{equation}%
After this we show that the dual Poisson inequality (\ref%
{poissonweightedineq'}) is implied by the $A_{p}$ condition \eqref{Ap} and
the dual Hilbert transform inequality,%
\begin{equation}
\int_{I}\left\vert H\left( \chi _{I}g\omega \right) \left( x\right)
\right\vert ^{p^{\prime }}d\sigma \left( x\right) \leq C_{\omega ,\sigma
,p}\int_{I}g\left( x\right) ^{p^{\prime }}d\omega \left( x\right) ,\ \ \ \ \ 
\text{for all }g\geq 0\text{ and intervals }I.  \label{Hilbert partial}
\end{equation}

\begin{lemma}
\label{doub piv}Suppose that $\sigma $ is doubling and $T=H$ is the Hilbert
transform. Then the dual pivotal condition \eqref{dual pivotal condition} is
implied by the $A_{p}$ condition \eqref{Ap} and the dual testing condition (%
\ref{Hilbert test}).
\end{lemma}

\textbf{Proof}: We begin by proving that for any interval $I$ and any
positive measure $\nu $ supported in $\mathbb{R}\setminus I$, we have%
\begin{equation}
\mathbb{P}\left( I;\nu \right) \leq \frac{1}{\left\vert I\right\vert }%
\int_{I}d\nu +2\left\vert I\right\vert \inf_{x,y\in I}\frac{H\left( \chi
_{I^{c}}\nu \right) \left( x\right) -H\left( \chi _{I^{c}}\nu \right) \left(
y\right) }{x-y},  \label{neccinequ}
\end{equation}%
where we here redefine%
\begin{equation}
\mathbb{P}\left( I;\nu \right) \equiv \frac{1}{\left\vert I\right\vert }%
\int_{I}d\nu +\frac{\left\vert I\right\vert }{2}\int_{\mathbb{R}\setminus I}%
\frac{1}{\left\vert z-z_{I}\right\vert ^{2}}d\nu \left( z\right) ,
\label{redef}
\end{equation}%
with $z_{I}$ the center of $I$. Note that this definition of $\mathbb{P}%
\left( I;\nu \right) $ is comparable to that in \eqref{PQnu} with $\delta
\left( s\right) =s$. Note also that $H\left( \chi _{I^{c}}\nu \right) $ is
defined by \eqref{2weightflat} on $I$, and increasing on $I$ when $\nu $ is
positive, so that the infimum in \eqref{neccinequ} is nonnegative.

To see \eqref{neccinequ}, we suppose without loss of generality that $%
I=\left( -a,a\right) $, and a calculation then shows that for $-a\leq
x<y\leq a$,%
\begin{eqnarray*}
H\left( \chi _{I^{c}}\nu \right) \left( y\right) -H\left( \chi _{I^{c}}\nu
\right) \left( x\right) &=&\int_{\mathbb{R}\setminus I}\left\{ \frac{1}{z-y}-%
\frac{1}{z-x}\right\} d\nu \left( z\right) \\
&=&\left( y-x\right) \int_{\mathbb{R}\setminus I}\frac{1}{\left( z-y\right)
\left( z-x\right) }d\nu \left( z\right) \\
&\geq &\frac{1}{4}\left( y-x\right) \int_{\mathbb{R}\setminus I}\frac{1}{%
z^{2}}d\nu \left( z\right) ,
\end{eqnarray*}%
since $\frac{1}{\left( z-y\right) \left( z-x\right) }$ is positive and
satisfies%
\begin{equation*}
\frac{1}{\left( z-y\right) \left( z-x\right) }\geq \frac{1}{4z^{2}}
\end{equation*}%
on each interval $\left( -\infty ,-a\right) $ and $\left( a,\infty \right) $
in $\mathbb{R}\setminus I$ when $-a\leq x<y\leq a$. Thus we have from (\ref%
{redef}),%
\begin{eqnarray*}
\mathbb{P}\left( I;\nu \right) &=&\frac{1}{\left\vert I\right\vert }%
\int_{I}d\nu +\frac{\left\vert I\right\vert }{2}\int_{\mathbb{R}\setminus I}%
\frac{1}{z^{2}}d\nu \left( z\right) \\
&\leq &\frac{1}{\left\vert I\right\vert }\int_{I}d\nu +2\left\vert
I\right\vert \inf_{x,y\in I}\frac{H\left( \chi _{I^{c}}\nu \right) \left(
y\right) -H\left( \chi _{I^{c}}\nu \right) \left( x\right) }{y-x}.
\end{eqnarray*}

Now we return to the dual pivotal condition \eqref{dual pivotal condition},
and let $C_{\omega ,\sigma ,p}$ be the best constant in the dual testing
condition \eqref{Hilbert test} for $H$. Let $Q_{0}=\bigcup_{r=1}^{\infty
}Q_{r}$ be a pairwise disjoint decomposition of $Q_{0}$ and consider $%
\varepsilon ,\delta >0$ which will be chosen at the end of the proof (we
will take $\delta =\frac{1}{2}$ and $\varepsilon >0$ very small). For each
interval $Q_{r}$ let $\alpha _{r}\in Q_{r}$ minimize $\left\vert H\left(
\chi _{Q_{r}^{c}}\omega \right) \right\vert $ on $Q_{r}$, i.e.%
\begin{equation*}
\left\vert H\left( \chi _{Q_{r}^{c}}\omega \right) \left( \alpha _{r}\right)
\right\vert =\min_{x\in I}\left\vert H\left( \chi _{Q_{r}^{c}}\omega \right)
\left( x\right) \right\vert ,
\end{equation*}%
and set%
\begin{equation*}
J_{r,\varepsilon }\equiv \left( \alpha _{r}-\varepsilon \left\vert
Q_{r}\right\vert ,\alpha _{r}+\varepsilon \left\vert Q_{r}\right\vert
\right) \bigcap Q_{r}.
\end{equation*}%
Now for each interval $Q_{r}$, consider the following three mutually
exclusive and exhaustive cases:

\begin{description}
\item[Case \#1] $\frac{1}{\left\vert Q_{r}\right\vert }\int_{Q_{r}}d\omega >%
\frac{\left\vert Q_{r}\right\vert }{4}\int_{\mathbb{R}\setminus Q_{r}}\frac{1%
}{\left\vert z-z_{Q_{r}}\right\vert ^{2}}d\omega \left( z\right) $,

\item[Case \#2] $\frac{1}{\left\vert Q_{r}\right\vert }\int_{Q_{r}}d\omega
\leq \frac{\left\vert Q_{r}\right\vert }{4}\int_{\mathbb{R}\setminus Q_{r}}%
\frac{1}{\left\vert z-z_{Q_{r}}\right\vert ^{2}}d\omega \left( z\right) $
and $\left\vert Q_{r}\setminus J_{r,\varepsilon }\right\vert _{\sigma }\geq
\delta \left\vert Q_{r}\right\vert _{\sigma }$,

\item[Case \#3] $\frac{1}{\left\vert Q_{r}\right\vert }\int_{Q_{r}}d\omega
\leq \frac{\left\vert Q_{r}\right\vert }{4}\int_{\mathbb{R}\setminus Q_{r}}%
\frac{1}{\left\vert z-z_{Q_{r}}\right\vert ^{2}}d\omega \left( z\right) $
and $\left\vert J_{r,\varepsilon }\right\vert _{\sigma }>\left( 1-\delta
\right) \left\vert Q_{r}\right\vert _{\sigma }$.
\end{description}

If $Q_{r}$ is a Case \#1 interval we have $\mathbb{P}\left( Q_{r},\chi
_{Q_{0}}\omega \right) \leq 3\frac{1}{\left\vert Q_{r}\right\vert }%
\int_{Q_{r}}d\omega $ and so%
\begin{eqnarray*}
\sum_{Q_{r}\text{ satisfies Case \#1}}\left\vert Q_{r}\right\vert _{\sigma }%
\mathbb{P}\left( Q_{r},\chi _{Q_{0}}\omega \right) ^{p^{\prime }} &\leq
&3^{p}\sum_{r=1}^{\infty }\left\vert Q_{r}\right\vert _{\sigma }\left( \frac{%
1}{\left\vert Q_{r}\right\vert }\int_{Q_{r}}d\omega \right) ^{p^{\prime }} \\
&\leq &C_{p}\sum_{r=1}^{\infty }\frac{\left\vert Q_{r}\right\vert _{\sigma
}\left\vert Q_{r}\right\vert _{\omega }^{p^{\prime }-1}}{\left\vert
Q_{r}\right\vert ^{p^{\prime }}}\int_{Q_{r}}d\omega \\
&\leq &C_{p}\left\Vert \left( \omega ,\sigma \right) \right\Vert
_{A_{p}}^{p^{\prime }}\int_{Q_{0}}d\omega .
\end{eqnarray*}%
If $Q_{r}$ is a Case \#2 or Case \#3 interval we have from \eqref{neccinequ}
with $\nu =\chi _{Q_{0}}\omega $ that for all $x\in Q_{r}\setminus
J_{r,\varepsilon }$, 
\begin{eqnarray*}
\mathbb{P}\left( Q_{r};\chi _{Q_{0}}\omega \right) &\leq &6\left\vert
Q_{r}\right\vert \frac{H\left( \chi _{Q_{0}\cap Q_{r}^{c}}\omega \right)
\left( x\right) -H\left( \chi _{Q_{0}\cap Q_{r}^{c}}\omega \right) \left(
\alpha _{r}\right) }{x-\alpha _{r}} \\
&\leq &6\left\vert Q_{r}\right\vert \frac{1}{\varepsilon \left\vert
Q_{r}\right\vert }\left\{ \left\vert H\left( \chi _{Q_{0}\cap
Q_{r}^{c}}\omega \right) \left( x\right) \right\vert +\left\vert H\left(
\chi _{Q_{0}\cap Q_{r}^{c}}\omega \right) \left( \alpha _{r}\right)
\right\vert \right\} \\
&\leq &\frac{12}{\varepsilon }\left\vert H\left( \chi _{Q_{0}\cap
Q_{r}^{c}}\omega \right) \left( x\right) \right\vert .
\end{eqnarray*}%
If now $Q_{r}$ is a Case \#2 interval we also have $\left\vert
Q_{r}\right\vert _{\sigma }\leq \frac{1}{\delta }\left\vert Q_{r}\setminus
J_{r,\varepsilon }\right\vert _{\sigma }$ and so%
\begin{eqnarray}
&&\sum_{Q_{r}\text{ satisfies Case \#2}}\left\vert Q_{r}\right\vert _{\sigma
}\mathbb{P}\left( Q_{r},\chi _{Q_{0}}\omega \right) ^{p^{\prime }}
\label{Case 2} \\
&\leq &\frac{1}{\delta }\sum_{Q_{r}\text{ satisfies Case \#2}}\left\vert
Q_{r}\setminus J_{r,\varepsilon }\right\vert _{\sigma }\mathbb{P}\left(
Q_{r},\chi _{Q_{0}}\omega \right) ^{p^{\prime }}  \notag \\
&\leq &\frac{1}{\delta }\sum_{r=1}^{\infty }\left( \frac{12}{\varepsilon }%
\right) ^{p^{\prime }}\int_{Q_{r}\setminus J_{r,\varepsilon }}\left\vert
H\left( \chi _{Q_{0}\cap Q_{r}^{c}}\omega \right) \left( x\right)
\right\vert ^{p^{\prime }}d\sigma \left( x\right)  \notag \\
&\leq &C_{\varepsilon ,\delta ,p}\sum_{r=1}^{\infty }\int_{Q_{r}\setminus
J_{r,\varepsilon }}\left\{ \left\vert H\left( \chi _{Q_{0}}\omega \right)
\left( x\right) \right\vert ^{p^{\prime }}+\left\vert H\left( \chi
_{Q_{r}}\omega \right) \left( x\right) \right\vert ^{p^{\prime }}\right\}
d\sigma \left( x\right)  \notag \\
&\leq &C_{\varepsilon ,\delta ,p}\left\{ \int_{Q_{0}}\left\vert H\left( \chi
_{Q_{0}}\omega \right) \left( x\right) \right\vert ^{p^{\prime }}d\sigma
\left( x\right) +\sum_{r=1}^{\infty }\int_{Q_{r}}\left\vert H\left( \chi
_{Q_{r}}\omega \right) \left( x\right) \right\vert ^{p^{\prime }}d\sigma
\left( x\right) \right\}  \notag \\
&\leq &C_{\varepsilon ,\delta ,p}\left\{ C\left\vert Q_{0}\right\vert
_{\omega }+\sum_{r=1}^{\infty }C\left\vert Q_{r}\right\vert _{\omega
}\right\} =C_{\varepsilon ,\delta ,p}\left\vert Q_{0}\right\vert _{\omega },
\notag
\end{eqnarray}%
where the final inequality follows from \eqref{Hilbert test} with $I=Q_{0}$
and then $I=Q_{r}$.

Now we use our assumption that $\sigma $\ is doubling. There are $C,\eta >0$
such that%
\begin{equation*}
\left\vert J\right\vert _{\sigma }\leq C\left( \frac{\left\vert J\right\vert 
}{\left\vert Q\right\vert }\right) ^{\eta }\left\vert Q\right\vert _{\sigma }
\end{equation*}%
whenever $J$ is a subinterval of an interval $Q$. If $Q_{r}$ is a Case \#3
interval we have both%
\begin{equation*}
\frac{\left\vert J_{r,\varepsilon }\right\vert }{\left\vert Q_{r}\right\vert 
}\leq 2\varepsilon \text{ and }\left\vert J_{r,\varepsilon }\right\vert
_{\sigma }>\left( 1-\delta \right) \left\vert Q_{r}\right\vert _{\sigma },
\end{equation*}%
which altogether yields%
\begin{equation*}
\left( 1-\delta \right) \left\vert Q_{r}\right\vert _{\sigma }<\left\vert
J_{r,\varepsilon }\right\vert _{\sigma }\leq C\left( \frac{\left\vert
J_{r,\varepsilon }\right\vert }{\left\vert Q_{r}\right\vert }\right) ^{\eta
}\left\vert Q_{r}\right\vert _{\sigma }\leq C\left( 2\varepsilon \right)
^{\eta }\left\vert Q_{r}\right\vert _{\sigma },
\end{equation*}%
which is a contradiction if $\delta =\frac{1}{2}$ and $\varepsilon >0$ is
chosen sufficiently small, $\varepsilon <\frac{1}{2}\left( \frac{1}{2C}%
\right) ^{\frac{1}{\eta }}$. With this choice, there are no Case \#3
intervals, and so we are done.

\begin{lemma}
\label{doub ineq}Suppose that $\sigma $ is doubling and $T=H$ is the Hilbert
transform. Then the dual Poisson inequality \eqref{poissonweightedineq'} is
implied by the $A_{p}$ condition \eqref{Ap} and the dual Hilbert transform
inequality \eqref{Hilbert partial}.
\end{lemma}

\textbf{Proof}: The proof is virtually identical to that of Lemma \ref{doub
piv} but with $d\nu =\chi _{Q_{0}}gd\omega $ in place of $\chi
_{Q_{0}}d\omega $ where $g\geq 0$. Indeed, if $Q_{r}$ is a Case \#1 interval
we then have $\mathbb{P}\left( Q_{r},\chi _{Q_{0}}g\omega \right) \leq 3%
\frac{1}{\left\vert Q_{r}\right\vert }\int_{Q_{r}}gd\omega $ and so%
\begin{eqnarray*}
\sum_{Q_{r}\text{ satisfies Case \#1}}\left\vert Q_{r}\right\vert _{\sigma }%
\mathbb{P}\left( Q_{r},\chi _{Q_{0}}g\omega \right) ^{p^{\prime }} &\leq
&3^{p}\sum_{r=1}^{\infty }\left\vert Q_{r}\right\vert _{\sigma }\left( \frac{%
1}{\left\vert Q_{r}\right\vert }\int_{Q_{r}}gd\omega \right) ^{p^{\prime }}
\\
&\leq &C_{p}\sum_{r=1}^{\infty }\frac{\left\vert Q_{r}\right\vert _{\sigma
}\left\vert Q_{r}\right\vert _{\omega }^{p^{\prime }-1}}{\left\vert
Q_{r}\right\vert ^{p^{\prime }}}\int_{Q_{r}}g^{p^{\prime }}d\omega \\
&\leq &C_{p}\left\Vert \left( \omega ,\sigma \right) \right\Vert
_{A_{p}}^{p^{\prime }}\int_{Q_{0}}g^{p^{\prime }}d\omega .
\end{eqnarray*}%
If $Q_{r}$ is a Case \#2 interval, then $\left\vert Q_{r}\right\vert
_{\sigma }\leq \frac{1}{\delta }\left\vert Q_{r}\setminus J_{r,\varepsilon
}\right\vert _{\sigma }$ and%
\begin{eqnarray*}
&&\sum_{Q_{r}\text{ satisfies Case \#2}}\left\vert Q_{r}\right\vert _{\sigma
}\mathbb{P}\left( Q_{r},\chi _{Q_{0}}g\omega \right) ^{p^{\prime }} \\
&\leq &\frac{1}{\delta }\sum_{Q_{r}\text{ satisfies Case \#2}}\left\vert
Q_{r}\setminus J_{r,\varepsilon }\right\vert _{\sigma }\mathbb{P}\left(
Q_{r},\chi _{Q_{0}}g\omega \right) ^{p^{\prime }} \\
&\leq &\frac{1}{\delta }\sum_{r=1}^{\infty }\left( \frac{12}{\varepsilon }%
\right) ^{p^{\prime }}\int_{Q_{r}\setminus J_{r,\varepsilon }}\left\vert
H\left( \chi _{Q_{0}\cap Q_{r}^{c}}g\omega \right) \left( x\right)
\right\vert ^{p^{\prime }}d\sigma \left( x\right) \\
&\leq &C_{\varepsilon ,\delta ,p}\sum_{r=1}^{\infty }\int_{Q_{r}\setminus
J_{r,\varepsilon }}\left\{ \left\vert H\left( \chi _{Q_{0}}g\omega \right)
\left( x\right) \right\vert ^{p^{\prime }}+\left\vert H\left( \chi
_{Q_{r}}g\omega \right) \left( x\right) \right\vert ^{p^{\prime }}\right\}
d\sigma \left( x\right) \\
&\leq &C_{\varepsilon ,\delta ,p}\left\{ \int_{Q_{0}}\left\vert H\left( \chi
_{Q_{0}}g\omega \right) \left( x\right) \right\vert ^{p^{\prime }}d\sigma
\left( x\right) +\sum_{r=1}^{\infty }\int_{Q_{r}}\left\vert H\left( \chi
_{Q_{r}}g\omega \right) \left( x\right) \right\vert ^{p^{\prime }}d\sigma
\left( x\right) \right\} \\
&\leq &C_{\varepsilon ,\delta ,p}\left\{ C\int_{Q_{0}}g^{p^{\prime }}d\omega
+\sum_{r=1}^{\infty }C\int_{Q_{r}}g^{p^{\prime }}d\omega \right\}
=C_{\varepsilon ,\delta ,p}\int_{Q_{0}}g^{p^{\prime }}d\omega ,
\end{eqnarray*}%
upon using \eqref{Hilbert partial} with $Q_{0}$ and $Q_{r}$, which is (\ref%
{poissonweightedineq'}). As before, Case \#3 intervals don't exist if $%
\sigma $ is doubling and $\varepsilon >0$ is sufficiently small.

\subsection{Proof of Theorem \protect\ref{improved}}

Theorem \ref{poisson sufficiency} shows that the dual Poisson inequality (%
\ref{poissonweightedineq'})\ holds uniformly in $Q_{0}$ and pairwise
disjoint $\left\{ Q_{r}\right\} _{r=1}^{\infty }$ satisfying $%
\bigcup_{r=1}^{\infty }Q_{r}\subset Q_{0}$, provided both the
half-strengthened $A_{p}$ condition \eqref{Ap half skirt} and the dual
pivotal condition \eqref{dual pivotal condition} hold when $1<p\leq 2$ - and
provided \eqref{Ap half skirt}, \eqref{dual pivotal condition} and the
Poisson condition \eqref{Poisson condition} hold when $p>2$. Since $\sigma $
is doubling, Lemma \ref{doub piv} shows that the dual pivotal condition (\ref%
{dual pivotal condition}) follows from the dual testing condition 
\eqref{dual
testing condition Q} - and Lemma \ref{doub ineq} shows that the dual Poisson
inequality \eqref{poissonweightedineq'}, hence also the Poisson condition %
\eqref{Poisson condition}, follows from the dual Hilbert transform
inequality \eqref{Hilbert partial}. Thus Theorem \ref{improved} now follows
from the claim proved in Subsubsection \ref{suff P} that (\ref%
{poissonweightedineq'}) can be substituted for \eqref{M2weightdual} in the
proof of Theorem \ref{twoweightHaar}.


\begin{thebibliography}{99}
\bibitem{CoSa} \textsc{M. Cotlar and C. Sadosky,} \textit{On the Helson-Szeg%
\"{o} theorem and a related class of modified Toeplitz kernels,} Harmonic
analysis in Euclidean spaces, Part 1 (Proc. Sympos. Pure Math. XXXV,
Williams Coll., Williamstown, Mass., 1978), MR\{545279 (81j:42022)\}.

\bibitem{CoSa2} \textsc{M. Cotlar and C. Sadosky,} \textit{On some }$L^{p}$%
\textit{\ versions of the Helson-Szeg\"{o} theorem,} Wadsworth Math.
Ser.(1983), 306--317, MR\{730075 (85i:42015)\}.

\bibitem{CrMaPe} \textsc{D. Cruz-Uribe, J. M. Martell and C. P\'{e}rez,} 
\textit{Sharp two weight inequalities for singular integrals, with
applications to the Hilbert transform and the Sarason conjecture, }Adv.
Math. \textbf{216} (2007), 647--676, MR\{2351373\}.

\bibitem{LaSaUr} \textsc{Lacey, Michael T., Sawyer, Eric T., Uriarte-Tuero,
Ignacio,} \textit{A two weight inequality for the Hilbert transform assuming
an energy hypothesis,} http://arxiv.org/abs/1001.4043v6 (2010).

\bibitem{MaMaNiOr} \textsc{J. Mateu, P. Matilla, A. Nicolau and J. Orobitg,} 
\textit{BMO for nondoubling measures,} Duke Math. J. \textbf{102} (2000),
533-565, MR\{1756109 (2001e:26019)\}.

\bibitem{Mu} \textsc{B. Muckenhuopt,} \textit{Weighted norm inequalities for
the Hardy maximal function,} Trans. A.M.S. \textbf{165} (1972), 207-226.

\bibitem{MR1887641} \textsc{Muscalu, Camil and Tao, Terence and Thiele,
Christoph}, \textit{Multi-linear operators given by singular multipliers}, {%
J. Amer. Math. Soc.}, \textbf{15} {(2002)}, {469--496}.

\bibitem{NTV1} \textsc{F. Nazarov, S. Treil and A. Volberg,} \textit{The
Bellman function and two weight inequalities for Haar multipliers}, J. Amer.
Math. Soc. \textbf{12} (1999), 909-928, MR\{1685781 (2000k:42009)\}.

\bibitem{NTV2} \textsc{F. Nazarov, S. Treil and A. Volberg,} \textit{Two
weight inequalities for individual Haar multipliers and other well localized
operators}, preprint.

\bibitem{NTV3} \textsc{F. Nazarov, S. Treil and A. Volberg,} \textit{Two
weight estimate for the Hilbert transform and corona decomposition for
non-doubling measures}, preprint. http://arxiv.org/abs/1003.1596.

\bibitem{NTV4} \textsc{F. Nazarov, S. Treil and A. Volberg,} \textit{The {$%
Tb $}-theorem on non-homogeneous spaces}, {Acta Math.}, {Acta Mathematica}, 
\textbf{190}, {2003}, {151--239}, {MR1998349 (2005d:30053)},

\bibitem{NTV5} \textsc{F. Nazarov, S. Treil and A. Volberg,} \textit{Cauchy
integral and {C}alder\'{o}n-{Z}ygmund operators on nonhomogeneous spaces}, {%
Internat. Math. Res. Notices}, {1997}, \textbf{15}, {703--726}, {MR1470373
(99e:42028)},

\bibitem{PeVoYu} \textsc{F. Peherstorfer, A. Volberg and P.Yuditskii,} 
\textit{Two weight Hilbert transform and Lipschitz property of Jacobi
matrices associated to hyperbolic polynomials,} J. Funct. Anal. \textbf{246}
(2007), 1--30, MR\{2316875\}.

\bibitem{Pet} \textsc{St. Petermichl,} \textit{Dyadic shift and a
logarithmic estimate for Hankel operators with matrix symbol}, C. R. Acad.
Sci. Paris \textbf{330} (2000), 455-460, MR\{1756958 (2000m:42016)\}.

\bibitem{PeVo} \textsc{St. Petermichl, S. Treil, and A. Volberg,} \textit{%
Why the Riesz transforms are averages of the dyadic shifts?}, {Publ. Mat.},
(2002) \textbf{Vol. Extra}, pages={209--228}, MR\{1964822 (2003m:42028)\}.

\bibitem{Ru} \textsc{W. Rudin,} \textit{Real and Complex Analysis,} McGraw
Hill, 1996.

\bibitem{Saw1} \textsc{E. Sawyer,} \textit{A characterization of a two
weight norm inequality for maximal operators}, Studia Math. \textbf{75}
(1982), 1-11, MR\{676801 (84i:42032)\}.

\bibitem{Saw} \textsc{E. Sawyer,} \textit{A two weight weak-type inequality
for fractional integrals}, Trans. A.M.S. \textbf{281} (1984), 339-345,
MR\{719674 (85j:26010)\}.

\bibitem{Saw2} \textsc{E. Sawyer,} \textit{A characterization of two weight
norm inequalities for fractional and Poisson integrals}, Trans. A.M.S. 
\textbf{308} (1988), 533-545, MR\{930072 (89d:26009)\}.

\bibitem{SaWh} \textsc{E. Sawyer and R. L. Wheeden,} Weighted inequalities
for fractional integrals on Euclidean and homogeneous spaces, \textit{Amer.
J. Math. }\textbf{114} (1992), 813-874.

\bibitem{St} \textsc{E. M. Stein,} \textit{Harmonic Analysis: real-variable
methods, orthogonality, and oscillatory integrals},\textit{\ }Princeton
University Press, Princeton, N. J., 1993.

\bibitem{St2} \textsc{E. M. Stein and R. Shakarchi,} \textit{Measure theory,
integration, and Hilbert spaces,} Princeton Lectures in Analysis, III
(2005), Princeton University Press, MR\{2129625 (2005k:28024)\}.

\bibitem{Tol} \textsc{X. Tolsa,} $L^{2}$\textit{-boundedness of the Cauchy
integral operator for continuous measures,} Duke Math. J. \textbf{98}
(1999), 269-304.

\bibitem{Vo} \textsc{A. Volberg,} \textit{Calder\'{o}n-Zygmund capacities
and operators on nonhomogeneous spaces,} CBMS Regional Conference Series in
Mathematics (2003), MR\{2019058 (2005c:42015)\}.

\bibitem{Zh} \textsc{D. Zheng,} \textit{The distribution function inequality
and products of Toeplitz operators and Hankel operators,} J. Funct. Anal. 
\textbf{138} (1996), 477--501, MR\{1395967 (97e:47040)\}.
\end{thebibliography}
\end{document}